\documentclass[12pt]{amsart}

\usepackage{amssymb,amscd,amsthm, verbatim,amsmath,color,fancyhdr, mathrsfs}
\usepackage{graphicx}
\usepackage{turnstile}
\usepackage[disable]{todonotes}
\usepackage[bookmarks=true]{hyperref}
\usepackage{csquotes}

\usepackage[letterpaper, left=2.5cm, right=2.5cm, top=2.5cm,
bottom=2.5cm,dvips]{geometry}

\newtheorem{thm}{Theorem}[section]
\newtheorem{prop}[thm]{Proposition}
\newtheorem{lem}[thm]{Lemma}
\newtheorem{rem}[thm]{Remark}
\newtheorem{cor}[thm]{Corollary}

\newtheorem{ex}[thm]{Example}
\newtheorem{defn}[thm]{Definition}
\newtheorem{fact}[thm]{Fact}
\newtheorem{notation}[thm]{Notation}
\newtheorem{conv}[thm]{Convention}
\newtheorem{conj}[thm]{Conjecture}

\title{Curve neighborhoods and minimal degrees in quantum products}

\author{Christoph B\"arligea}
\address{Heinrich-Heine-Universit\"at D\"usseldorf\\ Mathematisches Institut\\ Lehrstuhl f\"ur Algebra und Zahlentheorie\\ Universit\"atsstra{\ss}e 1\\ 40225 D\"usseldorf}
\email{baerligea@math.uni-duesseldorf.de}
\keywords{Gromov-Witten theory, quantum cohomology, homogeneous spaces, curve neighborhoods}
\subjclass[2010]{Primary 14N35; Secondary 14N15, 14M15}

\date{December 13, 2016}

\begin{document}

\begin{abstract}

Let $G$ be a connected, simply connected, simple, complex, linear algebraic group. Let $P$ be a arbitrary parabolic subgroup of $G$. Let $X=G/P$ be the $G$-homogeneous projective space attached to this situation. We consider the (small) quantum cohomology ring $(QH^*(X),\star)$ attached to $X$. We prove that there exists a unique degree $d$ which is minimal with the property that $q^d$ occurs with non-zero coefficient in the quantum product of two point classes. We denote this \emph{minimal degree in} $\mathrm{pt}\star\mathrm{pt}$ by $d_X$. We give an explicit formula to compute $d_X$ in terms of the cascade of orthogonal roots. We construct an explicit curve of degree $d_X$ passing through two general points in $X$. Moreover, we prove that $d_X$ is the unique maximal element of the set of all minimal degrees in some quantum product of two Schubert classes.

\end{abstract}

\maketitle

\section{Introduction}

Let $G$ be a connected, simply connected, simple, complex, linear algebraic group. Let $P$ be a fixed but arbitrary parabolic subgroup of $G$. Let $X=G/P$ be the $G$-homogeneous projective space attached to this situation. We select once and for all a maximal torus $T$ and a Borel subgroup $B$ of $G$ such that
$$
T\subseteq B\subseteq P\subseteq G\,.
$$

\begin{conv}

From now on, if we speak about a parabolic subgroup, we always mean what is usually called a standard parabolic subgroup (relative to the fixed $B$), i.e. a parabolic subgroup of $G$ containing $B$. In other words, by convention, all parabolic subgroups are standard.

\end{conv}

\begin{rem}

{\color{black} In the next few subsections (Subsection~\ref{subsec:roots}-\ref{subsec:gwinvariants}), we set up notation and summarize well-known terminology concerning the theory of algebraic groups. The reader can skip these subsections and use them as a dictionary to trace back the notation whenever needed. We encourage the reader to go directly to Subsection~\ref{subsec:summary} where we summarize the results of this paper.}


\end{rem}

\subsection{Root system and Weyl group}
\label{subsec:roots}

Let $R$ be the root system associated to $G$ and $T$. Let $R^+$ be the positive roots of $R$ associated to $B$. Let $\Delta$ be the set of simple roots associated to $R^+$. Let
$$
W=N_G(T)/T\text{ and }W_P=N_P(T)/T
$$
be the Weyl group of $G$ and $P$ respectively. The parabolic subgroup $P$ uniquely determines and is determined by its set of simple roots $\Delta_P=\{\beta\in\Delta\mid s_\beta\in W_P\}$. The group $W_P$ is a parabolic subgroup of $W$ in the sense that it is generated by the simple reflections $s_\beta$ for $\beta\in\Delta_P$. By the Bruhat decomposition (\cite[28.3, Theorem]{humphreys}), we have $P=BW_PB$. We set $R_P=R\cap\mathbb{Z}\Delta_P$ and $R_P^+=R_P\cap R^+$. The positive roots $R^+$ clearly induce a partial order \enquote{$\leq$} on $R$. In turn, this partial order induces via restriction a partial order on $R_P$ which is still denotes by \enquote{$\leq$} and coincides with the partial order induced by $R_P^+$.

\begin{notation}
\label{not:pbeta}

The parabolic subgroups of $G$ (relative to $B$) correspond one to one to subsets of $\Delta$ (cf. \cite[30.1]{humphreys}). For each $\beta\in\Delta$, we denote the maximal parabolic subgroup corresponding to the set $\Delta\setminus\{\beta\}$ by $P_\beta$. Consequently, the set of all maximal parabolic subgroups containing $P$ is given by $\{P_\beta\mid\beta\in\Delta\setminus\Delta_P\}$.

\end{notation}

Throughout the discussion, we fix a $W$-invariant scalar product $(-,-)$ on $\mathbb{R}\Delta$. This scalar product is unique up to non-zero scalar. Each root $\alpha\in R$ has a coroot $\alpha^\vee$ which is defined by $\alpha^\vee=\frac{2\alpha}{(\alpha,\alpha)}$. All coroots together form the dual root system $R^\vee=\{\alpha^\vee\mid\alpha\in R\}$. The set of simple coroots of $R^\vee$ is given $\Delta^\vee=\{\beta^\vee\mid\beta\in\Delta\}$. For each $\beta\in\Delta$ we denote by $\omega_\beta\in\mathbb{R}\Delta$ the corresponding fundamental weight. It is defined by the equation $(\omega_\beta,\beta')=\delta_{\beta,\beta'}$ for all $\beta'\in\Delta$.

On the Weyl group we have a natural length function. For $w\in W$, the length of $w$, denoted by $\ell(w)$, is defined to be the number of simple reflections in a reduced expression of $W$. It is well-known that this number does not depend on the choice of the reduced expression. Each coset $wW_P\in W/W_P$ has a unique minimal and maximal representative, i.e. contains a unique element of minimal and maximal length. We denote by $W^P$ the set of all minimal representatives of cosets in $W/W_P$. The length function carries over from $W$ to $W/W_P$. The length of a coset $wW_P\in W/W_P$, denoted by $\ell(wW_P)$, is defined to be the length of the minimal representative in $wW_P$.

\begin{notation}

We denote by $w_o$ the longest element of $W$, i.e. the unique element of $W$ with maximal length. Similarly, we denote by $w_P$ the longest element of $W_P$, i.e. the unique element of $W_P$ with maximal length. Note that $w_o$ and $w_P$ are both involutions. We denote by $w_X$ the minimal representative in $w_oW_P$. We have the relation $w_ow_P=w_X$ or equivalent $w_o=w_Xw_P$. For each $w\in W$, we write $w^*=w_ow$ for short. 

\end{notation}

\begin{notation}

By assumption, $R$ is an irreducible root system. Therefore there exists a highest root in $R$, i.e. a unique maximal element with respect to the partial order \enquote{$\leq$}. We always denote the highest root in $R$ by $\theta_1$.

\end{notation}

\begin{notation}

For a positive root $\alpha\in R^+$, we denote by $\Delta(\alpha)$ the support of $\alpha$, i.e. the set of simple roots $\beta\in\Delta$ such that $\beta\leq\alpha$.

\end{notation}

\subsection{Cohomology}

All homology and cohomology groups in this paper are taken with integral coefficients. By convention, we write
$$
H_*(X)=H_*(X,\mathbb{Z})\text{ and }H^*(X)=H^*(X,\mathbb{Z})\,.
$$
For a closed irreducible subvariety $Z\subseteq X$, we denote by $[Z]\in H^{2\mathrm{codim}(Z)}(X)$ the cohomology class of $Z$. By abuse of notation, we also denote with the same symbol $[Z]\in H_{2\mathrm{dim}(Z)}(X)$ the homology class of $Z$. Both definitions are Poincar\'e dual to each other. For a cohomology class $\sigma\in H^*(X)$, we denote by $\sigma^*\in H^*(X)$ the cohomology class which is dual to $\sigma$ with respect to the intersection pairing.

\subsection{Schubert varieties}
\label{subsec:schubert}

Let $B^-=w_oBw_o$ be the Borel subgroup of $G$ opposite to $B$. Let $w\in W$. We denote by $\Omega_w=BwP/P$ the Schubert cell associated to $w$. We denote by $X_w=\overline{\Omega}_w$ the Schubert variety associated to $w$. We denote by $Y_w=\overline{B^-wP/P}$ the opposite Schubert variety associated to $w$. Note that $\Omega_w$, $X_w$ and $Y_w$ depend only on $wW_P$. 
We have the following equality for the dimension and codimension of Schubert and opposite Schubert varieties:
\begin{equation}
\label{eq:dimschubert}
\mathrm{dim}(X_w)=\mathrm{codim}(Y_w)=\ell(wW_P)\,.
\end{equation}
Using $X_w$ and $Y_w$ we can define Schubert cycles
$$
\sigma(w)=[X_w]\in H_{2\ell(wW_P)}(X)
\text{ and }\sigma_w=[Y_w]\in H^{2\ell(wW_P)}(X)\,.
$$
From the Bruhat decomposition of
$
X=\coprod_{w\in W^P}\Omega_w
$
it follows easily that the cohomology of $X$ decomposes as direct sums
\begin{equation}
\label{eq:cohdecomp}
H^*(X)=\bigoplus_{w\in W^P}\mathbb{Z}\sigma(w)=\bigoplus_{w\in W^P}\mathbb{Z}\sigma_w
\end{equation}
Poincar\'e duality transforms one basis of Schubert cycles into the other basis of Schubert cycles and vice versa. Indeed, we have
$$
\sigma(w)^*=\sigma(w^*)=\sigma_w\text{ and }\sigma_w^*=\sigma_{w^*}=\sigma(w)\text{ for all }w\in W\,.
$$
Schubert varieties can be used to define a partial order \enquote{$\preceq$} on $W/W_P$ -- the so-called Bruhat order. For $u,v\in W$, we set
$$
uW_P\preceq vW_P\Longleftrightarrow X_u\subseteq X_v\Longleftrightarrow Y_v\subseteq Y_u\,.
$$
This partial order is clearly well-defined since Schubert and opposite Schubert varieties are parameterized by elements of $W/W_P$. Let $Q$ be a parabolic subgroup of $G$ containing $P$. It is immediately clear from the definition that the Bruhat order behaves well under projection, in the sense that we have
$$
uW_Q\preceq vW_Q\text{ if }uW_P\preceq vW_P\text{ for all }u,v\in W\,.
$$
Besides this geometric definition of the Bruhat order, there are at least two other equivalent combinatorial definitions, one in terms of the Bruhat graph, another one in terms of subexpressions. These combinatorial definitions and their properties are discussed in detail in \cite[Chapter~5]{humphreys3}. We will use all three equivalent definitions interchangeably in this article (see also \cite[Lemma~4.3]{fulton} for a comparison lemma). 

\subsection{Degrees}
\label{subsec:degrees}

Using Equation~\eqref{eq:dimschubert} and \eqref{eq:cohdecomp}, we see that we have the following decompositions
\begin{equation}
\label{eq:degdecomp}
H_2(X)=\bigoplus_{\beta\in\Delta\setminus\Delta_P}\mathbb{Z}\sigma(s_\beta)\text{ and }H^2(X)=\bigoplus_{\beta\in\Delta\setminus\Delta_P}\mathbb{Z}\sigma_{s_\beta}
\end{equation}
In this work, we will be very much concerned with elements of $H_2(X)$ and $H^2(X)$. Therefore, it is useful to use identifications as in \cite[Section~2]{curvenbhd2}. For a simple root $\beta\in\Delta\setminus\Delta_P$, we will always identify the Schubert cycles $\sigma(s_\beta)$ with $\beta^\vee+\mathbb{Z}\Delta_P^\vee\in\mathbb{Z}\Delta^\vee/\mathbb{Z}\Delta_P^\vee$ and $\sigma_{s_\beta}$ with the fundamental weight $\omega_\beta$. Using these identification, we will simply write Equation~\eqref{eq:degdecomp} as
$$
H_2(X)=\mathbb{Z}\Delta^\vee/\mathbb{Z}\Delta_P^\vee\text{ and }H^2(X)=\mathbb{Z}\{\omega_\beta\mid\beta\in\Delta\setminus\Delta_P\}\,.
$$
Under these identifications, the Poincar\'e pairing $H^2(X)\otimes H_2(X)\to\mathbb{Z}$ simply becomes the restriction of the $W$-invariant scalar product $(-,-)$ on $\mathbb{R}\Delta$. Note that $H_2(X)$ and $H^2(X)$ are naturally endowed with a partial order \enquote{$\leq$} which is given by comparing all the coefficients of the $\mathbb{Z}$-bases pointed out in Equation~\eqref{eq:degdecomp}.

\begin{conv}

Let $\beta\in\Delta$, $e\in H_2(G/P_\beta)$ and $c_1\in H^2(G/P_\beta)$. Then we always identify $e$ with the integer $(\omega_\beta,e)$ and $c_1$ with the integer $(c_1,\beta^\vee+\mathbb{Z}\Delta_{P_\beta}^\vee)=(c_1,\beta^\vee)$. Using these identifications, we can and will write
$$
H_2(G/P_\beta)=\mathbb{Z}\text{ and }H^2(G/P_\beta)=\mathbb{Z}\,.
$$
The Poincar\'e pairing $H^2(G/P_\beta)\otimes H_2(G/P_\beta)\to\mathbb{Z}$ becomes simply multiplication in $\mathbb{Z}$ under these identifications. The partial order \enquote{$\leq$} on $H_2(G/P_\beta)$ and $H^2(G/P_\beta)$ becomes the usual total order on $\mathbb{Z}$.

\end{conv}

\begin{conv}

If we speak about a degree, without further specification, then we always mean an effective class in $H_2(X)$. We usually denote such a degree by the letter $d$, $d'$, $d''$ or similar. In the course of this work, it will be necessary to speak simultaneously about degrees in $H_2(X)$ and about classes in the second homology of other projective $G$-homogeneous spaces which might be different from $X$. To illustrate our convention in this case, let $Q$ be a parabolic subgroup of $G$. For example, we will often have $Q=B$. Then we will say \enquote{Let $e\in H_2(G/Q)$ be a degree.} to denote with $e$ an effective class in $H_2(G/Q)$. In this case, where the lattice $e$ belongs to might be different from $H_2(X)$, we always explicitly mention $H_2(G/Q)$ in our terminology. To emphasize this fact, we always use a letter different from $d$ -- like $e$ -- to denote a degree in $H_2(G/Q)$. 

\end{conv}

\begin{rem}

The reader should bear in mind that a degree $d$ is not an integer and that the partial order on $H_2(X)$ and $H^2(X)$ is not a total order unless $P$ is maximal. We are typically concerned with parabolic subgroups $P$ which are not maximal. In fact, many of our results (Theorem~\ref{thm:uniqueness} and Theorem~\ref{thm:main}) become trivial if $P$ is maximal. 

\end{rem}

\begin{notation}
\label{not:dalpha}

Let $\alpha$ be a positive root. One degree associated to $\alpha$ will be ubiquitous in our discussion. By definition, the degree $d(\alpha)$ is given by the equation
$$
d(\alpha)=\alpha^\vee+\mathbb{Z}\Delta_P^\vee\in H_2(X)\,.
$$
Note that the degree $d(\alpha)$ depends not only on $\alpha$ but also on $P$ although $P$ is not explicitly mentioned in the notation $d(\alpha)$. No confusion will arise from this sloppiness since we refer always to one and the same parabolic subgroup $P$ which is fixed throughout the discussion. We will never use the notation $d(\alpha)$ with respect to a different parabolic subgroup $Q$ but rather write the full expression $\alpha^\vee+\mathbb{Z}\Delta_Q^\vee$ if we need to refer to this degree depending on $Q$.

\end{notation}

\todo[inline,color=green]{Explain the identifications with $\mathbb{Z}$ in the case that $P_\beta$ is maximal. Describe the set of all maximal parabolic subgroups containing $P$.}

\subsection{\texorpdfstring{$T$-fixed points and $T$-invariant curves}{T-fixed points and T-invariant curves}}

Let $X^T$ denote the fixed point set of the left $T$-action on $X$. The elements of $X^T$ are called $T$-fixed points. It is well known (cf. \cite[Lemma~1 and Lemma~2]{pand}) that we have a bijection
\begin{equation}
\label{eq:tfixed}
W/W_P\cong X^T\,.
\end{equation}
For any $w\in W$, we denote by $x(w)$ the image of $wW_P$ under the Bijection~\eqref{eq:tfixed}, i.e. $x(w)$ is the $T$-fixed point given by the equation
$$
x(w)=wP/P\,.
$$
For more information on $T$-fixed points and related notions, we refer to \cite[Section~1]{pand}.

For every root $\alpha\in R^+\setminus R_P^+$, there exists a unique irreducible $T$-invariant curve $C_\alpha$ passing through the $T$-fixed points $x(1)$ and $x(s_\alpha)$ (\cite[Lemma~4.2]{fulton}).
This curve $C_\alpha$ is isomorphic to $\mathbb{P}^1$. An explicit construction of the curve $C_\alpha$ can be found in \cite[Section~3]{fulton}. By \cite[Lemma~3.4]{fulton} the degree of the curve $C_\alpha$ is given by
$$
[C_\alpha]=d(\alpha)\in H_2(X)\,.
$$
In other words, the geometric meaning of the notation $d(\alpha)$ for a positive root $\alpha$ (cf. Notation~\ref{not:dalpha}) is given by the equation
$$
d(\alpha)=\begin{cases}
[C_\alpha] & \text{if }\alpha\in R^+\setminus R_P^+\\
0 & \text{otherwise.}
\end{cases}
$$

\subsection{Gromov-Witten invariants and quantum cohomology}
\label{subsec:gwinvariants}

Let $N$ be a non-negative integer. Let $d$ be a degree. The (coarse) moduli space $\overline{M}_{0,N}(X,d)$ of $N$-pointed genus zero stable maps to $X$ parameterizes isomorphism classes
$
[C,p_1,\ldots,p_N,\mu\colon C\to X]
$
where:
\begin{itemize}

\item 

$C$ is a complex, projective, connected, reduced, (at worst) nodal curve of arithmetic genus zero.

\item

The marked points $p_i\in C$ are distinct and lie in the nonsingular locus.

\item

$\mu$ is a morphism such that $\mu_*[C]=d$.

\item

The pointed map $\mu$ has no infinitesimal automorphisms.

\end{itemize}
Basic properties of the moduli space $\overline{M}_{0,N}(X,d)$ can be found in \cite{fultonpan}. In particular, since $X$ is convex, we know that $\overline{M}_{0,N}(X,d)$ is a normal projective irreducible variety of dimension
$$
\mathrm{dim}(X)+(c_1(X),d)+N-3\,.
$$
This statement corresponds to \cite[Theorem~2(i)]{fultonpan} and \cite[Corollary~1]{pand}. By $c_1(X)$ we mean the first Chern class of the tangent bundle $T_X$. In view of computations, it is useful to have an explicit description of $c_1(X)$ in terms of the root system. Indeed, according to \cite[Lemma~3.5]{fulton}, we have
$$
c_1(X)=\sum_{\alpha\in R^+\setminus R_P^+}\alpha\in H^2(X)\,.
$$
The reader can convince himself (or consult \cite[Section~2]{curvenbhd2}) that the above expression is actually an element of $H^2(X)$ (with respect to the identification made in Subsection~\ref{subsec:degrees}). The moduli space $\overline{M}_{0,N}(X,d)$ comes equipped with $N$ evaluation maps. The $i$th evaluation map
$$
\mathrm{ev}_i\colon\overline{M}_{0,N}(X,d)\to X
$$
is defined by
$$
\mathrm{ev}_i([C,p_1,\ldots,p_N,\mu])=\mu(p_i)\,.
$$

Let $u,v,w\in W$ and let $d$ be a degree. Then we define the (three-point genus zero) Gromov-Witten invariant $\left<\sigma_u,\sigma_v,\sigma_w\right>_d$ to be the integral
$$
\int_{\overline{M}_{0,3}(X,d)}\mathrm{ev}_1^{*}(\sigma_u)\cup\mathrm{ev}_2^{*}(\sigma_v)\cup\mathrm{ev}_3^*(\sigma_w)\,.
$$
This integral equals a non-negative integer which is by definition non-zero only if
$$
\ell(uW_P)+\ell(vW_P)+\ell(wW_P)=\dim(\overline{M}_{0,3}(X,d))\,.
$$
Intuitively, the Gromov-Witten invariant $\left<\sigma_u,\sigma_v,\sigma_w\right>_d$ counts the number of rational curves of degree $d$ passing through general translates of $Y_u$, $Y_v$ and $Y_w$.

Let $QH^*(X)=H^*(X)\otimes_{\mathbb{Z}}\mathbb{Z}[q]$ where $\mathbb{Z}[q]=\mathbb{Z}[q_\beta\mid\beta\in\Delta\setminus\Delta_P]$. Let $d=\sum_{\beta\in\Delta\setminus\Delta_P}d_\beta d(\beta)$ be a degree where $d_\beta$ are non-negative integers. We denote by $q^d$ the monomial in $QH^*(X)$ defined as
$$
q^d=\prod_{\beta\in\Delta\setminus\Delta_P}q_\beta^{d_\beta}\,.
$$
The $\mathbb{Z}[q]$-module $QH^*(X)$ has a product structure $\star$ whose structure coefficients are Gromov-Witten invariants. More precisely, for $u,v\in W$, we define
$$
\sigma_u\star\sigma_v=\sum_{d\text{ a degree }}q^d\sum_{w\in W^P}\left<\sigma_u,\sigma_v,\sigma_w^*\right>_d\sigma_w\,.
$$
This product structure $\star$ makes $QH^*(X)$ into a commutative, associative, graded
$\mathbb{Z}[q]$-algebra with unit $\sigma_1=[X]$ (cf. \cite[Theorem~4]{fultonpan}). The algebra $(QH^*(X),\star)$ is called the (small) quantum cohomology ring of $X$.


\subsection{Summary of results}
\label{subsec:summary}

In this work, we will be concerned with the study of the minimal degrees in the quantum product of two Schubert cycles. By a minimal degree $d$ in $\sigma_u\star\sigma_v$ we mean a degree $d$ which is minimal with the property that $q^d$ occurs with non-zero coefficient in the expression $\sigma_u\star\sigma_v$ (cf. Definition~\ref{def:minimal}). Temporarily, we define for all $u,v\in W$ the set
$$
\delta_P(u,v)=\{d\text{ a minimal degree in }\sigma_u\star\sigma_v\}
$$
To abbreviate, we set $\delta_P(u)=\delta_P(u,w_o)$ for all $u\in W$, i.e. $\delta_P(u)$ is the set of all minimal degrees in $\sigma_u\star\mathrm{pt}$ where $\mathrm{pt}$ is the cohomology class of a point. In the course of the exposition, we will choose different but equivalent definitions of $\delta_P(u)$ and $\delta_P(u,v)$ (cf. Definition~\ref{def:deltaw} and Definition~\ref{def:deltauv}) which are more suitable to derive explicit properties of these sets. In particular, $\delta_P(u)$ will be defined in terms of curve neighborhoods which makes the techniques developed in \cite{curvenbhd2} accessible for our purposes. In the end, Theorem~\ref{thm:fulton} will prove, as a consequence of \cite{fulton}, that both perspectives amount to the same.

\begin{conj}[Buch-Mihalcea]
\label{conj:unique}

For all $u,v\in W$, the set $\delta_P(u,v)$ consists of a unique element.

\end{conj}


Although we are not able to prove Conjecture~\ref{conj:unique} in full generality, we make an attempt to prove at least a partial result for all $X=G/P$. More specifically, we prove the following theorem.

\begin{thm}[Theorem~\ref{thm:uniqueness}]
\label{thm:intro}

The set $\delta_P(w_o)$ consists of a unique element $d_X$.

\end{thm}

\begin{rem}

By definition, Theorem~\ref{thm:intro} amounts to say that we can write
$$
\mathrm{pt}\star\mathrm{pt}=\sigma\cdot q^{d_X}+\text{terms of degree strictly larger than }d_X
$$
for some homogeneous cohomology class $\sigma$.

\end{rem}

In the course of the proof of Theorem~\ref{thm:intro}, we will construct an explicit curve of degree $d_X$ passing through the points $x(1)$ and $x(w_o)$. The degree of this curve has a natural interpretation in terms of Kostant's cascade of orthogonal roots \cite{kostant}. This allows us to give an efficient way to compute $d_X$, namely by summing over all coroots $\alpha^\vee$ where $\alpha$ is an element of the cascade of orthogonal roots (cf. Corollary~\ref{cor:identity}). In particular, this gives us an expression of $d_X$ in terms of the geometry of the maximal quotients $G/P_\beta$ of $X$ where $\beta\in\Delta\setminus\Delta_P$ (cf. Notation~\ref{not:pbeta}). As part of Theorem~\ref{thm:uniqueness}, we can show that
$$
d_X=\sum_{\beta\in\Delta\setminus\Delta_P}d_{G/P_\beta}d(\beta)\,.
$$

{\color{black} Even if we cannot prove Conjecture~\ref{conj:unique} yet, it is possible to prove consequences which become trivial once Conjecture~\ref{conj:unique} is established in full generality. Part of this paper is devoted to one of such weaker statements which is subject of the following theorem.}

\begin{thm}[Theorem~\ref{thm:main}]
\label{thm:intro_main}

For all $u,v\in W$ and all $d\in\delta_P(u,v)$ the inequality $d\leq d_X$ is satisfied.

\end{thm}

\begin{rem}

In other words, Theorem~\ref{thm:intro_main} means that we have an inclusion
$$
\bigcup_{u,v\in W}\delta_P(u,v)\subseteq\{0\leq d\leq d_X\}\,.
$$
It is fairly easy to see that this inclusion might be strict -- not an equality -- unless $P$ is maximal (cf. Example~\ref{ex:inclusionstrict}): The set of all minimal degrees in some quantum product of two Schubert cycles does not form an interval. Theorem~\ref{thm:intro_main} can also be reformulated by saying that $d_X$ is the unique maximal element of the set of all minimal degrees in some quantum product of two Schubert cycles.

\end{rem}

\subsection*{Acknowledgment}

This work grew out of the author's PhD thesis \cite{thesis} which is cited every now and then in the paper. Some of the results already appeared there in a less final and less general form and were generalized from the case of a maximal parabolic to the case of an arbitrary parabolic subgroup whenever possible. First and foremost, the author wants to thank Nicolas Perrin, the former adviser of the thesis \cite{thesis} for countless mathematical discussions and his constant interest in this work. Also, the germ of some of the presented ideas, namely to compute $d_X$ in the case of a maximal parabolic $P$ (where the existence is obvious) in terms of the chain cascade associated to the unique simple root in $\Delta\setminus\Delta_P$ (cf. Theorem~\ref{thm:thesis_main}), goes back to some unpublished notes \cite{dmax}. The author is grateful that the ideas in these notes were openly shared with him and that he had the opportunity to develop them further. Finally, the author wants to thank Pierre-Emmanuel Chaput for an invitation to Nancy where he had the possibility to communicate the results.

\todo[inline,color=green]{Convention. If we speak about a degree, without specifying in which $2$-homology, we usually mean an effective class in $H_2(X)$. If we consider degrees in the homology of other spaces, we say a degree $e$ in $H_2(G/B)$, and mean an effective class in $H_2(G/B)$. We use a different letter than $d$, namely $e$, where $d$ reserved for degrees in $H_2(X)$.}

\todo[inline,color=green]{Mathematics subject classification on the title page.}

\todo[inline,color=green]{Throughout the discussion we fix one parabolic subgroup $P$. Subsection on parabolic subgroups. All parabolic subgroups are standard. Notation: $\Delta_P$, $P_\beta$, $w_o$, $\theta_1$, $d(\alpha)$ for an arbitrary positive root $\alpha$, so that $d(\alpha)=0$ is possible, $w_P$ etc. Support of a positive root $\Delta(\alpha)$. $T$-fixed points $x(w)$. Poincar\'e duality and Schubert cycles. Summary of $T$-invariant curves, denote it by $C_\alpha$, and refer to the fact that $[C_\alpha]=d(\alpha)$ (I use $T$-invariant curves for example in the proof of Fact~\ref{fact:compatibility}(\ref{item:comp_chain})).}

\todo[inline,color=green]{Summary of the results. In other words, $d_X$ is the unique maximal element of the set $\bigcup_{u,v\in W}\delta_P(u,v)$.

Trivial properties of the Bruhat order, compatibility under projection: $u\preceq v$ then $uW_P\preceq vW_P$ then $uW_Q\preceq vW_Q$ (cf. proof of Theorem~\ref{thm:verycosmall}).

Describe in the summary of results what is meant by pt, this is the class of a point aka. $\sigma_{w_o}$.

I still want to cite humphreys (e.g. in the description of parabolics), fultonpan (for GW-invariants), dmax (in the acknowledgment).}

\todo[inline,color=green]{We will from now on use interchangeably all three equivalent definitions of the Bruhat order without reference to the literature. Recall them (subexpressions, adjacency, geometric).}

\section{The Hecke product}

In this section we introduce the Hecke product and collect basic properties which we need in the sequel. We do not claim any originality and closely follow the reference \cite[Section~3]{curvenbhd2}. There, the reader finds more detailed information and complete proofs of all statements we make in this section.
For us, the Hecke product is important since it was proved in \cite[Theorem~5.1]{curvenbhd2} that it can be used to compute the Weyl group elements parameterizing curve neighborhoods of Schubert varieties. We will focus on this feature of the Hecke product more closely in Section~\ref{sec:zd}

Roughly speaking the Hecke product defines a monoid structure on the Weyl group where you keep all the braid relations but you replace the involution relation with the idempotent relation. The specialization of the Hecke algebra to $q=0$ is isomorphic to the monoid algebra $\mathbb{Z}[(W,\cdot)]$ where $(W,\cdot)$ denotes the Hecke monoid.

\begin{defn}

Let $u,v\in W$ and $\beta\in\Delta$. Then we define the Hecke product of $u$ and $s_\beta$ by
$$
u\cdot s_\beta=\begin{cases}
us_\beta & \text{if }us_\beta\succ u\\
u & \text{if }us_\beta\prec u\,.
\end{cases}
$$
Let $v=s_{\beta_1}\cdots s_{\beta_l}$ be any reduced expression for $v$. Then we define the Hecke product of $u$ and $v$ by
$$
u\cdot v=u\cdot s_{\beta_1}\cdot\ldots\cdot s_{\beta_l}\,.
$$
That the expression $u\cdot v$ is well-defined (independent of the choice of the reduced expression for $v$) is proved in detail in \cite[Section~3]{curvenbhd2}.

\end{defn}

\begin{rem}

The definition of the Hecke product has only apparently a right-sided nature. Equally well, one can multiply simple reflections from the left and expand the definition to arbitrary Weyl group elements via their reduced expressions (cf. \cite[Equation~(7)]{curvenbhd2}).

\end{rem}

\begin{rem}

Let $u,v\in W$. The Hecke product also defines a product $W\times W/W_P\to W/W_P$ given by $u\cdot vW_P=(u\cdot v)W_P$. Again, that the expression $u\cdot vW_P$ is well-defined (independent of the choice of representative in $vW_P$) is proved in detail in \cite[Section~3]{curvenbhd2}.

\end{rem}

\begin{prop}[{\cite[Proposition~3.1]{curvenbhd2}}]
\label{prop:hecke}

Let $u,v,v',w\in W$.

\begin{enumerate}

\item
\label{item:monoid}

The Hecke product defines a monoid structure on $W$.

\item
\label{item:inverse}

We have $(u\cdot v)^{-1}=v^{-1}\cdot u^{-1}$.

\item 
\label{item:preceq}
 
If $v\preceq v'$, then $u\cdot v\cdot w\preceq u\cdot v'\cdot w$. 

\item
\label{item:uvsmaller}

We have $uv\preceq u\cdot v$.

\item
\label{item:heckeustrich}

The element $u'=(u\cdot v)v^{-1}$ satisfies $u'\preceq u$ and $u'v=u'\cdot v=u\cdot v$.

\item
\label{item:maximallength}

A Weyl group element $w$ is the maximal representative in $wW_P$ if and only if the equality $w\cdot w_P=w$ holds.

\item
\label{item:minimal}

If $w$ is the minimal representative in $wW_P$, then $ww_P=w\cdot w_P$ is the maximal representative in $wW_P$.

\item
\label{item:wpodot}

We have that $w\cdot w_P$ is the maximal representative in $wW_P$ and that $(w\cdot w_P)w_P$ is the minimal representative in $wW_P$.


\item
\label{item:preceqmodp}

If $vW_P\preceq v'W_P$, then $u\cdot vW_P\preceq u\cdot v'W_P$.

\end{enumerate}

\end{prop}

\begin{proof}

Item~(\ref{item:monoid}), (\ref{item:inverse}), (\ref{item:preceq}), (\ref{item:uvsmaller}), (\ref{item:heckeustrich}) correspond respectively to \cite[Proposition~3.1(a), (b), (c), (d), (e)]{curvenbhd2}. The reader finds a complete proof there. We only included the statements for the convenience of the reader and for later reference. 

We prove (\ref{item:maximallength}). It is clear that $w$ and $w\cdot w_P$ belong to the same class modulo $W_P$. Moreover, $w\preceq w\cdot w_P$ and thus $\ell(w)\leq\ell(w\cdot w_P)$.
Suppose that $w$ is the maximal representative in $wW_P$. Then we find that the length of the elements $w$ and $w\cdot w_P$ must be equal. This leads to the equality $w\cdot w_P=w$. 

Let $w$ be a Weyl group element which satisfies $w\cdot w_P=w$. Let $w'$ be the maximal representative in $wW_P$. Let $u\in W_P$ such that $w'=wu$. Then we have $w'\preceq w\cdot u\preceq w\cdot w_P=w$ by Item~(\ref{item:preceq}), (\ref{item:uvsmaller})
and thus $\ell(w')\leq\ell(w)$. Therefore the length of the elements $w'$ and $w$ must be equal. This leads to the equality $w'=w$. Therefore $w$ is the maximal representative in $wW_P$. 

Ad Item~(\ref{item:minimal}). It is well-known that $ww_P$ is the maximal representative in $wW_P$. Moreover, $ww_P$ and $w\cdot w_P$ belong to the same class modulo $W_P$. Therefore the relation $ww_P\preceq w\cdot w_P$ (cf. Item~\ref{item:uvsmaller})) leads to the equality $ww_P=w\cdot w_P$.

Ad Item~(\ref{item:wpodot}). By Item~(\ref{item:maximallength}) it suffices to show that $(w\cdot w_P)\cdot w_P=w\cdot w_P$ in order to see that $w\cdot w_P$ is the maximal representative in $wW_P$. But this is clear since we obviously have $w_P\cdot w_P=w_P$. That $(w\cdot w_P)w_P$ is the minimal representative in $wW_P$ follows directly from this and Item~(\ref{item:minimal}).

Ad Item~(\ref{item:preceqmodp}). By Item~(\ref{item:wpodot}) the statement $vW_P\preceq v'W_P$ is equivalent to the statement $v\preceq v'\cdot w_P$. By Item~(\ref{item:preceq}) the later statement leads to $u\cdot v\preceq u\cdot v'\cdot w_P$. Reduction modulo $W_P$ immediately gives $u\cdot vW_P\preceq u\cdot v'W_P$ -- as required. 
\end{proof}

We can use the Hecke product to compute the stabilizer of Schubert varieties. Let $w$ be a Weyl group element. We denote the stabilizer of $X_w$ in $G$ by $P_w$. Since $X_w$ is $B$-stable, it is clear that $P_w$ is a parabolic subgroup of $G$. Hence, we can write $P_w=BW_{P_w}B$ for some parabolic subgroup $W_{P_w}$ of $W$.

\begin{lem}
\label{lem:stabofschubert}

Let $w\in W$. Then we have $W_{P_w}=\{u\in W\mid u\cdot wW_P=wW_P\}$ and $\Delta_{P_w}=\{\beta\in\Delta\mid s_\beta w W_P\preceq wW_P\}$.

\end{lem}

\begin{proof}

The description of $\Delta_{P_w}$ follows directly from the description of $W_{P_w}$ (cf. \cite[Equation~(8)]{curvenbhd2}). We prove the later identity. Let $u$ be an arbitrary Weyl group element. By definition of the Hecke product it is clear that we can write
$$
BuB\cdot BwB=B(u\cdot w)B\amalg((B,B)\text{-orbits parameterized by }v\text{ such that }v\prec u\cdot w)\,.
$$
If we apply the natural projection $G\to X$ to this equality, we obtain
$$
BuB\cdot\Omega_w=\Omega_{u\cdot w}\amalg(\text{Schubert cells }\Omega_v\text{ such that }vW_P\prec u\cdot wW_P)\,.
$$
From this equality we see that $u\in W_{P_w}$ if and only if $u\cdot wW_P\preceq wW_P$ if and only if $u\cdot wW_P=wW_P$ (cf. Proposition~\ref{prop:hecke}(\ref{item:preceq})) -- as claimed.
\end{proof}

\begin{rem}

Let $u,w\in W$. Since $W_{P_w}$ is a group, Lemma \ref{lem:stabofschubert} gives that $u\cdot wW_P=wW_P$ if and only if $u^{-1}\cdot wW_P=wW_P$. In particular, if we apply this equivalence to $P=B$, we find that $u\cdot w=w$ if and only if $u^{-1}\cdot w=w$. 

\end{rem}

\section{Curve neighborhoods}
\label{sec:zd}

In this section we review the theory of curve neighborhoods. All non-trivial results in this section are rightfully due to Buch-Mihalcea \cite{curvenbhd2}. We concentrate our discussion on properties of $z_d^P$ -- the minimal representative parameterizing the degree $d$ curve neighborhood $\Gamma_d(X_1)$ of a point $X_1=\{x(1)\}$ -- which we will need later on in Section~\ref{sec:delta} to investigate the distance function $\delta_P$.

\begin{defn}[{\cite[Section~4.2]{curvenbhd2}}]

Let $d$ be a degree. The maximal elements of the set $\{\alpha\in R^+\setminus R_P^+\mid d(\alpha)\leq d\}$ are called maximal roots of $d$. A sequence of roots $(\alpha_1,\ldots,\alpha_r)$ is called a greedy decomposition of $d$ if $\alpha_1$ is a maximal root of $d$ and $(\alpha_2,\ldots,\alpha_r)$ is a greedy decomposition of $d-d(\alpha_1)$. The empty sequence is the unique greedy decomposition of $0$.

\end{defn}

\begin{defn}[{\cite[Section~4.2]{curvenbhd2}}]

A root $\alpha\in R^+\setminus R_P^+$ is called $P$-cosmall if $\alpha$ is a maximal root of $d(\alpha)$. A root $\alpha\in R^+$ is called very $P$-cosmall if $\alpha$ is a maximal root of $(\omega_\beta,d(\alpha))=d(\alpha)+\mathbb{Z}\Delta_{P_\beta}^\vee$ for all $\beta\in\Delta\setminus\Delta_P$.

\end{defn}

\begin{rem}
\label{rem:pcosmall}



Besides the definition, there are several equivalent ways to describe a $P$-cosmall root. Three of them were worked out in \cite[Theorem~6.1]{curvenbhd2}. Moreover, the reader finds a handy characterization of $B$-cosmall roots in \cite[Proposition~6.8]{curvenbhd2}. For this work, it is only important to know that a root $\alpha\in R^+\setminus R_P^+$ is $P$-cosmall if and only if 
$$
\ell\left(s_\alpha W_P\right)=(c_1(X),d(\alpha))-1\,.
$$
This equivalence is proved in \cite[Theorem~6.1: $(\mathrm{a})\Leftrightarrow(\mathrm{b})$]{curvenbhd2}

\end{rem}

\begin{rem}

Let $Q$ be a parabolic subgroup of $G$ containing $P$. It is obvious from the definitions that:

\begin{itemize}

\item 

Any $Q$-cosmall root is also $P$-cosmall. In particular, every $P$-cosmall root is also $B$-cosmall and every very $P$-cosmall root is also (ordinary) $P$-cosmall.

\item

Any very $P$-cosmall root is also very $Q$-cosmall.

\end{itemize}

\end{rem}

\todo[inline,color=green]{I want to add here (after the definition of $P$-cosmall and very $P$-cosmall) two remarks, one about the equivalent definitions of $P$-cosmall as in \cite[Theorem~6.1]{curvenbhd2}, one about the trivial relations for $P\subseteq Q$. Very $P$-cosmall $\implies$ $P$-cosmall. Clarification of the notion very $P$-cosmall with reference to the relevant examples (Example~\ref{ex:highestroot} and Example~\ref{ex:verypcosmall}).}

\begin{rem}

Note that the notion \enquote{very $P$-cosmall} is only of technical nature. Sometimes it is convenient to summarize the defining properties by a single attribute. Later on, in Example~\ref{ex:verypcosmall}, we will clarify the relation of this notion to others. 

\end{rem}

\begin{ex}[{\cite[Section~4.1]{curvenbhd2}}]

All simple roots and all positive long\footnote{If $R$ is simply laced, we declare all roots to be long and none to be short. In particular, if we speak about a short root (for example the highest short root as in Example~\ref{ex:highestshortroot}), we implicitly assume that that there are two root lengths and that a short root exists.} roots are $B$-cosmall. In particular, if $R$ is simply laced, all positive roots are $B$-cosmall.

\end{ex}

\begin{ex}
\label{ex:highestroot2}

The highest root $\theta_1$ is the unique root which is (very) $P$-cosmall for every parabolic subgroup $P$ (cf. Example~\ref{ex:highestroot}).

\end{ex}

\begin{ex}
\label{ex:highestshortroot}

The highest short root $\theta_s$ is never $B$-cosmall, in particular never $P$-cosmall for any parabolic subgroup $P$. Indeed, the coroot $\theta_s^\vee$ is the highest root of $R^\vee$ and the coroot $\theta_1^\vee$ is the highest short root of $R^\vee$. Therefore we have $\theta_1^\vee<\theta_s^\vee$ and $\theta_s<\theta_1$ which shows that $\theta_s$ is never a maximal root of $\theta_s^\vee$.

\end{ex}

\todo[inline,color=green]{I want to add further examples in the sense of the pencil notes on September 11, in the sense of the beginning of \cite[Section~4.1]{curvenbhd2}. Compare the todo block after Table~\ref{table:exceptional} (I wanted to add a footnote).}

\begin{defn}[{\cite[Section~4.2]{curvenbhd2}}]
\label{def:zdp}

Let $d$ be a degree. Let $(\alpha_1,\ldots,\alpha_r)$ be a greedy decomposition of $d$. Then we define an element $z_d^P\in W^P$ by the following equation
$$
z_d^Pw_P=s_{\alpha_1}\cdot\ldots\cdot s_{\alpha_r}\cdot w_P\,.
$$
By Proposition~\ref{prop:hecke}(\ref{item:wpodot}) it is clear that $z_d^P$ is the minimal representative in $s_{\alpha_1}\cdot\ldots\cdot s_{\alpha_r}W_P$. Well-definedness questions of the element $z_d^P$ (independence of the choice of the greedy decomposition of $d$) are discussed in detail in \cite[Section~4, in particular Definition~4.6]{curvenbhd2}.

\end{defn}

Let $d$ be a degree. The importance of the element $z_d^P$ lies in the fact that it can be used to compute the Weyl group element parameterizing any degree $d$ curve neighborhood of a Schubert variety (or a opposite Schubert variety). We review this important result of Buch-Mihalcea \cite[Theorem~5.1]{curvenbhd2} now.
Let $\Omega$ be a closed subvariety of $X$. Then we define the (degree $d$) curve neighborhood $\Gamma_d(\Omega)$ of $\Omega$ by the equation
\begin{align*}
\Gamma_d(\Omega)&=\text{closure of }
\bigcup_f f(\mathbb{P}^1)
&&\text{where }f\colon\mathbb{P}^1\to X\text{ such that }f_*[\mathbb{P}^1]=d\text{ and }f(\mathbb{P}^1)\cap\Omega\neq\emptyset\\
&=\mathrm{ev}_1(\mathrm{ev}_2^{-1}(\Omega))&&\text{where }\mathrm{ev}_1,\mathrm{ev}_2\colon\overline{M}_{0,2}(X,d)\to X\,.
\end{align*}
We will mostly use this definition when $\Omega$ is a Schubert variety $X_w$ or an opposite Schubert variety $Y_w$ parameterized by some $w\in W$. In this case, it turns out that $\Gamma_d(X_w)$ is itself a Schubert variety. More precisely, we have the equality
$$
\Gamma_d(X_w)=X_{w\cdot z_d^P}\text{ for any }w\in W\text{, in particular }\Gamma_d(X_1)=X_{z_d^P}\,.
$$

\begin{prop}[{\cite[Section~4.2]{curvenbhd2}}]
\label{prop:zd}

Let $d$ be a degree. Let $(\alpha_1,\ldots,\alpha_r)$ be a greedy decomposition of $d$.

\begin{enumerate}

\item
\label{item:unique}

The greedy decomposition of $d$ is unique up to reordering. In particular, if $\alpha\in R^+\setminus R_P^+$ is $P$-cosmall, then $\alpha$ is the unique maximal root of $d(\alpha)$. All elements of a greedy decomposition of $d$ are $P$-cosmall.

\item
\label{item:commute}

For all $1\leq i,j\leq r$ we have $s_{\alpha_i}\cdot s_{\alpha_j}=s_{\alpha_j}\cdot s_{\alpha_i}$.

\item
\label{item:zB}

For all sufficiently large degrees $e\in H_2(G/B)$ such that $e+\mathbb{Z}\Delta_P^\vee=d$ we have $z_d^Pw_P=z_e^B$.

\item
\label{item:stabofz}

We have $w_P\cdot z_d^Pw_P=z_d^Pw_P\cdot w_P=z_d^Pw_P$, in particular $P\subseteq P_{z_d^P}$.

\item
\label{item:projection}

Let $Q$ be a parabolic subgroup of $G$ containing $P$. Then we have $z_d^Pw_P\preceq z_{d+\mathbb{Z}\Delta_Q^\vee}^Qw_Q$.

\item
\label{item:zzinverse}

We have $z_d^Pw_P=(z_d^Pw_P)^{-1}\cdot w_P$, in particular $z_d^P\preceq(z_d^Pw_P)^{-1}\preceq z_d^Pw_P$, in particular $z_e^B=(z_e^B)^{-1}$ for all degrees $e\in H_2(G/B)$.

\item
\label{item:greedy}

Let $1\leq i\leq r$ be an index. Then $(\alpha_1,\ldots,\hat{\alpha}_i,\ldots,\alpha_r)$ is a greedy decomposition of $d-d(\alpha_i)$.

\end{enumerate}

\end{prop}

\begin{proof}

We prove Item~(\ref{item:unique}). For any sufficiently large $e\in H_2(G/B)$ such that $e+\mathbb{Z}\Delta_P^\vee=d$ there exist positive roots $\gamma_1,\ldots,\gamma_m\in R_P^+$ such that $(\alpha_1,\ldots,\alpha_r,\gamma_1,\ldots,\gamma_m)$ is a greedy decomposition of $e$. Therefore the uniqueness up to reordering of the greedy decomposition of $d$ follows from the uniqueness up to reordering of the greedy decomposition of $e$. The uniqueness up to reordering of a degree in $H_2(G/B)$ is discussed in \cite[Section~4.1]{curvenbhd2}. The other statements of Item~(\ref{item:unique}) are obvious now.

Ad Item~(\ref{item:commute}). Similarly as in the proof of Item~(\ref{item:unique}) the claimed commutation relation follows from the corresponding commutation relation for degrees in $H_2(G/B)$. The corresponding commutation relation for degrees in $H_2(G/B)$ is discussed in \cite[Proposition 4.8(b)]{curvenbhd2}.

Item~(\ref{item:zB}) is identical with \cite[Corollary~4.12(d)]{curvenbhd2}. We copied it for the convenience of the reader.

We next prove~(\ref{item:stabofz}). The equality $z_d^Pw_P\cdot w_P=z_d^Pw_P$ follows directly from Proposition~\ref{prop:hecke}(\ref{item:maximallength}). It follows from \cite[Theorem~5.1]{curvenbhd2} that $w_P\cdot z_d^Pw_P$ and $z_d^P w_P$ belong to the same class modulo $W_P$. The relation $z_d^P w_P\preceq w_P\cdot z_d^Pw_P$ is obvious (cf. Proposition~\ref{prop:hecke}(\ref{item:preceq})); and thus $\ell(z_d^Pw_P)\leq\ell(w_P\cdot z_d^Pw_P)$. By definition $z_d^Pw_P$ is the maximal representative in $z_d^PW_P$. Hence we find that the length of both elements must be equal. This leads to the desired equality. The statement $P\subseteq P_{z_d^P}$ follows directly from Lemma~\ref{lem:stabofschubert}.

To prove~(\ref{item:projection}) we first choose a sufficiently large degree $e\in H_2(G/B)$ such that $e+\mathbb{Z}\Delta_P^\vee=d$ and such that $z_d^Pw_P=z_e^B$. Since $e+\mathbb{Z}\Delta_Q^\vee=d+\mathbb{Z}\Delta_Q^\vee$, we can further choose a sufficiently large degree $e'\geq e$ such that $e'+\mathbb{Z}\Delta_Q^\vee=d+\mathbb{Z}\Delta_Q^\vee$ and such that $z_{d+\mathbb{Z}\Delta_Q^\vee}^Qw_Q=z_{e'}^B$. The result follows since $z_e^B\preceq z_{e'}^B$ by \cite[Corollary~4.12(b)]{curvenbhd2}.

We next prove~(\ref{item:zzinverse}). We have by Item~(\ref{item:commute}) that
$$
z_d^Pw_P=w_P\cdot z_d^Pw_P=w_P\cdot s_{\alpha_1}\cdot\ldots\cdot s_{\alpha_r}\cdot w_P=w_P\cdot s_{\alpha_r}\cdot\ldots\cdot s_{\alpha_1}\cdot w_P=(z_d^Pw_P)^{-1}\cdot w_P\,.
$$
This proves the desired equation. The relation $(z_d^Pw_P)^{-1}\preceq z_d^Pw_P$ is now obvious (cf. Proposition~\ref{prop:hecke}(\ref{item:preceq})). The relation $z_d^P\preceq (z_d^Pw_P)^{-1}$ follows from Proposition~\ref{prop:hecke}(\ref{item:heckeustrich}) since $z_d^P=((z_d^Pw_P)^{-1}\cdot w_P)w_P$. The very last statement is obvious since $w_B=1$.

Ad Item~(\ref{item:greedy}). Let $d'=d-d(\alpha_i)$ for short. Let $1\leq j<i$ be a index. By assumption $\alpha_j$ is a maximal root of $d-d(\alpha_1)-\cdots-d(\alpha_{j-1})$. Since
$$
d(\alpha_j)\leq d'-d(\alpha_1)-\cdots-d(\alpha_{j-1})<d-d(\alpha_1)-\cdots-d(\alpha_{j-1})\,,
$$
it follows that $\alpha_j$ is also a maximal roots of $d'-d(\alpha_1)-\cdots-d(\alpha_{j-1})$. Therefore there exists a greedy decomposition of $d'$ which starts with the sequence of roots $\alpha_1,\ldots,\alpha_{i-1}$. It is obvious that this sequence of roots can be completed to a greedy decomposition of $d'$ by adding the roots $\alpha_{i+1},\ldots,\alpha_r$. In total, it follows that $(\alpha_1,\ldots,\hat{\alpha}_i,\ldots,\alpha_r)$ is a greedy decomposition of $d'$.
\end{proof}

\begin{thm}
\label{thm:equalwx}

Let $d$ be a degree. Suppose that for all $\beta\in\Delta\setminus\Delta_P$ there exists a degree $d'\geq d+d(\beta)$ such that $z_d^P=z_{d'}^P$. Then we have $z_d^P=w_X$. 

\end{thm}

\begin{proof}

By assumption and \cite[Corollary~4.12(b)]{curvenbhd2} it is clear that $z_d^P=z_{d+d(\beta)}^P$ for all $\beta\in\Delta\setminus\Delta_P$. By \cite[Corollary~4.12(a)]{curvenbhd2} this equation leads to the equation $z_d^Pw_P=s_\beta\cdot z_d^Pw_P=z_{d+d(\beta)}^Pw_P$. This means in particular that $s_\beta\cdot z_d^PW_P=z_d^PW_P$. Let $Q$ be the stabilizer of $\Gamma_d(X_1)$. Lemma \ref{lem:stabofschubert} then shows that $s_\beta\in W_Q$ or equivalent $\beta\in\Delta_{Q}$ for all $\beta\in\Delta\setminus\Delta_P$ and thus $\Delta\setminus\Delta_P\subseteq\Delta_Q$. On the other hand, by Proposition~\ref{prop:zd}(\ref{item:stabofz}) the stabilizer of any curve neighborhood contains $P$. Therefore we find $\Delta_Q=\Delta$ or equivalent $Q=G$. This implies that $\Gamma_d(X_1)=X$ or equivalent $z_d^P=w_X$.
\end{proof}

\begin{thm}
\label{thm:verycosmall}

Let $\alpha$ be a very $P$-cosmall root. Let $d$ be a degree such that $s_{\alpha} W_P\preceq z_d^P W_P$. Then we have $d(\alpha)\leq d$.

\end{thm}

\begin{proof}

Let $\beta$ be an arbitrary simple root in $\Delta\setminus\Delta_P$. In order to prove that $d(\alpha)\leq d$ it suffices to prove that $(\omega_\beta,d(\alpha))\leq(\omega_\beta,d)$. If we naturally identify $H_2(G/P_\beta)$ with $\mathbb{Z}$ (as we usually do), then we have $(\omega_\beta,d(\alpha))=d(\alpha)+\mathbb{Z}\Delta_{P_\beta}^\vee$ and $(\omega_\beta,d)=d+\mathbb{Z}\Delta_{P_\beta}^\vee$. Since $P\subseteq P_\beta$, it follows from the relation $s_{\alpha}W_P\preceq z_d^P W_P$ and Proposition~\ref{prop:zd}(\ref{item:projection}) that 
$$
s_{\alpha}W_{P_\beta}\preceq z_d^P W_{P_\beta}\preceq z_{(\omega_\beta,d)}^{P_\beta}W_{P_\beta}\,.
$$ 
From this relation it directly follows that
$$
\ell(s_\alpha W_{P_\beta})\leq\ell\left(z_{(\omega_\beta,d)}^{P_\beta}\right)\,.
$$ 
By assumption $\alpha$ is $P_\beta$-cosmall, therefore we know by one of the equivalent definitions (cf. Remark~\ref{rem:pcosmall}) that 
$$
\ell(s_\alpha W_{P_\beta})=c_1(G/P_\beta)(\omega_\beta,d(\alpha))-1\,.
$$ 
From \cite[Theorem~6.2]{curvenbhd2} it follows that
$$
\ell\left(z_{(\omega_\beta,d)}^{P_\beta}\right)\leq c_1(G/P_\beta)(\omega_\beta,d)-1\,.
$$
Since $P_\beta$ is a maximal parabolic subgroup of $G$, all these facts together yield that $(\omega_\beta,d(\alpha))\leq(\omega_\beta,d)$ -- as required. 
\end{proof}

\begin{cor}
\label{cor:verycosmall}

Let $d$ be a degree such that $z_d^P=w_X$. Every greedy decomposition of $d$ starts with the highest root $\theta_1$.

\end{cor}

\begin{proof}

The highest root $\theta_1$ is obviously very $P$-cosmall. Therefore Theorem~\ref{thm:verycosmall} applies to the trivial relation $s_{\theta_1}W_P\preceq z_d^P W_P$. It follows that $d(\theta_1)\leq d$. Therefore every greedy decomposition of $d$ must start with $\theta_1$ -- as claimed.
\end{proof}

\begin{lem}
\label{lem:dualsmaller}

Let $d$ be a degree such that $z_d^P=w_X$. Let $\alpha$ be a root which occurs in a greedy decomposition of $d$. Then we have $s_\alpha^*\preceq z_{d-d(\alpha)}^P w_P$.

\end{lem}

\begin{proof}

Let $(\alpha_1,\ldots,\alpha_r)$ be a greedy decomposition of $d$. By Proposition~\ref{prop:zd}(\ref{item:unique}) we know that $\alpha=\alpha_i$ for some $1\leq i\leq r$. By Proposition~\ref{prop:zd}(\ref{item:greedy}) we know that $(\alpha_1,\ldots,\hat{\alpha}_i,\ldots,\alpha_r)$ is a greedy decomposition of $d-d(\alpha)$. Using Proposition \ref{prop:zd}(\ref{item:commute}) we then conclude that
$$
w_o=w_Xw_P=z_d^Pw_P=s_{\alpha_1}\cdot\ldots\cdot s_{\alpha_r}\cdot w_P=s_\alpha\cdot s_{\alpha_1}\cdot\ldots\cdot\hat{s}_{\alpha_i}\cdot\ldots\cdot s_{\alpha_r}\cdot w_P=s_\alpha\cdot z_{d-d(\alpha)}^Pw_P\,.
$$
Using Proposition~\ref{prop:hecke}(\ref{item:inverse}), (\ref{item:heckeustrich}) and Proposition~\ref{prop:zd}(\ref{item:zzinverse}) the previous equation gives
$$
s_\alpha^*=w_o^{-1}s_\alpha=\left(\left(z_{d-d(\alpha)}^Pw_P\right)^{-1}\cdot s_\alpha\right)s_\alpha\preceq\left(z_{d-d(\alpha)}^Pw_P\right)^{-1}\preceq z_{d-d(\alpha)}^Pw_P\,.
$$
This proves the desired relation.
\end{proof}

\subsection{The extended support of a degree}

Let $d$ be a degree. Let $(\alpha_1,\ldots,\alpha_r)$ be the greedy decomposition of $d$. We denote by $\Delta(d)$ the set of all simple roots $\beta\in\Delta\setminus\Delta_P$ such that $(d,\omega_\beta)>0$. We call $\Delta(d)$ the naive support of $d$. We define $\widetilde{\Delta}(d)$ to be the union
$
\widetilde\Delta(d)=\bigcup_{i=1}^r\Delta(\alpha_i)
$
and call it the extended support of $d$. The extended support is clearly well-defined since the the greedy decomposition is unique up to reordering. The naive and the extended support of $d$ are related by the obvious equation $\Delta(d)=\widetilde\Delta(d)\setminus\Delta_P$. 
In particular, if $P=B$ the two notions coincide.  

\begin{defn}

We say that a degree $d$ is a connected degree if $\widetilde\Delta(d)$ is a connected subset of the Dynkin diagram. We say that a degree $d$ is a disconnected degree if $\widetilde\Delta(d)$ is a disconnected subset of the Dynkin diagram.

\end{defn}

\begin{prop}[{\cite[Corollary~4.5]{curvenbhd2}}]
\label{prop:connecteddegree}

Let $d$ be a connected degree. Let $(\alpha_1,\ldots,\alpha_r)$ be a greedy decomposition of $d$. Then we have $\alpha_1\geq\alpha_i$ for all $1\leq i\leq r$. In particular, the first entry $\alpha_1$ of the greedy decomposition of $d$ is uniquely determined by $d$ -- does not depend on the choice of the greedy decomposition.

\end{prop}

\begin{proof}

Let $\gamma^\vee=\bigvee_{i=1}^r\alpha_i^\vee$ be the unique smallest element in $\mathbb{Z}\Delta^\vee$ that is greater than or equal to $\alpha_i^\vee$ for all $1\leq i\leq r$. By the connectedness assumption on $d$ and \cite[Lemma~4.4]{curvenbhd2}, it follows that $\gamma^\vee$ is a coroot. We define a root $\gamma$ in the obvious way as the dual of $\gamma^\vee$. We clearly have $\alpha_1^\vee\leq\gamma^\vee\leq\sum_{i=1}^r\alpha_i^\vee$ and thus $d(\alpha_1)\leq d(\gamma)\leq d$. By assumption, $\alpha_1$ is a maximal root of $d$, in particular a maximal root of $d(\gamma)$. A fortiori, it follows from this that $\alpha_1$ is also a maximal root of $\gamma^\vee$. By \cite[Lemma~4.4(a)]{curvenbhd2} there exists a unique maximal root of $\gamma^\vee$ and this unique maximal root is given by $\alpha_1$. But by definition of $\gamma^\vee$ we have $\alpha_i^\vee\leq\gamma^\vee$ for all $1\leq i\leq r$. Therefore the uniqueness of $\alpha_1$ gives $\alpha_1\geq\alpha_i$ for all $1\leq i\leq r$. The last statement follows easily from this.
\end{proof}

For a connected degree $d$, we usually write $\alpha(d)$ for the unique first entry of a (or any) greedy decomposition of $d$. With this notation we have for example $\widetilde\Delta(d)=\Delta(\alpha(d))$ for any connected degree $d$.

If $w\in W$, we define the support $\Delta(w)$ of $w$ to be the set $\Delta(w)=\{\beta\in\Delta\mid s_\beta\preceq w\}$.

\begin{prop}
\label{prop:support}

Let $d$ be a degree. Let $e\in H_2(G/B)$ be also a degree. Let $u,v,w\in W$.

\begin{enumerate}

\item
\label{item:def}

The support $\Delta(w)$ is the set of all simple roots $\beta$ such that $s_\beta$ occurs in some (or in any) reduced expression of $w$. 

\item
\label{item:relation}

If $u\preceq v$, then $\Delta(u)\subseteq\Delta(v)$.

\item
\label{item:supportanddot}

We have $\Delta(u\cdot v)=\Delta(u)\cup\Delta(v)$.

\item
\label{item:parabolic}

For any parabolic subgroup $Q$ of $G$ we have $\Delta(w_Q)=\Delta_Q$.

\item
\label{item:supportofs}

For any positive root $\alpha$ we have $\Delta(s_\alpha)=\Delta(\alpha)$.

\item 
\label{item:supportofz}

We have $\Delta\left(z_d^Pw_P\right)=\widetilde{\Delta}(d)\cup\Delta_P$, in particular $\Delta(z_e^B)=\widetilde\Delta(e)=\Delta(e)$.

\end{enumerate}

\end{prop}

\begin{proof}

Item~(\ref{item:def}) follows directly from the definition of $\Delta(w)$ and the definition of the Bruhat order. Note that the set of all simple roots $\beta$ such that $s_\beta$ occurs in a fixed reduced expression of $w$ does not depend on the choice of the reduced expression (cf. \cite[5.10, Corollary~(b)]{humphreys3}).

Item~(\ref{item:relation}) follows directly from Item~(\ref{item:def}) and the definition of the Bruhat order.

Ad Item~(\ref{item:supportanddot}). By Proposition~\ref{prop:hecke}(\ref{item:preceq}) we have $u\preceq u\cdot v$ and $v\preceq u\cdot v$. Item~(\ref{item:relation}) immediately implies that $\Delta(u)\cup\Delta(v)\subseteq\Delta(u\cdot v)$. By definition of the Hecke product, there exists a reduced expression of $u\cdot v$ which consists of a product of simple reflections $s_\beta$ where $\beta\in\Delta(u)\cup\Delta(v)$. Therefore the inclusion $\Delta(u\cdot v)\subseteq\Delta(u)\cup\Delta(v)$ follows from Item~(\ref{item:def}).

Item~(\ref{item:parabolic}) follows from Item~(\ref{item:def}) and a well-known fact on reduced expressions of the longest element of a Coxeter group (\cite[1.8, Exercise 2]{humphreys3}).

Ad Item~(\ref{item:supportofs}). Let $Q$ be the parabolic subgroup of $G$ such that $\Delta_Q=\Delta(\alpha)$. Then we clearly have $s_\alpha\preceq w_Q$. Item~(\ref{item:relation}) and Item~(\ref{item:parabolic}) imply that $\Delta(s_\alpha)\subseteq\Delta(w_Q)=\Delta_Q=\Delta(\alpha)$. This inclusion also means that $\alpha$ is contained in the $\mathbb{Z}$-span of $\Delta(\alpha)$. Since the simple roots are linearly independent, it is clear that $\alpha$ cannot be contained in the $\mathbb{Z}$-span of a proper subset of $\Delta(\alpha)$. Therefore the inclusion $\Delta(s_\alpha)\subseteq\Delta(\alpha)$ must be an equality.

Ad Item~(\ref{item:supportofz}). Let $(\alpha_1,\ldots,\alpha_r)$ be a greedy decomposition of $d$. By definition, we then have $z_d^Pw_P=s_{\alpha_1}\cdot\ldots\cdot s_{\alpha_r}\cdot w_P$. Item~(\ref{item:supportanddot}), (\ref{item:parabolic}) and (\ref{item:supportofs}) imply
$$
\Delta(z_d^Pw_P)=\Delta(s_{\alpha_1})\cup\cdots\cup\Delta(s_{\alpha_r})\cup\Delta(w_P)=\Delta(\alpha_1)\cup\cdots\cup\Delta(\alpha_r)\cup\Delta_P=\widetilde{\Delta}(d)\cup\Delta_P\,.
$$
The very last statement is obvious since $\Delta_B=\emptyset$ by definition.
\end{proof}

\todo[inline,color=green]{Still missing. Proof of Proposition~\ref{prop:support} included September 20.

The support of a \emph{positive root} has to be defined in the introduction (Proposition~\ref{prop:support}(\ref{item:supportofs})).}

\begin{prop}
\label{prop:relation}

Let $d$ be a degree. Let $(\alpha_1,\ldots,\alpha_r)$ be a greedy decomposition of $d$. For all simple roots $\beta\in\Delta\left(z_{d-d(\alpha_1)}^Pw_P\right)$ the relation $(\alpha_1,\beta)\geq 0$ holds.

\end{prop}

\begin{proof}

Let us write $\alpha=\alpha_1$ for short. By Proposition~\ref{prop:zd}(\ref{item:greedy}) we know that $(\alpha_2,\ldots,\alpha_r)$ is a greedy decomposition of $d-d(\alpha)$. Proposition~\ref{prop:support}(\ref{item:supportofz}) then gives the equality
$$
\Delta\left(z_{d-d(\alpha)}^Pw_P\right)=\Delta(\alpha_2)\cup\cdots\cup\Delta(\alpha_r)\cup\Delta_P\,.
$$
Let $\beta$ be a simple root as in the statement. By assumption, $\alpha$ is a maximal root of $d$, in particular $\alpha$ is a maximal root of $d(\alpha)+d(\beta)\leq d$. Let $e=\alpha^\vee+\beta^\vee\in H_2(G/B)$. Then we have $e+\mathbb{Z}\Delta_P^\vee=d(\alpha)+d(\beta)$ and $\alpha^\vee\leq e$. Therefore $\alpha$ is also a maximal root of $e$. By \cite[Proposition~4.8(b)]{curvenbhd2} we then know that $s_\alpha\cdot s_\beta=s_\beta\cdot s_\alpha$. By \cite[Proposition~4.8(a)]{curvenbhd2} we see that $(\alpha,\beta)\geq 0$ as required.
\end{proof}

\begin{cor}
\label{cor:pcosmall}

Let $\alpha$ be a $P$-cosmall root. Then we have $(\alpha,\beta)\geq 0$ for all simple roots $\beta\in\Delta_P$.

\end{cor}

\begin{proof}

This is Proposition~\ref{prop:relation} applied to the degree $d(\alpha)$ and its unique greedy decomposition consisting of one element $\alpha$. Here we use that $\Delta(w_P)=\Delta_P$ (cf. Proposition~\ref{prop:support}(\ref{item:parabolic})).
\end{proof}

\begin{lem}
\label{lem:roots}

Let $\alpha,\alpha'\in R$ be two roots such that $(\alpha,\alpha')\geq 0$ and such that $\alpha+\alpha'\in R^+$.
Then we have $0<(\alpha+\alpha')^\vee<\alpha^\vee+\alpha'^\vee$ and $\alpha^\vee+\alpha'^\vee\notin R^\vee$. 

\end{lem}

\begin{proof}

Let $\gamma=\alpha+\alpha'$ for short. Since $(\alpha,\alpha')\geq 0$, we know that $(\gamma,\gamma)\geq(\alpha,\alpha)+(\alpha',\alpha')$. Therefore there must be two root lengths. Moreover, $\gamma$ is a long root and both $\alpha$ and $\alpha'$ are short roots, in particular $(\alpha,\alpha)=(\alpha',\alpha')$. Let $n=(\gamma,\gamma)/(\alpha,\alpha)=(\gamma,\gamma)/(\alpha',\alpha')$. It is known that $n$ is always a positive integer. Since there are two root lengths, we must have $n>1$. (More concretely, we know by type considerations that $n\in\{2,3\}$.) A simple computation shows that $\alpha^\vee+\alpha'^\vee=n\gamma^\vee$. Since the root system $R^\vee$ is reduced and $n\neq\pm 1$, we see that $\alpha^\vee+\alpha'^\vee\notin R^\vee$. Since $\gamma>0$ we know that $\gamma^\vee>0$ and thus $\alpha^\vee+\alpha'^\vee=n\gamma^\vee>\gamma^\vee>0$. 
\end{proof}

\begin{cor}
\label{cor:relation}

Let $d$ be a degree. Let $\alpha$ and $\alpha'$ be two entries of a greedy decomposition of $d$. Then we have $(\alpha,\alpha')\geq 0$ and $\alpha+\alpha'\notin R$. 

\end{cor}

\begin{proof}

If $\alpha=\alpha'$, then there is nothing to prove. Suppose from now on that $\alpha$ and $\alpha'$ are distinct. Let $(\alpha_1,\ldots,\alpha_r)$ be a greedy decomposition of $d$. By the uniqueness of the greedy decomposition up to reordering we know that $\alpha=\alpha_i$ and $\alpha'=\alpha_j$ for some $1\leq i,j\leq r$. By replacing $\alpha$ and $\alpha'$ if necessary we may assume that $i<j$. By Proposition~\ref{prop:zd}(\ref{item:greedy}) we know that $(\alpha,\alpha')$ is a greedy decomposition of $d(\alpha)+d(\alpha')$. Let $d'=d(\alpha)+d(\alpha')$ for short. By Proposition~\ref{prop:relation} applied to the greedy decomposition $(\alpha,\alpha')$ of $d'$ we know that for all $\beta\in\Delta(\alpha')$ the relation $(\alpha,\beta)\geq 0$ holds. By bilinearity this means $(\alpha,\alpha')\geq 0$.  

Suppose now for a contradiction that $\alpha+\alpha'\in R$. By assumption $\alpha$ is a maximal root of $d'$. A fortiori, $\alpha$ is a maximal root of $\alpha^\vee+\alpha'^\vee$. It is clear that $\alpha<\alpha+\alpha'$. Moreover, by Lemma~\ref{lem:roots} we know that $(\alpha+\alpha')^\vee<\alpha^\vee+\alpha'^\vee$. This contradicts the maximality of $\alpha$. Therefore we find that $\alpha+\alpha'$ cannot be a root.
\end{proof}

\begin{cor}
\label{cor:stronglyorthogonal2}

Let $d$ be a degree. Let $(\alpha_1,\ldots,\alpha_r)$ be a greedy decomposition of $d$. For all simple roots $\beta\in\Delta\left(z_{d-d(\alpha_1)}^Pw_P\right)$ we have $\alpha_1+\beta\notin R$.

\end{cor}

\begin{proof}

Let $\alpha_1$ and $\beta$ be as in the statement. Let us write $\alpha=\alpha_1$ for short. It is easy to see that $(\alpha,\beta)$ is a greedy decomposition of $\alpha^\vee+\beta^\vee$. By Corollary~\ref{cor:relation} applied to the degree $\alpha^\vee+\beta^\vee$ it follows that $\alpha+\beta\notin R$.
\end{proof}

\begin{ex}
\label{ex:highestroot}

Let $P_1,\ldots,P_k$ be parabolic subgroups such that $\Delta=\bigcup_{i=1}^k\Delta_{P_i}$. Let $\alpha$ be a root which is $P_i$-cosmall for all $1\leq i\leq k$. Then $\alpha$ must be the highest root $\theta_1$. Indeed, Corollary~\ref{cor:pcosmall} implies that $(\alpha,\beta)\geq 0$ for all $\beta\in\Delta$. If $\alpha\neq\theta_1$, then there exists $\beta\in\Delta$ such that $\alpha<s_\beta(\alpha)$ and thus $(\alpha,\beta)<0$. Therefore we must have $\alpha=\theta_1$.

\end{ex}

\begin{ex}
\label{ex:verypcosmall}

If $P$ is not maximal, then the only very $P$-cosmall root is the highest root. If $P$ is maximal, then the notions $P$-cosmall and very $P$-cosmall are equivalent. This follows directly from Example~\ref{ex:highestroot}.

\end{ex}

\todo[inline,color=green]{For all $\alpha_1,\beta$ as in Proposition~\ref{prop:relation} we know $\alpha_1+\beta\notin R$. As a consequence, we prove strong orthogonality for simple roots in Section~?? / after Theorem~\ref{thm:orthogonality}. Use a greedy sub-decomposition $(\alpha_1,\beta)$.}

\section{The distance function} 
\label{sec:delta}

In this section, we introduce the distance function $\delta_P$ on the Weyl group $W$ with values in the power set of $H_2(X)$. We give a definition which has an immediate geometric interpretation in terms of curve neighborhoods. Later on in Section~\ref{sec:delta2}, we will see that this function is closely related to the quantum cohomology of $X$ (cf. Theorem~\ref{thm:fulton}) and also that it can be described more combinatorially in terms of 
chains (cf. Theorem~\ref{thm:descriptionofdelta}). We will focus in this section on the question when $\delta_P(w)$ where $w\in W$ consists of unique element. In particular, for our further investigations it is important to show that $\delta_P(w_o)$ consists of a unique element $d_X$ (cf. Theorem~\ref{thm:uniqueness}). Although the statement $\delta_P(w_o)=\{d_X\}$ is very natural, we will complete the full proof of it only in Section~\ref{sec:cascade} (cf. Theorem~\ref{thm:curve}). 
Our approach to the problem has the advantage that it gives by the way a simple formula to compute $d_X$ (cf. Corollary~\ref{cor:identity})

\begin{defn}
\label{def:deltaw}

Let $w\in W$. Then we define $\delta_P(w)$ to be set of all minimal elements of the set
$$
\left\{d\text{ a degree such that }wW_P\preceq z_d^PW_P\right\}\,.
$$
We call the function $\delta_P\colon W\to\mathcal{P}(H_2(X))$ the distance function. Here we denote by $\mathcal{P}(M)$ the power set of a set $M$.

\end{defn}

\begin{rem}
\label{rem:uniqueelement}

Let $w\in W$. By definition, any two distinct elements of $\delta_P(w)$ are incomparable. For all simple roots $\beta$ the lattice $H_2(G/P_\beta)=\mathbb{Z}$ is totally ordered. Therefore any two elements of $\delta_{P_\beta}(w)$ are comparable. It follows that the set $\delta_{P_\beta}(w)$ always consists of a unique element. 

\end{rem}

\begin{conv}
\label{conv:uniqueelement}

Let $w\in W$. If $\delta_P(w)$ consists of a unique element $d$, we identify the set $\delta_P(w)$ with its unique element and write $\delta_P(w)=d$. In particular, for all simple roots $\beta$ we identify $\delta_{P_\beta}$ with a function with values in $\mathbb{Z}$ (cf. Remark~\ref{rem:uniqueelement}). To abbreviate, we denote the unique element of $\delta_{P_\beta}(w_o)$ by $d_{G/P_\beta}$. Following the convention, we write $\delta_{P_\beta}(w_o)=d_{G/P_\beta}$.

\end{conv}

\begin{prop}
\label{prop:delta}

Let $u,v,w\in W$. 

\begin{enumerate}

\item
\label{item:invariant}

The function $\delta_P$ is $(W_P,W_P)$-invariant, i.e. we have $\delta_P(uwv)=\delta_P(w)$ for all $u,v\in W_P$.

\item
\label{item:invariantinverse}

The function $\delta_P$ is invariant under taking inverses, i.e. we have $\delta_P(w)=\delta_P(w^{-1})$ for all $w\in W$.

\item
\label{item:usmallerv}

Suppose that $uW_P\preceq vW_P$. For all $d\in\delta_P(v)$ there exists a degree $d'\in\delta_P(u)$ such that $d'\leq d$.

\item
\label{item:triangle}

For all $d\in\delta_P(u)$ and all $d'\in\delta_P(v)$ there exists a degree $d''\in\delta_P(u\cdot v)$ such that $d''\leq d+d'$. 

\item
\label{item:projection2}

Let $Q$ be a parabolic subgroup of $G$ containing $P$. For all $d\in\delta_P(w)$ there exists a degree $e\in\delta_Q(w)$ such that $e\leq d+\mathbb{Z}\Delta_Q^\vee$.

\item
\label{item:pcosmall}

Let $\alpha$ be a $P$-cosmall root. Then we have $d(\alpha)\in\delta_P(s_\alpha)$.

\item
\label{item:verycosmall}

Let $\alpha$ be a very $P$-cosmall root. Then we have $\delta_P(s_\alpha)=d(\alpha)$.

\item
\label{item:simple}

Let $\beta$ be a simple root. Then we have $\delta_P(s_\beta)=d(\beta)$.

\item
\label{item:reduce}

Let $d$ be a degree such that $d\in\delta_P(z_d^P)$. Let $\alpha$ be a root which occurs in a greedy decomposition of $d$. Then we have $d-d(\alpha)\in\delta_P(z_{d-d(\alpha)}^P)$.

\item
\label{item:d}

For all $d\in\delta_P(w)$ we have $d\in\delta_P(z_d^P)$.

\end{enumerate}

\end{prop}

\todo[inline,color=green]{Two properties are missing. See notes on September 15 (completed) and July 10: projection from $B$ to $P$ (this is now (\ref{item:projection2})) and\ldots for all $d\in\delta_P(w)$ we have $d\in\delta_P(z_d^P)$ (as last item). Actually three: invariance under inverses. 

Characterization of the highest root as $P$-cosmall for all $P$.

Counterexample to the investigations in the subsection \enquote{The simply laced case} (suppose that $R$ is simply laced, i.e. of type $\mathsf{ADE}$ -- I can assume this notion is well-known) (included as Example~\ref{ex:5.9} and Example~\ref{ex:5.9'}). 

Further point / property: For all $\beta\in\Delta$ we have $\delta_P(s_\beta)=d(\beta)$ (notes on September 24) (included as Proposition~\ref{prop:delta}(\ref{item:simple})).}

\begin{proof}

Ad Item~(\ref{item:invariant}). It is immediately clear from the definition of $\delta_P$ that $\delta_P(wv)=\delta_P(w)$ for all $v\in W_P$. On the other hand, let $u\in W_P$, $w\in W$ and $d\in\delta_P(w)$. By Proposition~\ref{prop:hecke}(\ref{item:preceq}), (\ref{item:uvsmaller}), (\ref{item:preceqmodp}) and Proposition~\ref{prop:zd}(\ref{item:stabofz}) we then have $uwW_P\preceq u\cdot wW_P\preceq w_P\cdot z_d^PW_P=z_d^PW_P$. Therefore there exists an element $d'\in\delta_P(uw)$ such that $d'\leq d$. Similarly, we see (by considering the expression $w=u^{-1}(uw)$) that for each element $d\in\delta_P(uw)$ there exists an element $d'\in\delta_P(w)$ such that $d'\leq d$. Altogether, this implies that $\delta_P(uw)=\delta_P(w)$ (cf. Remark~\ref{rem:uniqueelement}).\footnote{A different way to see that for all $u\in W_P$, $w\in W$ we have $\delta_P(uw)=\delta_P(w)$ is to use Item~(\ref{item:invariantinverse}). Indeed, we have $\delta_P(uw)=\delta_P(w^{-1}u^{-1})=\delta_P(w^{-1})=\delta_P(w)$ since $u^{-1}\in W_P$.}

Ad Item~(\ref{item:invariantinverse}). Let $d\in\delta_P(w)$. Then we have $w\preceq z_d^Pw_P$ and thus, by Proposition~\ref{prop:zd}(\ref{item:zzinverse}), $w^{-1}\preceq(z_d^P w_P)^{-1}\preceq z_d^Pw_P$. This means that there exists an element $d'\in\delta_P(w^{-1})$ such that $d'\leq d$. Similarly, we see (by considering the expression $w=(w^{-1})^{-1}$) that for each element $d\in\delta_P(w^{-1})$ there exists an element $d'\in\delta_P(w)$ such that $d'\leq d$. Altogether, this implies $\delta_P(w)=\delta_P(w^{-1})$ (cf. Remark~\ref{rem:uniqueelement}).

Item~(\ref{item:usmallerv}) is immediately clear from the definition of $\delta_P$.

Ad Item~(\ref{item:triangle}). Let $d$ and $d'$ be as in the statement. Then we have $uW_P\preceq z_d^PW_P$ and $vW_P\preceq z_{d'}^PW_P$ or equivalent $u\preceq z_d^Pw_P$ and $v\preceq z_{d'}^Pw_P$. By Proposition~\ref{prop:hecke}(\ref{item:preceq}), (\ref{item:minimal}), Proposition~\ref{prop:zd}(\ref{item:stabofz}) and \cite[Corollary~4.12(b)]{curvenbhd2} it follows that 
$$
u\cdot v\preceq z_d^Pw_P\cdot z_{d'}^Pw_P=z_d^P\cdot w_P\cdot z_d^Pw_P=z_d^P\cdot z_{d'}^Pw_P\preceq z_{d+d'}^Pw_P\,.
$$ 
Reduction modulo $W_P$ directly leads to $u\cdot vW_P\preceq z_{d+d'}^P W_P$. This implies that there exists a $d''\in\delta_P(u\cdot v)$ such that $d''\leq d+d'$.

Ad Item~(\ref{item:projection2}). Let $d\in\delta_P(w)$. By definition and Proposition~\ref{prop:zd}(\ref{item:projection}) we have $w\preceq z_d^Pw_P\preceq z_{d+\mathbb{Z}\Delta_Q^\vee}^Qw_Q$ and thus $wW_Q\preceq z_{d+\mathbb{Z}\Delta_Q^\vee}^QW_Q$. Therefore there exists a $e\in\delta_Q(w)$ such that $e\leq d+\mathbb{Z}\Delta_Q^\vee$ -- as claimed.

Ad Item~(\ref{item:pcosmall}). By definition there exists a degree $d\in\delta_P(s_\alpha)$ such that $d\leq d(\alpha)$. This degree obviously satisfies $s_\alpha W_P=z_d^PW_P$. Therefore \cite[Theorem~6.2]{curvenbhd2} and the fact that $\alpha$ is a $P$-cosmall root (cf. Remark~\ref{rem:pcosmall}) give that
$$
(c_1(X),d(\alpha))-1=\ell(s_\alpha W_P)=\ell\left(z_d^P\right)\leq(c_1(X),d)-1
$$
and thus $(c_1(X),d(\alpha)-d)\leq 0$. Again, by \cite[Theorem~6.2]{curvenbhd2} this means that $d(\alpha)-d$ cannot be positive. Since $d\leq d(\alpha)$ by the choice of $d$, it follows that $d=d(\alpha)$ and thus $d(\alpha)\in\delta_P(s_\alpha)$.

Ad Item~(\ref{item:verycosmall}). By Theorem~\ref{thm:verycosmall} every element $d\in\delta_P(s_\alpha)$ satisfies $d(\alpha)\leq d$. Item~(\ref{item:pcosmall}) and Remark~\ref{rem:uniqueelement} immediately implies that $\delta_P(s_\alpha)=d(\alpha)$.

Ad Item~(\ref{item:simple}). If $\beta\in\Delta_P$, then we clearly have $\delta_P(s_\beta)=0$. Assume that $\beta\in\Delta\setminus\Delta_P$. Let $d\in\delta_P(s_\beta)$. Then we have $s_\beta\preceq z_d^Pw_P$. Proposition~\ref{prop:support}(\ref{item:supportofz}) implies that $\beta\in\widetilde\Delta(d)\cup\Delta_P$ and thus $\beta\in\widetilde\Delta(d)\setminus\Delta_P=\Delta(d)$. This means $d(\beta)\leq d$. On the other hand, by \cite[Corollary~4.12(a)]{curvenbhd2} we clearly have $s_\beta W_P\preceq z_{d(\beta)}^PW_P$. This immediately implies $d(\beta)\in\delta_P(s_\beta)$ and thus $d=d(\beta)$ by Remark~\ref{rem:uniqueelement}. In total, we have $\delta_P(s_\beta)=d(\beta)$ as claimed. 

Ad Item~(\ref{item:reduce}). Let $(\alpha_1,\ldots,\alpha_r)$ be a greedy decomposition of $d$. By Proposition~\ref{prop:zd}(\ref{item:unique}) we know that $\alpha=\alpha_i$ for some $1\leq i\leq r$. By Proposition~\ref{prop:zd}(\ref{item:greedy}) we know that $(\alpha_1,\ldots,\hat{\alpha}_i,\ldots,\alpha_r)$ is a greedy decomposition of $d-d(\alpha)$. Using Proposition~\ref{prop:zd}(\ref{item:commute}) we conclude that $z_d^Pw_P=s_\alpha\cdot z_{d-d(\alpha)}^Pw_P$. By definition there exists a $d'\in\delta_P(z_{d-d(\alpha)}^P)$ such that $d'\leq d-d(\alpha)$. Item~(\ref{item:triangle}) and (\ref{item:pcosmall}) then imply that there exists a $d''\in\delta_P(z_d^P)$ such that $d''\leq d(\alpha)+d'\leq d$. Since $d\in\delta_P(z_d^P)$ by assumption, this implies that $d''=d$ and thus $d'=d-d(\alpha)$.

Item~\ref{item:d} is immediately clear from the definition of $\delta_P$.
\end{proof}

\begin{lem}
\label{lem:bound}

Let $d$ be a degree such that $z_d^P=w_X$. Then we have $d_{G/P_\beta}\leq(\omega_\beta,d)$ for all $\beta\in\Delta\setminus\Delta_P$.

\end{lem}

\begin{rem}

From a geometric point of view, Lemma~\ref{lem:bound} says that if a curve in $X$ passes through two general points in $X$, then the image of this curve under $X\to G/P_\beta$ passes through two general points in $G/P_\beta$ where $\beta\in\Delta\setminus\Delta_P$. The combinatorial argument we give in the following proof has the advantage that it works for every finite Coxeter group (even if there is no geometry attached to it).

\end{rem}

\begin{proof}

Let $\beta$ be an arbitrary simple root in $\Delta\setminus\Delta_P$. Then we clearly have $P\subseteq P_\beta$. Thus Proposition~\ref{prop:zd}(\ref{item:projection}) yields $w_o=z_d^Pw_P=z_{(\omega_\beta,d)}^{P_\beta}w_{P_\beta}$. By definition of $\delta_{P_\beta}$ this gives $d_{G/P_\beta}\leq(\omega_\beta,d)$.
\end{proof}

\begin{thm}
\label{thm:uniqueness}

The set $\delta_P(w_o)$ consists of a unique element $d_X$; in formulas $\delta_P(w_X)=d_X$. This element $d_X$ is given by 
$$
d_X=\sum_{\beta\in\Delta\setminus\Delta_P}d_{G/P_\beta}d(\beta)\,.
$$

\end{thm}

\begin{rem}
\label{rem:circular}

The complete proof of Theorem~\ref{thm:uniqueness} relies on techniques which we will only develop later in Section~\ref{sec:cascade}, in particular on Theorem~\ref{thm:curve}. It is convenient to state the theorem already now and draw conclusions. The reader can convince himself that we have avoided circular reasoning everywhere. For this, it suffices to remark that the statement is obvious if $P$ itself is a maximal parabolic subgroup (cf. Remark~\ref{rem:uniqueelement} and Convention~\ref{conv:uniqueelement}).  

\end{rem}

\begin{proof}

Let us write $d_X'=\sum_{\beta\in\Delta\setminus\Delta_P}d_{G/P_\beta}d(\beta)$ for the purpose of this proof. Lemma~\ref{lem:bound} shows that every element $d\in\delta_P(w_o)$ satisfies $d_X'\leq d$. By Theorem~\ref{thm:curve}, there exists a curve of degree $d_X'$ which passes through $x(1)$ and $x(w_o)$. In other words, we have $z_{d_X'}^P=w_X$. Therefore there exists an element $d'\in\delta_P(w_o)$ such that $d'\leq d_X'$ (cf. Proposition~\ref{prop:delta}(\ref{item:invariant})). Altogether, this implies that $d_X'$ is the unique element of $\delta_P(w_o)$ (cf. Remark~\ref{rem:uniqueelement}). We have $d_X=d_X'$ with the notation from the statement. 
\end{proof}

\begin{prop}
\label{prop:geqdx}

Let $u\in W$. For all $d\in\delta_P(u)$ and all $d^*\in\delta_P(u^*)$ we have $d+d^*\geq d_X$.

\end{prop}

\begin{proof}

We obviously have 
$$
\{x(u)\}=X_u\cap Y_u\subseteq X_{z_d^P}\cap Y_{\left(z_{d^*}^P\right)^*}
$$
since $uW_P\preceq z_d^PW_P$ and $u^*\preceq z_{d^*}^PW_P$ by assumption. In particular, this implies that $\Gamma_d(X_1)\cap\Gamma_{d^*}(Y_{w_o})\neq\emptyset$ (cf. \cite[Remark~5.2]{curvenbhd2}). Therefore we can find a curve of degree $d+d^*$ which passes through $x(1)$ and $x(w_o)$. This means that $z_{d+d^*}^P=w_X$. The minimality and the uniqueness of $d_X$ (Theorem~\ref{thm:uniqueness}) imply that $d+d^*\geq d_X$.
\end{proof}

\begin{cor}
\label{cor:geqdx}

Let $u\in W$. Suppose that there exists a $d\in\delta_P(u)$ and a $d^*\in\delta_P(u^*)$ such that $d+d^*=d_X$. Then we have $\delta_P(u)=d$ and $\delta_P(u^*)=d^*$.

\end{cor}

\begin{proof}

Let $d'$ be an arbitrary element of $\delta_P(u)$. By Proposition~\ref{prop:geqdx} and the assumption we have $d'+d^*\geq d_X=d+d^*$. It follows that $d\leq d'$. Since two distinct elements of $\delta_P(u)$ are incomparable (Remark~\ref{rem:uniqueelement}), it follows that $d=d'$ and thus $\delta_P(u)=d$. Similarly, we find $\delta_P(u^*)=d^*$.
\end{proof}

\begin{thm}
\label{thm:inductive}

Let $\alpha$ be a root which occurs in a greedy decomposition of $d_X$. Then we have
$$
\delta_P(s_\alpha)=\delta_P\left(\left(z_{d_X-d(\alpha)}^{P}\right)^*\right)=d(\alpha)\text{ and }\delta_P(s_\alpha^*)=\delta_P\left(z_{d_X-d(\alpha)}^P\right)=d_X-d(\alpha)\,.
$$

\end{thm}

\begin{proof}

Since $\alpha$ is $P$-cosmall (Proposition~\ref{prop:zd}(\ref{item:unique})), Proposition~\ref{prop:delta}(\ref{item:pcosmall}) shows that $d(\alpha)\in\delta_P(s_\alpha)$. By Lemma~\ref{lem:dualsmaller} there exists a degree $d\in\delta_P(s_\alpha^*)$ such that $d\leq d_X-d(\alpha)$. Proposition~\ref{prop:geqdx} implies that $d(\alpha)+d\geq d_X$. This means that $d=d_X-d(\alpha)$. Therefore Corollary~\ref{cor:geqdx} shows that $\delta_P(s_\alpha)=d(\alpha)$ and $\delta_P(s_\alpha^*)=d_X-d(\alpha)$.

Let $z=z_{d_X-d(\alpha)}^P$ for short. By Proposition~\ref{prop:delta}(\ref{item:usmallerv}) and the uniqueness of $\delta_P(s_\alpha^*)$ it follows that we have $d_X-d(\alpha)\leq d$ for all $d\in\delta_P(z)$. To prove that $\delta_P(z)=d_X-d(\alpha)$ it therefore suffices to show that $d_X-d(\alpha)\in\delta_P(z)$. This latter statement is obvious since $\delta_P(z)$ contains by definition a degree $d\leq d_X-d(\alpha)$.

By Lemma~\ref{lem:dualsmaller} we have $s_\alpha^* W_P\preceq zW_P$ and thus $z^*W_P\preceq s_\alpha W_P$. By Proposition~\ref{prop:delta}(\ref{item:usmallerv}) and the previous result $\delta_P(s_\alpha)=d(\alpha)$, this means that there exists a $d\in\delta_P(z^*)$ such that $d\leq d(\alpha)$. By Proposition~\ref{prop:geqdx} and the previous result $\delta_P(z)=d_X-d(\alpha)$, it follows that $d_X-d(\alpha)+d\geq d_X$ and thus $d=d(\alpha)$. By Corollary~\ref{cor:geqdx} this implies $\delta_P(z^*)=d(\alpha)$.
\end{proof}

\begin{cor}
\label{cor:inductive}

Let $\theta_1$ be the highest root. Then we have
$$
\delta_P(s_{\theta_1})=\delta_P\left(\left(z_{d_X-d(\theta_1)}^P\right)^*\right)=d(\theta_1)\text{ and }\delta_P(s_{\theta_1}^*)=\delta_P\left(z_{d_X-d(\theta_1)}^P\right)=d_X-d(\theta_1)\,.
$$

\end{cor}

\begin{proof}

By Corollary~\ref{cor:verycosmall} the highest root $\theta_1$ occurs in a greedy decomposition of $d_X$. Therefore Theorem~\ref{thm:inductive} applies and yields the result.
\end{proof}

\subsection{Restriction and induction of degrees}

Let $d$ be a degree. Let $Q$ be a parabolic subgroup of $G$ containing $P$. Let $e\in H_2(G/Q)$ be a degree.
We call the image $d_Q$ of $d$ under the natural map $H_2(X)\to H_2(G/Q)$ the restriction of $d$, i.e we have $d_Q=d+\mathbb{Z}\Delta_Q^\vee$.
Let $(\alpha_1,\ldots,\alpha_r)$ be a greedy decomposition of $e$. Then 
we define a degree $e^P\in H_2(X)$ by the equation
$
e^P=\sum_{i=1}^r d(\alpha_i)
$.
The degree $e^P$ clearly does not depend on the choice of the greedy decomposition of $e$, since the greedy decomposition of $e$ is unique up to reordering. We call the degree $e^P$ the induction of $e$. We have the obvious identity $(e^P)_Q=e$. If $Q=P_\beta$ for some $\beta\in\Delta\setminus\Delta_P$, we simplify the notation and write $d_{P_\beta}=d_\beta$ for short. With the usual identification $H_2(G/P_\beta)=\mathbb{Z}$, this means $d_\beta=(\omega_\beta,d)$.




\todo[inline,color=yellow]{I wanted to say, and refer to it in the next theorem, that restriction is compatible with the partial order, and that the greedy decomposition of the induction stays the same (cf. out-commented fact before). --- I won't do this is. The general statement will be clear from the context. The first three sentences of the proof of Theorem~\ref{thm:resandind} are enough.}

\begin{thm}
\label{thm:resandind}

Let $Q$ be a parabolic subgroup of $G$ containing $P$. Let $e\in H_2(G/Q)$ be a degree. Then there exists a degree $d\in\delta_P(z_d^P)$ such that
$d_Q\leq e$ and such that 
\begin{equation}
\label{eq:induction}
z_e^Qw_Q=z_{e^P}^P\cdot w_Q=z_d^P w_P\,.
\end{equation}
If in addition $e\in\delta_Q(z_e^Q)$, then there exists a degree $d\in\delta_P(z_d^P)$ such that $d_Q=e$ and such that Equation~(\ref{eq:induction}) is satisfied.

\end{thm}

\begin{proof}

Let $(\alpha_1,\ldots,\alpha_r)$ be a greedy decomposition of $e$. Note that for two degree $d,d'\in H_2(X)$ such that $d\leq d'$ the inequality $d_Q\leq d_Q'$ holds. Therefore it is easy to see that $(\alpha_1,\ldots,\alpha_r)$ is also a greedy decomposition of $e^P$. By definition and Proposition~\ref{prop:hecke}(\ref{item:preceq}), (\ref{item:uvsmaller}), it follows that
$$
z_e^Qw_Q=s_{\alpha_1}\cdot\ldots\cdot s_{\alpha_r}\cdot w_Q= s_{\alpha_1}\cdot\ldots\cdot s_{\alpha_r}\cdot w_P\cdot w_Q=z_{e^P}^Pw_P\cdot w_Q\preceq z_{e^P}^P\cdot w_P\cdot w_Q=z_{e^P}^P\cdot w_Q\,.
$$
Here we used twice the fact that $w_P\cdot w_Q=w_Q$. On the other hand, it is clear that $z_{e^P}^P\preceq z_{e^P}^Pw_P$ and thus, by Proposition~\ref{prop:hecke}(\ref{item:preceq}), that $z_{e^P}^P\cdot w_Q\preceq z_{e^P}^Pw_P\cdot w_Q=z_e^Qw_Q$. Thus, we find the first claimed equality $z_e^Qw_Q=z_{e^P}^P\cdot w_Q$.

It is intuitively obvious and also easy to see formally that there exists a degree $d'\in\mathbb{Z}\Delta_Q^\vee/\mathbb{Z}\Delta_P^\vee$ such that $w_Q\preceq z_{d'}^Pw_P$.\footnote{We will address such questions more systematically in Section~\ref{sec:local}. An ad hoc fashion to see this is to compile a reduced expression $w_Q=s_{\beta_1}\cdots s_{\beta_l}$. Then we know that $\{\beta_1,\ldots,\beta_l\}=\Delta(w_Q)=\Delta_Q$ (cf. Proposition~\ref{prop:support}(\ref{item:def}),   (\ref{item:parabolic})). By Proposition~\ref{prop:hecke}(\ref{item:uvsmaller}) and Proposition~\ref{prop:delta}(\ref{item:usmallerv}), (\ref{item:triangle}), (\ref{item:simple}) there exists a $d'\in\delta_P(w_Q)$ such that $d'\leq d(\beta_1)+\cdots+d(\beta_l)$. It clearly follows that $d'\in\mathbb{Z}\Delta_Q^\vee/\mathbb{Z}\Delta_P^\vee$ and that $w_Q\preceq z_{d'}^Pw_P$.}
%
%
\todo[color=green]{I should cite Proposition~\ref{prop:zd}(\ref{item:stabofz}) (in the new footnote this is no longer necessary). Note that I use $d(\beta_i)$ for $\beta\in\Delta_P$ (I remarked this in the block in the introduction).
Even better to use the triangle inequality for $\delta_P$. Even better to cite also $\delta_\varphi=\delta$. I cannot use Theorem~\ref{thm:compatibility} because $\Delta_Q$ is not necessarily connected.} 
Using Proposition~\ref{prop:hecke}(\ref{item:preceq}) and \cite[Corollary~4.12(b)]{curvenbhd2}, we see that
$$
z_e^Qw_Q=z_{e^P}^P\cdot w_Q\preceq z_{e^P}^P\cdot z_{d'}^Pw_P\preceq z_{e^P+d'}^Pw_P\,.
$$
By definition of $\delta_P$, it now follows that there exists a degree $d\in\delta_P(z_e^Qw_Q)$ such that $d\leq e^P+d'$. Restriction of the inequality gives $d_Q\leq e$. Using Proposition~\ref{prop:zd}(\ref{item:projection}), we obtain that
\begin{equation}
\label{eq:restriction}
z_e^Qw_Q\preceq z_d^Pw_P\preceq z_{d_Q}^Qw_Q\preceq z_e^Qw_Q\,.
\end{equation}
This shows the second claimed equality $z_e^Qw_Q=z_d^Pw_P$ for some $d\in\delta_P(z_d^P)$ such that $d_Q\leq e$.

Finally, assume that $e\in\delta_Q(z_e^Q)$. Let $d$ be defined as before. Equation~(\ref{eq:restriction}) shows that $z_e^Qw_Q=z_{d_Q}^Qw_Q$. By the minimality of $e$, it follows that $d_Q$ cannot be strictly smaller than $e$. Since $d_Q\leq e$, it follows that $d_Q=e$.
\end{proof}

\subsection{The simply laced case}

In this subsection we briefly want to mention some properties of $\delta_P$ which are specific for the case of $R$ being simply laced. This subsection will not be logically needed anywhere else in the text. The impatient reader can skip it.

\begin{lem}[{\cite[Lemma~5.14]{thesis}}]
\label{lem:5.14}

Suppose that $R$ is simply laced. Let $\alpha$ be a positive root. Let $\beta$ be a simple root. Then we have $\delta_{P_\beta}(s_\alpha)=(\omega_\beta,\alpha^\vee)$.

\end{lem}

\begin{proof}

This is precisely the statement of \cite[Lemma~5.14]{thesis}. The reader finds a complete proof there. The proof relies on a elementary numerical analysis which we do not want to unwrap here. Also, the proof naturally makes use of the description of $\delta_{P_\beta}(s_\alpha)$ in terms of chain which start at $s_\alpha W_{P_\beta}$ (cf. Proposition~\ref{prop:equalu} and Theorem~\ref{thm:descriptionofdelta}). All the necessary techniques to fully understand the proof will be developed in Section~\ref{sec:delta2}.
\end{proof}

\begin{thm}
\label{thm:simplylaced}

Suppose that $R$ is simply laced. Let $\alpha$ be a positive root. Then we have $\delta_P(s_\alpha)=d(\alpha)$.

\end{thm}

\begin{proof}

Let $d\in\delta_P(s_\alpha)$. By Proposition~\ref{prop:delta}(\ref{item:projection2}) and Lemma~\ref{lem:5.14} it follows that $(\omega_\beta,d(\alpha))\leq d_\beta$ for all $\beta\in\Delta\setminus\Delta_P$, in other words that $d(\alpha)\leq d$. On the other hand, by \cite[Corollary~4.12(a)]{curvenbhd2}, we clearly have $s_\alpha W_P\preceq z_{d(\alpha)}^PW_P$, so that there exists a $d'\in\delta_P(s_\alpha)$ such that $d'\leq d(\alpha)\leq d$. Remark~\ref{rem:uniqueelement} implies that $d=d(\alpha)$ and thus $\delta_P(s_\alpha)=d(\alpha)$.
\end{proof}

\begin{rem}

Later on, in Example~\ref{ex:5.9} and Example~\ref{ex:5.9'}, we will give instructive counterexamples for Lemma~\ref{lem:5.14} and Theorem~\ref{thm:simplylaced} in the case where $R$ is not simply laced.

\end{rem}

\section{Description of the distance function in terms of \texorpdfstring{$T$-invariant}{T-invariant} curves: relation to minimal degrees in quantum products}
\label{sec:delta2}

In this section, we give 
two more equivalent descriptions of the distance function $\delta_P$. Let $u,v,w\in W$. Firstly, using the relation between $T$-invariant curves and chains, it is not hard to see that the set $\delta_P(u)$ consists of all minimal degrees of chains from $uW_P$ to $w_oW_P$ (cf. Theorem~\ref{thm:descriptionofdelta}). On the other hand, it is an insight of Fulton-Woodward that the minimal degrees of chains from $uW_P$ to $vW_P$ are precisely the minimal degrees in the quantum product $\sigma_u\star\sigma_v$ (\cite[Thoerem~9.1: $(1)\Leftrightarrow(2)$]{fulton}). Thus, we secondly find that the set $\delta_P(w)$ consists of all minimal degrees in the quantum product $\sigma_w\star\mathrm{pt}$ (cf. Theorem~\ref{thm:fulton}). 
It should be noted that all the results in this section can be understood as corollaries of the main theorem by Fulton-Woodward \cite[Theorem~9.1]{fulton}.

\todo[inline,color=green]{Convention concerning $\delta_P(u,v)=0$. Remark concerning incomparability. The nature is similar to the of $\delta_P(u)$. In fact, there is a close connection, cf.\ldots

The notion of a chain is equivalent to the notion of a $T$-invariant curve. Can refer to the relevant statement / remark in the thesis, for more details\ldots I decided not to include this remark, because the relation between $T$-invariant curves and chains is still mysterious to me (cf. blue colored passage after Proposition~\ref{prop:equalu}). Neither a chain determines a $T$-invariant curve, nor a $T$-invariant curve determines a chain. Both directions seem strange to me.

Remark / Convention about the use of $\bar u$ for elements in $W/W_P$ represented by $u$.}

\begin{notation}

Let $u,v\in W$. We always denote by $\bar u=uW_P$ the class of $u$ modulo $W_P$. If we speak about a class $\bar u\in W/W_P$ without specific reference to $u$, we always mean that $u\in W$ such that $\bar u=uW_P$. Usually it does not matter in this case which representative $u$ of $\bar u$ is to be chosen, so that we can replace $u$ by any $v$ such that $\bar u=\bar v$.

\end{notation}

\begin{defn}[{\cite[Section~4]{fulton}}]

Let $u,v\in W$. We say that $\bar u$ and $\bar v$ are adjacent if they are related as in \cite[Lemma~4.1]{fulton}, i.e. if there exists an $\alpha\in R^+\setminus R_P^+$ such that $s_\alpha\bar{u}=\bar{v}$ or equivalent if there exists an $\alpha\in R^+\setminus R_P^+$ such that $u\bar s_\alpha=\bar v$. Note that adjacency is a symmetric and irreflexive relation on $W/W_P$.

We define a chain to be a sequence $\bar u_0,\ldots,\bar u_r$ such that $\bar u_{i-1}$ and $\bar u_i$ are adjacent for all $1\leq i\leq r$. We define a chain from $\bar u$ to $\bar v$ to be a chain $\bar u_0,\ldots,\bar u_r$ such that $\bar u\preceq\bar u_0$ and $\bar u_r\preceq\bar v^*$.

We define the degree of a chain $\bar u_0,\ldots,\bar u_r$ to be the sum $\sum_{i=1}^rd(\alpha_i)$ where $\alpha_i\in R^+\setminus R_P^+$ are roots such that $u_{i-1}\bar s_{\alpha_i}=\bar u_i$ for all $1\leq i\leq r$. By \cite[Lemma~3.1 and Lemma~4.1]{fulton} the degree of a chain is well-defined (does only depend on the residue classes $\bar u_0,\ldots,\bar u_r$ and not on the choice of representatives $u_0,\ldots,u_r$).

\end{defn}

\begin{defn}
\label{def:deltauv}

Let $u,v\in W$. Then we define $\delta_P(u,v)$ to be the set of all minimal elements of the set
$$
\left\{d\text{ a degree of a chain from }\bar u\text{ to }\bar v\right\}\,.
$$

\end{defn}

\begin{rem}
\label{rem:uniqueelement2}

Let $u,v\in W$. The nature of the sets $\delta_P(u,v)$ is similar to the nature of the distance function described in Remark~\ref{rem:uniqueelement}. Indeed, there is a close connection between the two which we will prove in Theorem~\ref{thm:descriptionofdelta}. For now, let us remark that by definition any two distinct elements of $\delta_P(u,v)$ are incomparable. Consequently, it follows that for all simple roots $\beta$ the set $\delta_{P_\beta}(u,v)$ consists of a unique element.

\end{rem}

\begin{conv}
\label{conv:uniqueelement2}

This convention is similar to Convention~\ref{conv:uniqueelement}.
Let $u,v\in W$. If $\delta_P(u,v)$ consists of a unique element $d$, we identify the set $\delta_P(u,v)$ with its unique element and write $\delta_P(u,v)=d$. In particular, for all simple roots $\beta$ we identify $\delta_{P_\beta}(u,v)$ with a non-negative integer (cf. Remark~\ref{rem:uniqueelement2}). 

\end{conv}

\begin{prop}[{\cite[Lemma~5.8]{thesis}}]
\label{prop:delta2}

Let $u,u',v,v',w,w'\in W$.

\begin{enumerate}

\item 
\label{item:invariant2}

The set $\delta_P(u,v)$ depends only on the class of $u$ and $v$ moduli $W_P$, i.e. we have $\delta_P(uw,vw')=\delta_P(u,v)$ for all $w,w'\in W_P$. 

\item
\label{item:commutative}

We have $\delta_P(u,v)=\delta_P(v,u)$.

\item
\label{item:zero}

$\delta(u,v)=0$ if and only if $\bar{u}\preceq\bar{v}^*$.

\item 
\label{item:usmallerv2}

Suppose that $\bar{u}\preceq\bar{u}'$ and that $\bar{v}\preceq\bar{v}'$. For all $d'\in\delta_P(u',v')$ there exists a degree $d\in\delta_P(u,v)$ such that $d\leq d'$.

\item
\label{item:u'smallerv'}

Let $d\in\delta_P(u,v)$. Let $\bar{u}_0,\ldots,\bar{u}_r$ be a chain from $\bar{u}$ to $\bar{v}$ of degree $d$. Suppose that $\bar{u}\preceq\bar{u}'\preceq\bar{u}_0$ and that $\bar{v}\preceq\bar{v}'\preceq\bar{u}_r^*$. Then we have $d\in\delta_P(u',v')$.

\end{enumerate}

\end{prop}

\begin{proof}

Item~(\ref{item:invariant2}) is immediately clear from the definition.

Ad Item~(\ref{item:commutative}). Note that if $\bar u$ and $\bar v$ are adjacent, then the translates $w\bar u$ and $w\bar v$ are also adjacent and the degree of the chain $\bar u,\bar v$ equals the degree of the chain $w\bar u,w\bar v$. Let $d\in\delta_P(u,v)$. Let $\bar u_0,\ldots,\bar u_r$ be a chain from $\bar u$ to $\bar v$ of degree $d$. Then $\bar u_r^*,\ldots,\bar u_0^*$ is a chain from $v$ to $u$ of degree $d$. Therefore there exists a $d'\in\delta_P(v,u)$ such that $d'\leq d$. By replacing $u$ and $v$, we also see that for all $d\in\delta_P(v,u)$ there exists a $d'\in\delta_P(u,v)$ such that $d'\leq d$. In view of Remark~\ref{rem:uniqueelement2}, this implies $\delta_P(u,v)=\delta_P(v,u)$.\footnote{A different way to see Item~(\ref{item:commutative}) is to use Theorem~\ref{thm:fulton} and the commutativity of $\star$.}

Item~(\ref{item:zero}) and Item~(\ref{item:usmallerv2}) are immediately clear from the definition.

Ad Item~(\ref{item:u'smallerv'}. It is clear that $\bar u_0,\ldots,\bar u_r$ is also a chain from $\bar u'$ to $\bar v'$. Therefore there exists a $d'\in\delta_P(u',v')$ such that $d'\leq d$. By Item~(\ref{item:usmallerv2}) it follows that there exists a $d''\in\delta_P(u,v)$ such that $d''\leq d'\leq d$. Remark~\ref{rem:uniqueelement2} implies that $d''=d'=d$ and thus $d\in\delta_P(u',v')$.
\end{proof}

\begin{prop}[{\cite[Lemma~6.7]{thesis}}]
\label{prop:equalu}

Let $u\in W$. Let $d\in\delta_P(u,w_o)$. Then there exists a chain $\bar u_0,\ldots,\bar u_r$ from $\bar u$ to $\bar w_o$ of degree $d$ such that $\bar u_0=\bar u$ and $\bar u_r=\bar 1$.

\end{prop}

\begin{proof}

Let $u\in W$ and $d\in\delta_P(u,w_o)$. By definition there exists a chain from $\bar u$ to $\bar w_o$ of degree $d$. By \cite[Theorem~9.1: $(2)\Leftrightarrow(11)$]{fulton} this means that
$$
\mathrm{ev}_1^{-1}(Y_u)\cap\mathrm{ev}_2^{-1}(X_1)\neq\emptyset\text{ in }\overline{M}_{0,3}(X,d)\,.
$$
By applying $\mathrm{ev}_1$ to this expression and using the definition of $\Gamma_d(X_1)$, we see that $Y_u\cap\Gamma_d(X_1)\neq\emptyset$. Since $Y_u\cap\Gamma_d(X_1)$ is not empty and $T$-stable, there exists a $v\in W$ such that $x(v)\in\ Y_u\cap\Gamma_d(X_1)$ or equivalently such that $\bar u\preceq\bar v\preceq\bar z_d^P$. This means in particular that there exists a point $(C,p_1,p_2,f)$ in $\overline{M}_{0,2}(X,d)$ such that $f(p_1)=x(u)$ and $f(p_2)=x(1)$.

Since $C$ is connected the image curve $f(C)$ is also connected. By definition it is clear that $[f(C)]\leq d$. By \cite[Lemma~5.11]{thesis} the image curve $f(C)$ converges to a $T$-invariant curve $C'$ which passes through the $T$-fixed points $x(u)$ and $x(1)$. We clearly have $[C']=[f(C)]$ and thus $[C']\leq d$. Since $f(C)$ is connected the limit $C'$ is also connected. Therefore we can choose a sequence $C_1,\ldots,C_r\subseteq C'$ of irreducible $T$-invariant curves such that $x(u)\in C_1$, $x(1)\in C_r$ and such that $C_{i-1}$ meets $C_i$ for all $1<i\leq r$. The sequence $C_1,\ldots,C_r$ determines a unique chain $\bar u_0,\ldots,\bar u_r$ from $\bar u$ to $\bar w_o$ of degree $\sum_{i=1}^r[C_i]\leq[C']\leq d$ such that $\bar u_0=\bar u$ and $\bar u_r=\bar 1$. Since $d\in\delta_P(u,w_o)$ is minimal, it follows that $\sum_{i=1}^r[C_i]=d$. This completes the proof.
\end{proof}

\todo[inline,color=green]{{\color{blue} In \cite[Section~4, page~8]{fulton} we can read: \enquote{A chain from $\bar u$ to $\bar v$ determines, and is determined by, a sequence of irreducible $T$-invariant curves $C_1,\ldots,C_r$ in $X$, each meeting the next, with $C_1$ meeting $Y_u$ and $C_r$ meeting $X_{v^*}$.} I do not really understand it. What happens if all curves $C_1,\ldots,C_r$ meet in a single point? Then they do not determine a chain. Is this behavior excluded by the terminology? In the proof of Proposition~\ref{prop:equalu} I can definitely exclude this \enquote{bad} behavior by choosing the sequence $C_1,\ldots,C_r\subseteq C'$ cleverly.}}

\begin{cor}
\label{cor:delta2}

Let $u,v\in W$. Let $d\in\delta_P(u,v)$. Let $\bar{u}_0,\ldots,\bar{u}_r$ be a chain from $\bar{u}$ to $\bar{v}$ of degree $d$. Then we have $d\in\delta_P(u_r^{-1}u_0,w_o)$

\end{cor}

\begin{proof}
Let $u,v,d,\bar u_0,\ldots,\bar u_r$ as in the statement. Then $u_r^{-1}\bar u_0,u_r^{-1}\bar u_1,\ldots,u_r^{-1}\bar u_{r-1},u_r^{-1}\bar u_r=\bar 1$ is a chain from $u_r^{-1}\bar u_0$ to $\bar w_o$ of degree $d$. Therefore there exists a $d'\in\delta_P(u_r^{-1}u_0,w_o)$ such that $d'\leq d$. By Proposition~\ref{prop:equalu} there exists a chain $\bar v_0,\ldots,\bar v_s$ from $u_r^{-1}\bar u_0$ to $\bar w_o$ of degree $d'$ such that $\bar v_0=u_r^{-1}\bar u_0$ and $\bar v_s=\bar 1$. Then $u_r\bar v_0=\bar u_0,u_r\bar v_1,\ldots,u_r\bar v_{s-1},u_r\bar v_s=\bar u_r$ is a chain from $\bar u_0$ to $\bar u_r^*$ of degree $d'$. Therefore there exists a $d''\in\delta_P(u_0,u_r^*)$ such that $d''\leq d'\leq d$. But by Proposition~\ref{prop:delta2}(\ref{item:u'smallerv'}) applied to $u'=u_0$ and $v'=u_r^*$ we know that $d\in\delta_P(u_0,u_r^*)$. In view of Remark~\ref{rem:uniqueelement2}, it follows that $d''=d'=d$ and thus $d\in\delta_P(u_r^{-1}u_0,w_o)$ as claimed.
\end{proof}

\begin{thm}
\label{thm:descriptionofdelta}

For all $u\in W$, we have $\delta_P(u)=\delta_P(u,w_o)$.

\end{thm}

\begin{proof}

Let $d\in\delta_P(u,w_o)$. In the first paragraph of the proof of Proposition~\ref{prop:equalu} we saw using \cite[Theorem~9.1]{fulton} that $\bar u\preceq\bar z_d^P$. This means that there exists a $d'\in\delta_P(u)$ such that $d'\leq d$. Now, let $d\in\delta_P(u)$. Then we have $\bar u\preceq\bar z_d^P$. By the second paragraph of the proof of Proposition~\ref{prop:equalu} we see that there exists a chain $\bar u_0,\ldots,\bar u_r$ from $\bar u$ to $\bar w_0$ of degree $d'\leq d$ such that $\bar u_0=\bar u$ and $\bar u_r=\bar 1$. This means that there exists a $d''\in\delta_P(u,w_o)$ such that $d''\leq d'\leq d$. In view of Remark~\ref{rem:uniqueelement} and Remark~\ref{rem:uniqueelement2}, it follows that $\delta_P(u)=\delta_P(u,w_o)$.
\end{proof}

\begin{thm}
\label{thm:delta2}

For all $u,v\in W$ and all $d\in\delta_P(u,v)$ we have $d\in\delta_P(z_d^P)$.

\end{thm}

\begin{proof}

Let $d\in\delta_P(u,v)$ for some $u,v\in W$. Let $\bar{u}_0,\ldots,\bar{u}_r$ be a chain from $\bar u$ to $\bar v$ of degree $d$. By Corollary~\ref{cor:delta2} and Theorem~\ref{thm:descriptionofdelta} we have $d\in\delta_P(u_r^{-1}u_0,w_o)=\delta_P(u_r^{-1}u_0)$. Proposition~\ref{prop:delta}(\ref{item:d}) implies that $d\in\delta_P(z_d^P)$.
\end{proof}

\begin{rem}

Let $u\in W$. In view of Theorem~\ref{thm:descriptionofdelta}, we have two equivalent definitions of the set $\delta_P(u)$. Firstly, $\delta_P(u)$ is the set of all minimal elements of the set
$$
\left\{d\text{ a degree such that }\bar u\preceq\bar z_d^P\right\}\,.
$$
Secondly, $\delta_P(u)$ is the set of all minimal elements of the set
$$
\left\{d\text{ a degree of a chain from }\bar u\text{ to }\bar w_o\right\}\,.
$$
From now on, we use these definitions of $\delta_P(u)$ interchangeably without constantly referring to Theorem~\ref{thm:descriptionofdelta}.

\end{rem}

\begin{prop}[{\cite[Corollary~6.11]{thesis}}]
\label{prop:cor6.11}

Let $u\in W$. Let $d\in\delta_P(u)$. Let $\bar u_0,\ldots,\bar u_r$ be a chain from $\bar u$ to $\bar w_o$ of degree $d$. Let $\alpha_1,\ldots,\alpha_r\in R^+\setminus R_P^+$ be roots such that $u_{i-1}\bar s_{\alpha_i}=\bar u_i$ for all $1\leq i\leq r$. Then we have $\sum_{j=i+1}^rd(\alpha_j)\in\delta_P(u_i)$ for all $0\leq i\leq r$.

\end{prop}

\begin{proof}

Let $0\leq i\leq r$. Let $d_i=\sum_{j=i+1}^r d(\alpha_j)$ for short. Then $\bar u_i,\ldots,\bar u_r$ is a chain from $\bar u_i$ to $\bar w_o$ of degree $d_i$. Thus there exists a $d_i'\in\delta_P(u_i)$ such that $d_i'\leq d_i$. By Proposition~\ref{prop:equalu} there exists a chain $\bar v_i,\ldots,\bar v_s$ from $\bar u_i$ to $\bar w_o$ of degree $d_i'$ such that $\bar v_i=\bar u_i$ and $\bar v_s=\bar 1$. Then $\bar u_0,\ldots,\bar u_i=\bar v_i,\ldots,\bar v_s$ is a chain from $\bar u$ to $\bar w_o$ of degree $d-d_i+d_i'\leq d$. Since $d\in\delta_P(u)$ is minimal, it follows that $d-d_i+d_i'=d$ and thus $d_i=d_i'$ and thus $d_i\in\delta_P(u_i)$.
\end{proof}

\begin{cor}[{\cite[Corollary~6.11]{thesis}}]
\label{cor:cor6.11}

Let $u\in W$. Let $d\in\delta_P(u)$. Let $\bar u_0,\ldots,\bar u_r$ be a chain from $\bar u$ to $\bar w_o$ of degree $d$. Then we have $\bar u_0\succ\bar u_1\succ\cdots\succ\bar u_{r-1}\succ\bar u_r$.

\end{cor}

\begin{proof}

Let $u,v\in W$. Note that if $\bar u$ and $\bar v$ are adjacent, then they are comparable in the Bruhat order, i.e. we have $\bar u\prec\bar v$ or $\bar v\prec\bar u$.

Let $u,d,\bar u_0,\ldots,\bar u_r$ be as in the statement. Let $\alpha_1,\ldots,\alpha_r\in R^+\setminus R_P^+$ be roots as in the statement of Proposition~\ref{prop:cor6.11}. Let $1\leq i\leq r$. By the previous comment, we know that $\bar u_{i-1}$ and $\bar u_i$ are comparable in the Bruhat order. Suppose for a contradiction that $\bar u_{i-1}\prec\bar u_i$. By Proposition~\ref{prop:delta}(\ref{item:usmallerv}) and Proposition~\ref{prop:cor6.11} there exists a $d\in\delta_P(u_{i-1})$ such that 
$$
d\leq\sum_{j=i+1}^r d(\alpha_j)<\sum_{j=i}^r d(\alpha_j)\,.
$$
Since $\sum_{j=i}^r d(\alpha_j)\in\delta_P(u_{i-1})$ by Proposition~\ref{prop:cor6.11}, this contradicts Remark~\ref{rem:uniqueelement}. Therefore we conclude that $\bar u_{i-1}\succ\bar u_i$. Since $1\leq i\leq r$ was arbitrary, this completes the proof.
\end{proof}

\todo[inline,color=green]{{\color{blue} Do you know a simple argument for this? Let $u,v\in W$ and $\alpha\in R$ such that $u=vs_\alpha$. Then $\ell(u)\neq\ell(v)$. I require this statement to conclude that any two adjacent elements are comparable in the Bruhat order. I am quite sure that this is true, but the way I can prove it seems unnecessarily involved to me.}

Answered on MSE, cf. \url{http://math.stackexchange.com/questions/1957028/are-adjacent-elements-always-comparable-in-the-bruhat-order}.}

\begin{defn}
\label{def:minimal}

Let $d$ be a degree. Let $u,v\in W$. We say that $d$ is a minimal degree in $\sigma_u\star\sigma_v$ if the power $q^d$ appears with non-zero coefficient in the expression $\sigma_u\star\sigma_v$ and if $d$ is minimal with this property, i.e. for all $d'<d$ the power $q^{d'}$ does not occur in the expression $\sigma_u\star\sigma_v$.

\end{defn}

\begin{thm}[{\cite[Theorem~9.1]{fulton}}]
\label{thm:fulton}

Let $u,v,w\in W$. We have the following identities
\begin{align*}
\delta_P(u,v)=&\left\{d\text{ a minimal degree in }\sigma_u\star\sigma_v\right\}\\
\delta_P(w)=&\left\{d\text{ a minimal degree in }\sigma_w\star\mathrm{pt}\right\}
\end{align*}
which interpret the sets $\delta_P(u,v)$ and $\delta_P(w)$ in terms of minimal degrees in quantum products.

\end{thm}

\begin{proof}

The first identity follows directly from the definitions and \cite[Theorem~9.1: $(1)\Leftrightarrow(2)$]{fulton}. The second identity follows from the first and Theorem~\ref{thm:descriptionofdelta}.
\end{proof}

\todo[inline,color=green]{Theorem: $d\in\delta_P(u,v)$ then $d\in\delta_P(z_d^P)$. \cite[Lemma~6.7]{thesis} as Proposition before the Corollary. After this, I relate to quantum cohomology (theorem for both $\delta(u,v)$ and $\delta(u)$), give the necessary definitions. After the discussion of the definition in terms of chains. It is also obvious because $\star$ is commutative.

I suppressed comments on: translation invariance of adjacent elements with invariance of the corresponding degree (now part of the proof of Proposition~\ref{prop:delta2}(\ref{item:commutative}), the relation between $T$-invariant curves and chains (now part of the proof of Proposition~\ref{prop:equalu}), the fact that any two adjacent elements are comparable in the Bruhat order (now part of the proof of Corollary~\ref{cor:cor6.11}, cf. notes on October 7).}

\section{Local curve neighborhoods}
\label{sec:local}

In this section, we introduce for each positive root $\varphi$ a homogeneous subspace $X(\varphi)$ of $X$. It turns out that $X(\varphi)$ is actually a Schubert variety in $X$ parameterized by some Weyl group element $w_o(\varphi)$ which in addition carries the structure of a homogeneous space under the action of a subgroup $G(\varphi)$ of $G$. All concepts discussed so far (curve neighborhoods, distance function, etc.) have a local version with respect to the situation attached to $X(\varphi)$. We will be mostly concerned with various compatibility results between local and global notions (Fact~\ref{fact:compatibility} and Theorem~\ref{thm:compatibility}). In fact, the compatibility between local and global distance functions is the main result we aim for in this section (cf. Theorem~\ref{thm:compatibility}). The local point of view proves to be fertile for the proof of many of our results (see for example Theorem~\ref{thm:thesis_main}) although this is not directly visible in the statements. For example, in the proof of Theorem~\ref{thm:main}, we will preform a strong induction which incorporates $X(\varphi)$ for all $\varphi\in R^+$ in its induction hypothesis (cf. statement $(\mathrm{H}_n)$ 
in the beginning of Section~\ref{sec:main}). The philosophy behind the technique is that all required information about a connected degree $e\in H_2(G/B)$ is already contained in the situation attached to $G(\varphi)/B(\varphi)$ where $\varphi$ is a positive root such that $\Delta(e)=\Delta(\varphi)$ (e.g. $\varphi=\alpha(e)$, c.f. Fact~\ref{fact:local}).

\begin{notation}

Let $\varphi$ be a positive root. We denote by $R(\varphi)$ the root subsystem of $R$ which has as set of simple roots the set $\Delta(\varphi)$. In other words, $R(\varphi)$ is the root subsystem of $R$ generated by $\Delta(\varphi)$. Since $\Delta(\varphi)$ is a connected subset of the Dynkin diagram, the root system $R(\varphi)$ is always irreducible. We denote by $R(\varphi)^+$ the positive roots of $R(\varphi)$ associated to $\Delta(\varphi)$. Equivalently, we can define $R(\varphi)^+=R(\varphi)\cap R^+$. We denote by $G(\varphi)$ the simple (connected, simply connected, complex) linear algebraic subgroup of $G$ which has $R(\varphi)$ as root system. Furthermore, we set $B(\varphi)=G(\varphi)\cap B$ and $P(\varphi)=G(\varphi)\cap P$. By definition, it is clear that $B(\varphi)$ is a Borel subgroup of $G(\varphi)$ and that $P(\varphi)$ is a parabolic subgroup of $G(\varphi)$ (relative to $B(\varphi)$). It is clear that $P(\varphi)$ is the parabolic subgroup of $G(\varphi)$ which corresponds to the set of simple roots $\Delta(\varphi)\cap\Delta_P$. 
To this situation, we can attach a root system $R(\varphi)_{P(\varphi)}$ with positive roots $R(\varphi)_{P(\varphi)}^+$ analogously as we attached $R_P$ and $R_P^+$ to the situation in Subsection~\ref{subsec:roots}.
%
We denote by $W_{G(\varphi)}$ the Weyl group associated to $G(\varphi)$ and $G(\varphi)\cap T$ and by $W_{P(\varphi)}$ the Weyl group associated to $P(\varphi)$ and $G(\varphi)\cap T$. Equivalently, we can define $W_{G(\varphi)}$ as the parabolic subgroup of $W$ generated by the simple reflections $s_\beta$ where $\beta\in\Delta(\varphi)$ and $W_{P(\varphi)}$ as the parabolic subgroup of $W$ (or $W_{G(\varphi)}$) generated by the simple reflections $s_\beta$ where $\beta\in\Delta(\varphi)\cap\Delta_P$. We denote by $W_{G(\varphi)}^{P(\varphi)}$ the set of all minimal representatives of cosets in $W_{G(\varphi)}/W_{P(\varphi)}$. We denote by $w_o(\varphi)$ the longest element of $W_{G(\varphi)}$. 


\end{notation}

\begin{notation}

Let $\varphi$ be a positive root. We set $X(\varphi)=G(\varphi)/P(\varphi)$. We have a natural inclusion $X(\varphi)\subseteq X$. In the same way we associated in Subsection~\ref{subsec:schubert} a Schubert variety in $X$ to each element of $W$, we can associate to each element $w\in W_{G(\varphi)}$ a Schubert variety in $X(\varphi)$ denoted by $X(\varphi)_w$. Each Schubert variety $X(\varphi)_w$ is a closed, irreducible, $B(\varphi)$-stable subvariety of $X(\varphi)$. Note that the notation $X(\varphi)_w$ is only of temporary nature since we will see in Fact~\ref{fact:compatibility}(\ref{item:comp_schubert}) that $X(\varphi)_w$ can be identified with a Schubert variety in $X$ (parameterized by the very same element $w$).


\end{notation}

\begin{conv}

Let $\varphi$ be a positive roots. The natural inclusion $X(\varphi)\subseteq X$ induces natural inclusions
$$
H_2(X(\varphi))\subseteq H_2(X)\text{ and }H^2(X(\varphi))\subseteq H^2(X)\,.
$$
In addition to the \enquote{degree conventions} in Subsection~\ref{subsec:degrees}, we identify $H_2(X(\varphi))$ always with a sublattice of $H_2(X)$, i.e. we consider degrees in $H_2(X(\varphi))$ interchangeably as degrees in $H_2(X(\varphi))$ and $H_2(X)$. Similarly, we identify $H^2(X(\varphi))$ with a sublattice of $H^2(X)$.


\end{conv}

\begin{defn}[{\cite[Section~1]{kostant}}]

Let $\varphi\in R^+$. We say $\varphi$ is locally high if $\varphi$ is the highest root of $R(\varphi)$.

\end{defn}


{\color{black} Let $\varphi$ be a positive root. While the compatibility of the distance function with the homogeneous subspace $X(\varphi)$ is the main issue of this section (cf. Theorem~\ref{thm:compatibility}), many other compatibility statements are obvious. We gather them in the following fact.}

\begin{fact}
\label{fact:compatibility}

Let $\varphi\in R^+$. Let $u,v,w\in W_{G(\varphi)}$. Let us temporarily denote the Bruhat order on $W_{G(\varphi)}$ by $\preceq_\varphi$ and the Hecke product on $W_{G(\varphi)}$ by $\cdot_\varphi$.

\begin{enumerate}

\item 
\label{item:comp_tfixed}

We have a natural inclusion $W_{G(\varphi)}/W_{P(\varphi)}\subseteq W/W_P$ which means that $uW_{P(\varphi)}=vW_{P(\varphi)}$ if and only if $uW_P=vW_P$.

\item
\label{item:comp_reduced}

The reduced expressions of $w$ in $W_{G(\varphi)}$ and the reduced expressions of $w$ in $W$ coincide. In particular, the length function on $W_{G(\varphi)}$ is the restriction of the length function on $W$.

\item
\label{item:comp_schubert}

We have $X(\varphi)_w=X_w$, in particular $X(\varphi)=X_{w_o(\varphi)}$.

\item
\label{item:comp_length}

We have $\ell\left(wW_{P(\varphi)}\right)=\ell\left(wW_P\right)$ and thus $W_{G(\varphi)}^{P(\varphi)}\subseteq W^P$

\item
\label{item:comp_bruhat}

We have $uW_{P(\varphi)}\preceq_\varphi vW_{P(\varphi)}$ if and only if $uW_P\preceq vW_P$, in particular $u\preceq_\varphi v$ if and only if $u\preceq v$.

\item
\label{item:comp_hecke}

We have $u\cdot_\varphi v=u\cdot v$.

\item
\label{item:comp_chain}

Suppose that $\bar u$ and $\bar v$ are adjacent. Let $\alpha\in R^+\setminus R_P^+$ such that $u\bar s_\alpha=\bar v$. Then we have $\alpha\in R(\varphi)^+\setminus R(\varphi)_{P(\varphi)}^+$.

\end{enumerate}

\end{fact}

\begin{proof}

We have a natural inclusion $X(\varphi)\subseteq X$ which yields an inclusion $X(\varphi)^T\subseteq X^T$ of $T$-fixed points. Whence, Item~(\ref{item:comp_tfixed}) follows. It is also very easy to see this directly.

Item~(\ref{item:comp_reduced}) is discussed in \cite[5.5, Theorem~(b)]{humphreys3}. 

Item~(\ref{item:comp_schubert}) is very easy to prove. The reader finds complete details in \cite[Lemma~4.8]{thesis}.

Item~(\ref{item:comp_length}) and (\ref{item:comp_bruhat}) follow immediately from Item~(\ref{item:comp_schubert}). A purely combinatorial proof that $u\preceq_\varphi v$ if and only if $u\preceq v$ can also be found in \cite[5.10, Corollary~(a)]{humphreys3}. Item~(\ref{item:comp_length}) in the case that $P=B$ follows also from Item~(\ref{item:comp_reduced}).

Item~(\ref{item:comp_hecke}) follows from Item~\ref{item:comp_reduced}, (\ref{item:comp_bruhat}) and the definition of the Hecke product.

Ad Item~(\ref{item:comp_chain}). Let $u\bar s_\alpha=\bar v$ as in the statement. By considering the expression $\bar s_\alpha =u^{-1}\bar v$, we may assume without loss of generality that $u=1$. Suppose for a contradiction that $C_\alpha\cap(X\setminus X(\varphi))$ is not empty. Since $C_\alpha\cap(X\setminus X(\varphi))$ is $T$-stable, then there exists a $T$-fixed point in $C_\alpha$ which is not contained in $X(\varphi)$. Since $C_\alpha$ contains only the $T$-fixed points $x(1)$ and $x(s_\alpha)$ and since clearly $x(1)\in X(\varphi)$, it follows that $x(s_\alpha)\notin X(\varphi)$. This contradicts the fact that $\bar s_\alpha=\bar v$ where $v\in W_{G(\varphi)}$. Therefore it follows that $C_\alpha\subseteq X(\varphi)$. In other words, $C_\alpha$ is an irreducible $T$-invariant curve in $X(\varphi)$ containing $x(1)$. This means that there exists a root $\alpha'\in R(\varphi)^+\setminus R(\varphi)_{P(\varphi)}^+$ such that $C_\alpha=C_{\alpha'}$. This immediately implies $x(s_\alpha)=x(s_{\alpha'})$ by passing to the $T$-fixed points of $C_\alpha$ and $C_{\alpha'}$ respectively. This means $\bar s_\alpha=\bar s_{\alpha'}$ and thus by \cite[Lemma~4.1, Claim]{fulton} that $\alpha=\alpha'$. This proves Item~(\ref{item:comp_chain}).   
\end{proof}

\begin{rem}

Let $\varphi\in R^+$. Fact~\ref{fact:compatibility}(\ref{item:comp_bruhat}), (\ref{item:comp_hecke}) show that the notations $\preceq_\varphi$ and $\cdot_\varphi$ are superfluous since the Bruhat order and the Hecke product on $W_{G(\varphi)}$ are compatible with the Bruhat order and the Hecke product on $W$. Therefore we will not use these notations further on and simply identify both.

\end{rem}

\begin{defn}
\label{def:local}

Let $\varphi\in R^+$. Let $e\in H_2(X(\varphi))$ be a degree. Let $w\in W_{G(\varphi)}$. In Definition~\ref{def:zdp}, we defined for a degree $d$ an element $z_d^P\in W^P$ relative to the situation attached to $X$. In the same way, we can define an element $z_e^{P(\varphi)}\in W_{G(\varphi)}^{P(\varphi)}$ relative to the situation attached to $X(\varphi)$, i.e. we consider the greedy decomposition $(\alpha_1,\ldots,\alpha_r)$ of $e$ in $R(\varphi)$ and define $z_e^{P(\varphi)}$ to be the minimal representative in $s_{\alpha_1}\cdot\ldots\cdot s_{\alpha_r}W_{P(\varphi)}$ (cf. Fact~\ref{fact:compatibility}(\ref{item:comp_hecke})). By \cite[Theorem~5.1]{curvenbhd2} applied to $X(\varphi)$, the element $z_e^{P(\varphi)}$ is the minimal representative parameterizing the degree $e$ curve neighborhood of $\{x(1)\}$ in $X(\varphi)$. Moreover, by \cite[Theorem~5.1]{curvenbhd2} and Fact~\ref{fact:compatibility}(\ref{item:comp_schubert}), (\ref{item:comp_hecke}), the degree $e$ curve neighborhood of $X_w$ in $X(\varphi)$ is given by $X_{w\cdot z_e^{P(\varphi)}}$. We call $X_{w\cdot z_e^{P(\varphi)}}$ the local (degree $e$) curve neighborhood of $X_w$ in $X(\varphi)$.

\end{defn}

\begin{rem}

Let $\varphi\in R^+$. Let $u,v,w\in W_{G(\varphi)}$. Similar to Definition~\ref{def:local} of local curve neighborhoods and the element $z_e^{P(\varphi)}$ for a degree $e\in H_2(X(\varphi))$, we can \enquote{localize} other notions. Each time, we define a local notion with respect to the situation attached to $X(\varphi)$ as we defined the original notion with respect to the situation attached to $X$. In particular, we can and will speak about the local distance function, i.e. the sets $\delta_{P(\varphi)}(w)$ and $\delta_{P(\varphi)}(u,v)$, chains from $uW_{P(\varphi)}$ to $vW_{P(\varphi)}$ in $X(\varphi)$, $P(\varphi)$-cosmall and very $P(\varphi)$-cosmall roots. All our previous results carry over to a version relative to $X(\varphi)$ which we will freely use from now on. The reader is invited to write down the precise definitions on his own.

\end{rem}


\begin{fact}
\label{fact:local}

Let $e\in H_2(G/B)$ be a degree. Let $(\alpha_1,\ldots,\alpha_r)$ be a greedy decomposition of $e$. Let $\varphi$ be a positive root such that $\alpha_i\leq\varphi$ for all $1\leq i\leq r$. Then we have $z_e^B=z_e^{B(\varphi)}$. In particular, if $e$ is a connected degree, then we have $z_e^B=z_e^{B(\alpha(e))}$. 

\end{fact}

\begin{proof}

This follows directly from Fact~\ref{fact:compatibility}(\ref{item:comp_hecke}) (compatibility of the Hecke product), since the greedy decomposition $(\alpha_1,\ldots,\alpha_r)$ of $e$ in $R$ is also a greedy decomposition of $e$ in $R(\varphi)$. If $e$ is a connected degree, then we have $\alpha_i\leq\alpha(e)$ for all $1\leq i\leq r$ by Proposition~\ref{prop:connecteddegree}. Therefore, the very last statement is a special case of the first statement.
\end{proof}

\begin{thm}[{\cite[Lemma~6.60]{thesis}}]
\label{thm:compatibility}

For all $\varphi\in R^+$ and all $u\in W_{G(\varphi)}$ we have $\delta_P(u)=\delta_{P(\varphi)}(u)$, in particular $\delta_P(u)\subseteq H_2(X(\varphi))$.

\end{thm}

\begin{proof}

Let $\varphi\in R^+$ and $u\in W_{G(\varphi)}$. By definition and Fact~\ref{fact:compatibility}(\ref{item:comp_tfixed}) every chain from $uW_{P(\varphi)}$ to $w_o(\varphi)W_{P(\varphi)}$ in $X(\varphi)$ is also a chain from $\bar u$ to $\bar w_o$ in $X$ simply by applying the natural inclusion $W_{G(\varphi)}/W_{P(\varphi)}\subseteq W/W_P$ to its members. This shows that $\delta_P(u)$ is the set of all minimal elements of a larger set than $\delta_{P(\varphi)}(u)$. Consequently, for all $d\in\delta_{P(\varphi)}(u)$ there exists a $d'\in\delta_P(u)$ such that $d'\leq d$.

Conversely, let $d\in\delta_P(u)$. By Proposition~\ref{prop:equalu} there exists a chain $\bar u_0,\ldots,\bar u_r$ from $\bar u$ to $\bar w_o$ of degree $d$ such that $\bar u_0=\bar u$ and $\bar u_r=\bar 1$. Let $0\leq i\leq r$. By Corollary~\ref{cor:cor6.11} we have $\bar u_i\preceq\bar u\preceq\bar w_o(\varphi)$. By Fact~\ref{fact:compatibility}(\ref{item:comp_schubert}) this means that $x(u_i)\in X(\varphi)$ and thus $\bar u_i\in W_{G(\varphi)}/W_{P(\varphi)}$. By replacing $u_i$ by another representative of $\bar u_i$, we may assume that $u_i\in W_{G(\varphi)}$ for all $0\leq i\leq r$. By Fact~\ref{fact:compatibility}(\ref{item:comp_chain}) we therefore see that $u_0W_{P(\varphi)},\ldots,u_rW_{P(\varphi)}$ is a chain from $uW_{P(\varphi)}$ to $w_o(\varphi)W_{P(\varphi)}$ in $X(\varphi)$ of degree $d\in H_2(X(\varphi))$. This means that there exists a $d'\in\delta_{P(\varphi)}(u)$ such that $d'\leq d$. All in all, this proves the theorem (cf. Remark~\ref{rem:uniqueelement}).
\end{proof}

\begin{cor}
\label{cor:locallyhigh}

Let $\alpha$ be a locally high root. Then we have $\delta_P(s_\alpha)=d(\alpha)$.

\end{cor}

\begin{proof}

Let $\alpha$ be a locally high root. Then $\alpha$ is in particular very $P(\alpha)$-cosmall. Proposition~\ref{prop:delta}(\ref{item:verycosmall}) says that $\delta_{P(\alpha)}(s_\alpha)=d(\alpha)$. Theorem~\ref{thm:compatibility} then implies $\delta_P(s_\alpha)=d(\alpha)$.
\end{proof}

\begin{rem}

We actually gathered all the ingredients for a more refined result than Theorem~\ref{thm:compatibility}: the compatibility of local and global curve neighborhoods. Since this result is of no importance for the exposition, we only mention the statement. The reader finds a complete proof in \cite[Lemma~6.58]{thesis}. Based on the techniques developed so far, it is not more difficult than and similar to the proof of Theorem~\ref{thm:compatibility}.
Let $\varphi\in R^+$, $d\in H_2(X(\varphi))$ be a degree and $w\in W_{G(\varphi)}$. Then it was proved in \cite[Lemma~6.58]{thesis} that we have
$$
X_{w\cdot z_d^{P(\varphi)}}=X_{w\cdot z_d^P}\cap X(\varphi)\text{, in particular }X_{z_d^{P(\varphi)}}=X_{z_d^P}\cap X(\varphi)\,.
$$

\end{rem}

\todo[inline,color=green]{In this section, I have to include the definition of locally high roots. Any locally high root $\varphi$ satisfies $\delta_P(s_\varphi)=d(\varphi)$.}

\todo[inline,color=green]{Equality of curve neighborhoods $z_d^B=z_d^{B(\alpha)}$ for appropriate $\alpha$ covering all elements in the greedy decomposition of $d$. Needed for the reduction in the proof of Theorem~\ref{thm:main}.

I have to say (the least possible) about the compatibility issues, also (not only to establish the previous point) in view of the proof of
$$
d_{G/P_\beta}=\sum_{\alpha\in C_R(\beta)}(\omega_\beta,\alpha^\vee)\,.
$$

It seems that \cite[Lemma~6.58]{thesis} is strictly speaking superfluous to prove $\delta=\delta_\varphi$. We can use \cite[Corollary~6.11]{thesis}.

Compatibility: Hecke, Schubert, $W/W_P$ ($u,v\in W(\varphi)$: $uW_P=vW_P$ if and only if $uW_{P(\varphi)}=vW_{P(\varphi)}$ -- the inclusion $X(\varphi)\subseteq X$ already shows this), $W^P$, length, adjacent, chain (follows from $W/W_P$-inclusion) in $R(\varphi)$ if and only if $\alpha_i\in R(\varphi)$ for all $i$ (pencil note at \cite[Lemma~6.59]{thesis}).}

\todo[inline,color=green]{In a remark I will mention that $\Gamma^\varphi=X(\varphi)\cap\Gamma$ with reference to the thesis, after the proof of $\delta_\varphi=\delta$. But I will not include the proof of this in this paper. For a more refined statement concerning the compatibility of curve neighborhoods in general, we refer to\ldots I cannot add this remark yet, because I am missing the notation for local curve neighborhoods (October 10).

Note that we always identify $H_2(X(\varphi))\subseteq H_2(X)$ and also for $H^2$. Why is this possible?}

\section{The cascade of orthogonal roots}
\label{sec:cascade}

In this section, we introduce the cascade of orthogonal roots $\mathcal{B}_R$. This notion and its basic properties are due to Kostant \cite{kostant}. The set $\mathcal{B}_R$ is a special maximal set of strongly orthogonal roots (\cite[Theorem~1.8]{kostant}). In addition, all roots in $\mathcal{B}_R$ are locally high (\cite[Proposition~1.4]{kostant}). We recall the precise definitions of all required notions in this section. We use them to complete the proof of Theorem~\ref{thm:uniqueness} by proving Theorem~\ref{thm:curve}. Moreover, we give a simple formula which expresses $d_X$ in terms of the cascade of orthogonal roots: Namely, in Corollary~\ref{cor:identity} we will prove that
$$
d_X=\sum_{\alpha\in\mathcal{B}_R}d(\alpha)\,.
$$

\begin{notation}

Let $\theta_1$ be the highest root of $R$. We denote by $R^\circ$ the root subsystem of $R$ which consists of all roots $\alpha\in R$ such that $(\alpha,\theta_1)=0$.\footnote{Note that $R^\circ$ might be reducible or even empty. It is easy to see that $R^\circ=\emptyset$ if and only if $R$ is of type $\mathsf{A}_1$ or $\mathsf{A}_2$ (cf. Table~\ref{table:exceptional}).} We denote by $\Delta^\circ$ the set of simple roots $\beta\in\Delta$ such that $(\beta,\theta_1)=0$.

\end{notation}

\begin{defn}[{\cite[VI, 1, 3]{bourbaki_roots}}]

We say that two roots $\alpha$ and $\alpha'$ are strongly orthogonal if and only if $\alpha\pm\alpha'\notin R\cup\{0\}$.

\end{defn}

\begin{lem}
\label{lem:stronglyorthogonal}

Two roots $\alpha$ and $\alpha'$ are strongly orthogonal if and only if $(\alpha,\alpha')=0$ and $\alpha+\alpha'\notin R$ if and only if $(\alpha,\alpha')=0$ and $\alpha-\alpha'\notin R$.

\end{lem}

\begin{proof}

Suppose that $\alpha$ and $\alpha'$ are strongly orthogonal. By assumption, $\alpha$ and $\alpha'$ are non-proportional roots. Therefore we can look at the $\alpha$-string through $\alpha'$.\footnote{For background information on this notion, we refer to \cite[VI, 1, 3]{bourbaki_roots} or to \cite[9.4]{humphreys2}.}
Let $p$ and $q$ be non-negative integers such that $\alpha'-p\alpha,\ldots,\alpha'+q\alpha$ is the $\alpha$-string through $\alpha'$. By assumption, it follows that $p=q=0$ and thus $p-q=(\alpha',\alpha^\vee)=0$. This means $(\alpha,\alpha')=0$. This proves one implication.

Suppose that $\alpha$ and $\alpha'$ are orthogonal. It is clear that $\alpha+\alpha'\notin R$ if and only if $\alpha-\alpha'\notin R$ since $s_{\alpha'}(\alpha+\alpha')=\alpha-\alpha'$. This proves the other implication.
\end{proof}

\begin{defn}[{\cite[Section~1]{kostant}}]

Two subsets of roots $S$ and $S'$ are called totally disjoint if every element of $S$ is strongly orthogonal to every element of $S'$.

\end{defn}

\begin{fact}
\label{fact:disjoint}

Let $\varphi_1,\ldots,\varphi_k$ be locally high roots such that $\Delta(\varphi_1),\ldots,\Delta(\varphi_k)$ are the distinct connected components of $\Delta^\circ$. Then $R(\varphi_1),\ldots,R(\varphi_k)$ are the distinct irreducible components of $R^\circ$. In particular, $\Delta^\circ$ is a set of simple roots of $R^\circ$. Moreover, the set of roots $R(\varphi_1),\ldots,R(\varphi_k)$ are pairwise totally disjoint. 

\end{fact}

\begin{proof}

It is obvious that for distinct indices $1\leq i\neq j\leq k$ every element of $R(\varphi_i)$ is orthogonal to every element of $R(\varphi_j)$. To prove that $R(\varphi_1),\ldots,R(\varphi_k)$ are the distinct irreducible components of $R^\circ$, it suffices to prove that $R^\circ=\bigcup_{i=1}^k R(\varphi_i)$. The inclusion \enquote{$\supseteq$} is obvious. Let $\alpha\in R^\circ$. Since $\theta_1$ is the highest root, we must have $(\beta,\theta_1)\geq 0$ for all $\beta\in\Delta$. Therefore the orthogonality relation $(\alpha,\theta_1)=0$ leads by bilinearity to the inclusion $\Delta(\alpha)\subseteq\Delta^\circ$. Since $\Delta(\alpha)$ is connected, this means $\Delta(\alpha)\subseteq\Delta(\varphi_j)$ for some $1\leq j\leq k$ and thus $\alpha\in R(\varphi_j)$. It is clear from the decomposition of $R^\circ$ into irreducible components that $\Delta^\circ$ is a set of simple roots of $R^\circ$. Let $\alpha\in R(\varphi_i)$ and $\alpha'\in R(\varphi_j)$ for distinct indices $1\leq i\neq j\leq k$. To see that $\alpha$ and $\alpha'$ are strongly orthogonal, it suffices to prove $\alpha+\alpha'\notin R$ (Lemma~\ref{lem:stronglyorthogonal}). Suppose for a contradiction that $\alpha+\alpha'\in R$. Then we also have $\alpha+\alpha'\in R^\circ$ and thus $\alpha+\alpha'$ must be contained either in $R(\varphi_i)$ or in $R(\varphi_j)$ which is absurd since $\alpha+\alpha'$ contains elements from $\Delta(\varphi_i)$ and $\Delta(\varphi_j)$ with non-zero coefficient in an expression as integral linear combination of simple roots.
\end{proof}

\begin{defn}[{\cite[Section~1]{kostant}}]

Let $\theta_1$ be the highest root. Let $\varphi_1,\ldots,\varphi_k$ be locally high roots such that $\Delta(\varphi_1),\ldots,\Delta(\varphi_k)$ are the distinct connected components of $\Delta^\circ$. Then we define recursively a set of roots $\mathcal{B}_R$ by the assignment 
\begin{equation}
\label{eq:union}
\mathcal{B}_R=\{\theta_1\}\cup\bigcup_{i=1}^k\mathcal{B}_{R(\varphi_i)}=\{\theta_1,\varphi_1,\ldots,\varphi_k,\ldots\}
\end{equation}
This is well-defined since the formula expresses $\mathcal{B}_R$ in terms of strictly smaller root subsystems of $R$. By convention, we set $\mathcal{B}_\emptyset=\emptyset$. We call $\mathcal{B}_R$ the cascade of orthogonal roots. We call $C\subseteq\mathcal{B}_R$ a chain cascade if there exists a positive root $\varphi$ such that
$
C=\{\alpha\in\mathcal{B}_R\mid\alpha\geq\varphi\}
$.
In this case we write $C=C_R(\varphi)$.

\end{defn}

\begin{rem}
\label{rem:disjoint}

Note that the union in Equation~(\ref{eq:union}) is disjoint. Indeed, by 
Fact~\ref{fact:disjoint} the sets $R(\varphi_1),\ldots,R(\varphi_k)$ are totally disjoint, in particular the sets $\mathcal{B}_{R(\varphi_1)},\ldots,\mathcal{B}_{R(\varphi_k)}$ are totally disjoint since $\mathcal{B}_{R(\varphi_i)}\subseteq R(\varphi_i)$ for all $1\leq i\leq k$. Moreover, $\theta_1\notin R(\varphi_i)$, since $\varphi_i\neq\theta_1$, since $(\varphi_i,\theta_1)=0$ for all $1\leq i\leq k$.

\end{rem}

\begin{prop}[{\cite[Section~1]{kostant}}]
\label{prop:cascade}

\leavevmode

\begin{enumerate}

\item
\label{item:totallyordered}

Any chain cascade is totally ordered.

\item 
\label{item:locallyhigh}

All elements of $\mathcal{B}_R$ are locally high.

\item
\label{item:stronglyorthogonal}

Two distinct elements of $\mathcal{B}_R$ are strongly orthogonal.

\item
\label{item:totallydisjoint}

Let $\alpha$ and $\alpha'$ be two elements of $\mathcal{B}_R$ such that there exists no chain cascade which contains both $\alpha$ and $\alpha'$. Then there exists a positive root $\varphi\in\mathcal{B}_R$ such that $\alpha$ and $\alpha'$ belong to different irreducible components of $R(\varphi)^\circ$. In particular, $R(\alpha)$ and $R(\alpha')$ are totally disjoint.

\end{enumerate}

\end{prop}

\begin{proof}

Item~(\ref{item:totallyordered}), (\ref{item:locallyhigh}), (\ref{item:stronglyorthogonal}) and (\ref{item:totallydisjoint}) correspond respectively to \cite[Remark~1.3, Proposition~1.4, Lemma~1.6, Proposition~1.7]{kostant}. The reader may also consult \cite[Fact~4.5, Fact~4.6]{thesis} for a more detailed account.
\end{proof}

\begin{prop}[{\cite[Proposition~1.10]{kostant}}]
\label{prop:w_o}

For any positive root $\varphi$ we have
$$
w_o(\varphi)=\prod_{\alpha\in\mathcal{B}_{R(\varphi)}}s_\alpha
$$
where the product is well-defined (does not depend on the ordering of $\mathcal{B}_{R(\varphi)}$) since any two distinct elements of $\mathcal{B}_{R(\varphi)}$ are orthogonal (Proposition~\ref{prop:cascade}(\ref{item:stronglyorthogonal})).

\end{prop}

\begin{proof}

This corresponds to a part of the statement of \cite[Proposition~1.10]{kostant}. The reader finds a complete proof there.
\end{proof}

\begin{thm}[{\cite[Proposition~8.16]{thesis}}]
\label{thm:thesis_main}

For all simple roots $\beta$ we have the identity
$$
d_{G/P_\beta}=\sum_{\alpha\in C_R(\beta)}(\omega_\beta,\alpha^\vee)\,.
$$

\end{thm}

\begin{proof}\renewcommand{\qedsymbol}{}

Let $\beta$ be a simple root. Let $C_R(\beta)=\{\theta_1,\ldots,\theta_k\}$ where $\theta_1>\theta_2>\cdots>\theta_k$ (cf. Proposition~\ref{prop:cascade}(\ref{item:totallyordered})). Consistently with our usual notation, $\theta_1$ denotes the highest root. Let $X^i=G(\theta_i)/P(\theta_i)_\beta$ for all $1\leq i\leq k$. We now show by decreasing induction on $i$ that
\begin{equation}
\label{eq:theta-sequence}
d_{X^i}=\sum_{\alpha\in C_{R\left(\theta_i\right)}(\beta)}(\omega_\beta,\alpha^\vee)\text{ for all }1\leq i\leq k\,.
\end{equation}
Since Equation~\eqref{eq:theta-sequence} for $i=1$ is equivalent to the statement of the theorem, this will clearly complete the proof.
\end{proof}

\begin{proof}[Induction base]\renewcommand{\qedsymbol}{}

By definition, it is clear that $C_{R(\theta_k)}(\beta)=\{\theta_k\}$. Therefore, by replacing $R(\theta_k)$ and $R$, we may assume that $k=1$ and thus $C_R(\beta)=\{\theta_1\}$. In this case, Equation~(\ref{eq:theta-sequence}) for $i=k=1$ is equivalent to $d_{X^1}=(\omega_\beta,\theta_1)$. By Proposition~\ref{prop:cascade}(\ref{item:stronglyorthogonal}) and Proposition~\ref{prop:w_o} we have $w_oW_{P_\beta}=s_{\theta_1}W_{P_\beta}$. Therefore, by Remark~\ref{rem:uniqueelement}, Convention~\ref{conv:uniqueelement} and Proposition~\ref{prop:delta}(\ref{item:invariant}) the claimed equality is equivalent to $\delta_{P_\beta}(s_{\theta_1})=(\omega_\beta,\theta_1)$. But this follows directly from Proposition~\ref{prop:delta}(\ref{item:verycosmall}).
\end{proof}

\begin{proof}[Induction step]

Let $1\leq i<k$ and assume that Equation~(\ref{eq:theta-sequence}) for $i+1$ holds. By definition, it is clear that $C_{R(\theta_i)}(\beta)=\{\theta_i,\ldots,\theta_k\}$. Therefore, by replacing $R(\theta_i)$ and $R$, we may assume that $i=1$. (Here we use the natural identification $R(\theta_i)(\theta_{i+1})=R(\theta_{i+1})$.) By Proposition~\ref{prop:cascade}(\ref{item:stronglyorthogonal}) and Proposition~\ref{prop:w_o} we have $w_oW_{P_\beta}=s_{\theta_1}\cdots s_{\theta_k}W_{P_\beta}$ and thus $s_{\theta_1}^*W_{P_\beta}=s_{\theta_2}\cdots s_{\theta_k}W_{P_\beta}$. On the other hand, by the same token we have $w_o(\theta_2)W_{P(\theta_2)_\beta}=s_{\theta_2}\cdots s_{\theta_k}W_{P(\theta_2)_\beta}$ since $C_{R(\theta_2)}(\beta)=\{\theta_2,\ldots,\theta_k\}$. Both equations together yield $s_{\theta_1}^*W_{P_\beta}=w_o(\theta_2)W_{P_\beta}$. Using this equation we find
\begin{align*}
\delta_{P_\beta}(s_{\theta_1}^*)&=\delta_{P_\beta}(w_o(\theta_2))&&\text{by Proposition~\ref{prop:delta}(\ref{item:invariant})}\\
&=\delta_{P(\theta_2)_\beta}(w_o(\theta_2))&&\text{by Theorem~\ref{thm:compatibility} since }w_o(\theta_2)\in W_{G(\theta_2)}\\
&=d_{X^2}&&\text{by Remark~\ref{rem:uniqueelement} and Convention~\ref{conv:uniqueelement}}\\
&=(\omega_\beta,\theta_2^\vee)+\cdots+(\omega_\beta,\theta_k^\vee)&&\text{by induction hypothesis.}
\end{align*}
At the same time, we also have $\delta_{P_\beta}(s_{\theta_1}^*)=d_{X^1}-(\omega_\beta,\theta_1^\vee)$ by Corollary~\ref{cor:inductive}. Both facts together give the desired expression for $d_{X^1}$.\footnote{Note that we used Corollary~\ref{cor:inductive} and indirectly Theorem~\ref{thm:uniqueness} only for a maximal parabolic subgroup $P_\beta$ in which case the statement of Theorem~\ref{thm:uniqueness} is obvious (cf. Remark~\ref{rem:circular}).}
\end{proof}

\begin{thm}
\label{thm:curve}

There exists a curve of degree 
$
\sum_{\beta\in\Delta\setminus\Delta_P}d_{G/P_\beta}d(\beta)
$
which passes through $x(1)$ and $x(w_o)$.

\end{thm}

\begin{proof}

Let $d_X'=\sum_{\beta\in\Delta\setminus\Delta_P}d_{G/P_\beta}d(\beta)$ for short. Let $\mathcal{B}_R=\{\alpha_1,\ldots,\alpha_r\}$ be an arbitrary ordering of $\mathcal{B}_R$. For all $0\leq i\leq r$ we define $u_i=\prod_{j=i+1}^{r}s_{\alpha_j}$. We obviously have $u_r=1$. By Proposition~\ref{prop:w_o}, we have $u_0=w_o$. For all $1\leq i\leq r$ we have $u_{i-1}s_{\alpha_i}=u_i$ (cf. Proposition~\ref{prop:cascade}(\ref{item:stronglyorthogonal})). In other words, $u_0,\ldots,u_r$ is a chain in $W$ from $w_o$ to $w_o$ of degree $\sum_{\alpha\in\mathcal{B}_R}\alpha^\vee$. We now define a subsequence $1\leq i_1<\cdots<i_s\leq r$ by the requirement that
$$
\{i_1,\ldots,i_s\}=\{1\leq i\leq r\mid\alpha_i\in R^+\setminus R_P^+\}\,.
$$
To simplify notation, we define $i_{j}'=i_{j+1}-1$ for all $0\leq j<s$ and $i_s'=r$ and $i_0=0$. Note that for all $1\leq i\leq r$ we have $\bar u_{i-1}=\bar u_i$ if and only if $\alpha_i\in R_P^+$. Therefore the definition of the subsequence $i_1,\ldots, i_s$ directly leads to the equations
$$
\bar u_{i_j}=\cdots=\bar u_{i_j'}\text{ for all }0\leq j\leq s\,.
$$
Consequently, we have $u_{i_{j-1}'}\bar s_{\alpha_{i_j}}=\bar u_{i_j'}$ for all $1\leq j\leq s$. In other words this means that $\bar u_{i_0'},\ldots,\bar u_{i_s'}$ is a chain in $W/W_P$ from $\bar w_o$ to $\bar w_o$ of degree $\sum_{j=1}^s d(\alpha_{i_j})$.

We claim that the degree of this chain equals $d_X'$. Indeed, by definition the degree of this chain equals
\begin{align*}
\sum_{j=1}^s d(\alpha_{i_j})&=\sum_{\alpha\in\mathcal{B}_R}d(\alpha)&&  \text{since $d(\alpha)=0$ for all $\alpha\in R_P^+$.}\\
&=\sum_{\beta\in\Delta\setminus\Delta_P}\left(\sum_{\alpha\in\mathcal{B}_R}(\omega_\beta,\alpha^\vee)\right)d(\beta)&&\text{by rearranging}\\
&=\sum_{\beta\in\Delta\setminus\Delta_P}\left(\sum_{\alpha\in C_R(\beta)}(\omega_\beta,\alpha^\vee)\right)d(\beta)&&\text{since $(\omega_\beta,\alpha^\vee)=0$ if $\beta\not\leq\alpha$}\\
&=d_X'&&\text{by definition and Theorem~\ref{thm:thesis_main}.}
\end{align*}
By \cite[Theorem~9.1: $(2)\Leftrightarrow(3)$ or $(2)\Leftrightarrow(8)$]{fulton} the chain $\bar u_{i_0'},\ldots,\bar u_{i_s'}$ from $\bar w_o$ to $\bar w_o$ of degree $d_X'$ gives rise to a curve of degree $d\leq d_X'$ which passes through $x(1)$ and $x(w_o)$. Since $d$ definitely satisfies $z_d^P=w_X$, Lemma~\ref{lem:bound} implies that $d_X'\leq d$. Altogether, this means $d=d_X'$. This completes the proof of Theorem~\ref{thm:curve} and also the proof of Theorem~\ref{thm:uniqueness}.
\end{proof}

\begin{cor}
\label{cor:identity}

We have the following identities:
$$
d_{G/B}=\sum_{\alpha\in\mathcal{B}_R}\alpha^\vee\text{ and }(d_{G/B})_P=d_X\,.
$$

\end{cor}

\begin{proof}

In view of Theorem~\ref{thm:uniqueness}, the claimed identities follow from the displayed equations in the proof of Theorem~\ref{thm:curve} by applying them to $G/B$ and $X=G/P$ respectively.
\end{proof}

\begin{cor}
\label{cor:combinatorial}

Let $\varphi$ be a positive root. Then we have the inequality $d_{X(\varphi)}\leq d_X$. This leads to the combinatorially interesting identities
$$
\sum_{\alpha\in C_{R(\varphi)}(\beta)}(\omega_\beta,\alpha^\vee)\leq\sum_{\alpha\in C_R(\beta)}(\omega_\beta,\alpha^\vee)\text{ for all }\beta\in\Delta(\varphi)\,.
$$

\end{cor}

\begin{proof}

By Theorem~\ref{thm:uniqueness} we have $\delta_P(w_o)=d_X$ and $\delta_{P(\varphi)}(w_o(\varphi))=d_{X(\varphi)}$. By Theorem~\ref{thm:compatibility}, it follows that $\delta_P(w_o(\varphi))=d_{X(\varphi)}$. Since clearly $\bar w_o(\varphi)\preceq\bar w_o$, Proposition~\ref{prop:delta}(\ref{item:usmallerv}) implies that $d_{X(\varphi)}\leq d_X$. In particular, by applying this to the case $P=B$, we obtain $d_{G(\varphi)/B(\varphi)}\leq d_{G/B}$. Comparing the coefficients of simple coroots $\beta^\vee\in\Delta(\varphi)^\vee$, we obtain in view of Theorem~\ref{thm:uniqueness} that $d_{G(\varphi)/P(\varphi)_\beta}\leq d_{G/P_\beta}$ for all $\beta\in\Delta(\varphi)$. This yields the desired combinatorial identities by invoking Theorem~\ref{thm:thesis_main}.
\end{proof}

\begin{thm}

Let $F$ be an arbitrary subset of $\mathcal{B}_R$. Then we have
$$
\delta_P\left(\prod_{\alpha\in F}s_\alpha\right)=\sum_{\alpha\in F}d(\alpha)
$$
where the product is well-defined (does not depend on the ordering of $F$) since any two distinct elements of $F$ are orthogonal (Proposition~\ref{prop:cascade}(\ref{item:stronglyorthogonal})).

\end{thm}

\begin{proof}

Let $u=\prod_{\alpha\in F}s_\alpha$ and $d=\sum_{\alpha\in F}d(\alpha)$ for short. We first want to show that there exists a degree $d'\in\delta_P(u)$ such that $d'\leq d$. Indeed, by Proposition~\ref{prop:cascade}(\ref{item:locallyhigh}) we know that all elements of $F$ are locally high. Therefore, Corollary~\ref{cor:locallyhigh} shows that $\delta_P(s_\alpha)=d(\alpha)$ for all $\alpha\in F$. By repeated application of Proposition~\ref{prop:delta}(\ref{item:triangle}) (and by application of Proposition~\ref{prop:hecke}(\ref{item:uvsmaller} and Proposition~\ref{prop:delta}(\ref{item:usmallerv}) it follows that there exists a degree $d'\in\delta_P(u)$ such that $d'\leq d$.

It is instructive to give a second proof for the existence of a degree $d'\in\delta_P(u)$ such that $d'\leq d$. Indeed, it is easy to see from Corollary~\ref{cor:identity} that any greedy decomposition $(\alpha_1,\ldots,\alpha_r)$ of $d_{G/B}$ satisfies $\mathcal{B}_R=\{\alpha_1,\ldots,\alpha_r\}$.\footnote{This is an important observation which motivates our further approach. It is the starting point for a theory of generalized chain cascades (cf. Theorem~\ref{thm:orthogonality}).} Let $e=\sum_{\alpha\in F}\alpha^\vee$. By repeated application of Proposition~\ref{prop:delta}(\ref{item:reduce}) it follows that $e\in\delta_B(z_e^B)$. From Proposition~\ref{prop:zd}(\ref{item:greedy}) it follows that any greedy decomposition $(\gamma_1,\ldots,\gamma_m)$ of $e$ satisfies $F=\{\gamma_1,\ldots,\gamma_m\}$. Therefore Proposition~\ref{prop:hecke}(\ref{item:uvsmaller}) implies that $u\preceq z_e^B$. By Proposition~\ref{prop:delta}(\ref{item:usmallerv}) there exists a degree $e'\in\delta_B(u)$ such that $e'\leq e$. By Proposition~\ref{prop:delta}(\ref{item:projection2}) there exists a degree $d'\in\delta_P(u)$ such that $d'\leq e_P'\leq e_P=d$ -- as claimed.

We now proceed with the proof. Let $F^*=\mathcal{B}_R\setminus F$ and $d^ *=\sum_{\alpha\in F^*}d(\alpha)$. By Proposition~\ref{prop:w_o} we have $u^*=\prod_{\alpha\in F^*}s_\alpha$. By the previous considerations applied to $F^*$ there exists a degree $d'^*\in\delta_P(u^*)$ such that $d'^*\leq d^*$. By Corollary~\ref{cor:identity} we have $d'+d'^*\leq d+d^*=d_X$. By Proposition~\ref{prop:geqdx} it follows that $d=d'$ and $d^*=d'^*$. Corollary~\ref{cor:geqdx} then implies $\delta_P(u)=d$ (and $\delta_P(u^*)=d^*$). This completes the proof.
\end{proof}

\todo[inline,color=green]{Combinatorially interesting identity from August 19. Any greedy decomposition $(\alpha_1,\ldots,\alpha_r)$ of $d_X$ satisfies $\mathcal{B}=\{\alpha_1,\ldots,\alpha_r\}$. Starting point for the theory of generalized chain cascades. Reference to Theorem~\ref{thm:orthogonality}. The order in the previous proposition (I meant Proposition~\ref{prop:cascade}) should be different, adapted to the needs of the subsequent publication (quotes in order -- properties which are preserved under generalization). I chose to do the order as in the reference \cite{kostant}.

$\delta_P$-Additivity on $\mathcal{B}_R$. Emphasize the equations $d_{G/B}=\sum_{\alpha\in\mathcal{B}_R}\alpha^\vee$ and $(d_{G/B})_P=d_{X}$.}

\subsection{Reduction in type \texorpdfstring{$\mathsf{A}$}{A}}

\begin{defn}

We say that a simple root $\beta$ is a boundary root of $\Delta$ if $\Delta\setminus\{\beta\}$ is a connected subset of the Dynkin diagram.\footnote{Here we allow the empty set as a connected subset of the Dynkin diagram.}

\end{defn}

\begin{ex} 
\label{ex:reduction}

Let $R$ be of type $\mathsf{A}_n$. Let $\varphi=\theta_1$ be the highest root of $R$. Let $\beta$ be a boundary root of $\Delta$. Let $\alpha=\alpha_{\beta,\varphi}$ be the unique positive root such that $\beta\in\Delta(\alpha)$ and such that $\mathrm{card}(\Delta(\alpha))=\lceil n/2\rceil$. If $n\geq 2$, then $\varphi-\beta$ is a positive root and the degree $d_{G(\varphi-\beta)/B(\varphi-\beta)}$ is well-defined. If $n=1$, then we define $d_{G(\varphi-\beta)/B(\varphi-\beta)}=0$. With this notation, we have the following equality:
$$
d_{G/B}=\alpha^\vee+d_{G(\varphi-\beta)/B(\varphi-\beta)}\,.
$$

\end{ex}

\section{Orthogonality relations in greedy decompositions}

In this section we prove orthogonality relations in greedy decompositions of degrees $d$ such that $d\in\delta_P\left(z_d^P\right)$. These relations are important for the proof of the theorem that any minimal degree in a quantum product of two Schubert cycles is bounded by $d_X$ {\color{black} (Theorem~\ref{thm:fulton} and Theorem~\ref{thm:main})}. They help us to reduce a part of the problem to a situation with finitely many cases which can be analyzed type by type. Moreover, orthogonality relations establish the very first step towards a theory of generalized chain cascades which we want to develop in a subsequent publication.
\todo[inline,color=green]{In the introduction (or somewhere else before the proof of Theorem~\ref{thm:main}) I should say that it is obvious that no minimal degree $d$ in a quantum product can satisfy $d>d_X$, and that there exists at least one minimal degree $d$ in any quantum product which satisfies $d\leq d_X$. Non-trivial is only that all minimal degrees in quantum products are bounded by $d_X$. It is also trivial for maximal parabolic subgroups.

I don't go into the details that generalized chain cascades have consequences for quasi-homogeneity of the moduli space of stable maps to $X$. I do so in the subsequent publication.}


\begin{thm}
\label{thm:orthogonality}

Let $d$ be a degree such that $d\in\delta_P(z_d^P)$. Let $(\alpha_1,\ldots,\alpha_r)$ be a greedy decomposition of $d$. For all simple roots $\beta\in\Delta(d-d(\alpha_1))$
the orthogonality relation $(\alpha_1,\beta)=0$ holds.

\end{thm}

\begin{proof}

Let us write $\alpha=\alpha_1$ for short. Let $\beta\in\Delta(d-d(\alpha))$. By Proposition~\ref{prop:support}(\ref{item:supportofz}) we know that
$$
\Delta(d-d(\alpha))=\widetilde\Delta\left(d-d(\alpha)\right)\setminus\Delta_P=\Delta\left(z_{d-d(\alpha)}^Pw_P\right)\setminus\Delta_P\,.
$$
Therefore Proposition~\ref{prop:relation} shows that $(\alpha,\beta)\geq 0$. Let $2\leq i\leq r$ be an index such that $\beta\in\Delta(\alpha_i)\setminus\Delta_P$. Suppose for a contradiction that $(\alpha,\beta)>0$. 

We now define an element $\hat{s}_{\alpha_i}\in W$ with specific properties. Let $s_{\alpha_i}=s_{\beta_1}\cdots s_{\beta_l}$ be a reduced expression of $s_{\alpha_i}$. Then we know that $\{\beta_1,\ldots,\beta_l\}=\Delta(\alpha)$ (cf. Proposition~\ref{prop:support}(\ref{item:supportofs})). Using the definition of the Hecke product, we clearly have $s_{\alpha_i}=s_{\beta_1}\cdot\ldots\cdot s_{\beta_l}$. Let $1\leq i_1<\cdots<i_m\leq l$ be the ordered set of indices such that $\{1\leq j\leq l\mid\beta_j\neq\beta\}=\{i_1,\ldots,i_m\}$. Then we define $\hat{s}_{\alpha_i}=s_{\beta_{i_1}}\cdot\ldots\cdot s_{\beta_{i_m}}$. It is obvious from the definition and Proposition~\ref{prop:hecke}(\ref{item:preceq}) that $\hat{s}_{\alpha_i}\preceq s_{\alpha_i}$. But we want to prove more.

Let $\varphi_1,\ldots,\varphi_k$ be locally high roots such that $\Delta(\varphi_1),\ldots,\Delta(\varphi_k)$ are the distinct connected components of $\{\beta_{i_1},\ldots,\beta_{i_m}\}$. This means that for any two distinct indices $1\leq j\neq j'\leq k$ every root in $\Delta(\varphi_j)$ is (strongly) orthogonal to every root in $\Delta(\varphi_{j'})$. By \cite[Proposition~4.8(a)]{curvenbhd2} this means in particular that for every root $\gamma\in\Delta(\varphi_j)$ and every root $\gamma'\in\Delta(\varphi_{j'})$ we have $s_{\gamma}\cdot s_{\gamma'}=s_{\gamma'}\cdot s_\gamma$. Therefore we can reorder the terms $s_{\beta_{i_1}},\ldots,s_{\beta_{i_m}}$ in the Hecke product $\hat{s}_{\alpha_i}$ in such a way that the first $n_1$ indices $j=i_1,\ldots,i_{n_1}$ satisfy $\beta_{i_j}\in\Delta(\varphi_1)$, that the second $n_2$ indices $j=i_{n_1+1},\ldots,i_{n_1+n_2}$ satisfy $\beta_{i_j}\in\Delta(\varphi_{2})$, etc. Without loss of generality we may assume that the indices $i_1,\ldots,i_m$ are ordered in this way. Now, for all $1\leq j\leq k$ let $w_j$ be the ordered Hecke product of the $n_j$ factors $s_{\beta_{j'}}$ of $\hat{s}_{\alpha_i}$ which satisfy $\beta_{j'}\in\Delta(\varphi_j)$. Here we preserve of course the original order of the factors in $\hat{s}_{\alpha_i}$. Then we clearly have $\hat{s}_{\alpha_i}=w_1\cdot\ldots\cdot w_k$, in particular $w_j\preceq\hat{s}_{\alpha_i}\preceq s_{\alpha_i}$ for all $1\leq j\leq k$ (cf. Proposition~\ref{prop:hecke}(\ref{item:preceq})). 

Since $\alpha_i$ is part of a greedy decomposition of $d$ we know that $\alpha_i$ is $P$-cosmall (Proposition~\ref{prop:zd}(\ref{item:unique})). Therefore Proposition~\ref{prop:delta}(\ref{item:pcosmall}) shows that $d(\alpha_i)\in\delta_P(s_{\alpha_i})$. By Proposition~\ref{prop:delta}(\ref{item:usmallerv}) there exists for each $1\leq j\leq k$ a degree $d_j\in\delta_P(w_j)$ such that $d_j\leq d(\alpha_i)$. It is clear from the definition of $w_j$ as a Hecke product of simple reflections along roots in $\Delta(\varphi_j)$ that $w_j\in W_{G(\varphi_j)}$. By {\color{black} Theorem~\ref{thm:compatibility}} we therefore know that $\delta_P(w_j)=\delta_{P(\varphi_j)}(w_j)\subseteq H_2(X(\varphi_j))$, in particular $d_j\in H_2(X(\varphi_j))$. Here we identify as usual $H_2(X(\varphi_j))$ with a sublattice of $H_2(X)$. This means that the support $\Delta(d_j)$ of $d_j$ is contained in $\Delta(\varphi_j)\setminus\Delta_P$.

\todo[inline,color=green]{I should replace $G(\varphi)/P(\varphi)$ by $X(\varphi)$ everywhere in this proof.}

Since the sets $\Delta(\varphi_1),\ldots,\Delta(\varphi_k)$ are pairwise disjoint (even pairwise totally disjoint) 
\todo[color=green]{The terms \enquote{totally disjoint} and \enquote{strongly orthogonal} need to be explained in footnotes or in Section~\ref{sec:cascade} with a reference to Bourbaki. Terms explained in the section \enquote{The cascade of orthogonal roots}.}
the supports $\Delta(d_1),\ldots,\Delta(d_k)$ must also be disjoint. Therefore we can infer from the inequality $d_j\leq d(\alpha_i)$ for all $1\leq j\leq k$ the inequality $\sum_{j=1}^k d_j\leq d(\alpha_i)$. 

Let us write $d(\alpha_i)'=\sum_{j=1}^k d_j$ for short. We just saw that we have an inequality $d(\alpha_i)'\leq d(\alpha_i)$. We claim that we even have a strict inequality $d(\alpha_i)'<d(\alpha_i)$. Indeed, by our choice of the index $i$ we know that $\beta\in\Delta(\alpha_i)\setminus\Delta_P=\Delta(d(\alpha_i))$ and thus $d(\alpha_i)_\beta>0$. On the other hand, we know by our choice of the indices $i_1,\ldots,i_m$ and our definition of the connected components $\Delta(\varphi_j)$ that
$$
\bigcup_{j=1}^k\Delta(\varphi_j)=\{\beta_{i_1},\ldots,\beta_{i_m}\}=\{\beta_1,\ldots,\beta_l\}\setminus\{\beta\}\,.
$$
Therefore we have $\beta\notin\Delta(\varphi_j)$, in particular $\beta\notin\Delta(d_j)\subseteq\Delta(\varphi_j)\setminus\Delta_P$ for all $1\leq j\leq k$. This shows that $(d_j)_\beta=0$ for all $1\leq j\leq k$ and thus $d(\alpha_i)'_\beta=0$. Therefore we find the desired strict inequality $d(\alpha_i)'<d(\alpha_i)$.

By construction, we have $\hat{s}_{\alpha_i}=w_1\cdot\ldots\cdot w_k$ and $d_j\in\delta_P(w_j)$ for all $1\leq j\leq k$. By repeated application of Proposition~\ref{prop:delta}(\ref{item:triangle}) we can find an element $\hat{d}(\alpha_i)\in\delta_P(\hat{s}_{\alpha_i})$ such that $\hat{d}(\alpha_i)\leq d(\alpha_i)'$. Since $d(\alpha_i)'<d(\alpha_i)$ this element clearly satisfies the strict inequality $\hat{d}(\alpha_i)<d(\alpha_i)$. The construction of a degree $\hat{d}(\alpha_i)\in\delta_P(\hat{s}_{\alpha_i})$ such that $\hat{d}(\alpha_i)<d(\alpha_i)$ was the whole purpose of the last four paragraphs.

We now return to the situation in the second paragraph of the proof. Recall our contradictory assumption that $(\alpha,\beta)>0$. Recall that we constructed an element $\hat{s}_{\alpha_i}$ from $s_{\alpha_i}$ by canceling all simple reflections $s_\beta$ from a reduced expression of $s_{\alpha_i}$ and Hecke multiply them to an element $\hat{s}_{\alpha_i}$ which satisfies $\hat{s}_{\alpha_i}\preceq s_{\alpha_i}$. Using our assumption $(\alpha,\beta)>0$, we now want to prove that $s_\alpha\cdot s_{\alpha_i}=s_{\alpha}\cdot\hat{s}_{\alpha_i}$. Indeed, for all $1\leq j\leq l$ we have $\beta_j\in\Delta(\alpha_i)\subseteq\Delta(z_{d-d(\alpha)}^Pw_P)$. Therefore Proposition~\ref{prop:relation} gives $(\alpha,\beta_j)\geq 0$. By \cite[Proposition~4.8(a)]{curvenbhd2} this means that we have the commutation relation $s_\alpha\cdot s_{\beta_j}=s_{\beta_j}\cdot s_\alpha$. Moreover, by the definition of the Hecke product, the assumption $(\alpha,\beta)>0$ is equivalent to $s_\alpha\cdot s_\beta=s_\beta\cdot s_\alpha=s_\alpha$.
\todo[color=green]{I have to ask Nicolas Perrin why this true. I took it from the proof of \cite[Proposition~4.8(a)]{curvenbhd2}. Important for generalizations (cf. notes on September 16). Explanation in \cite[proof of the exchange condition in 1.7]{humphreys3}.}
Thus, we can commute in the expression $s_\alpha\cdot s_{\alpha_i}$ every simple reflection $s_{\beta_j}$ such that $\beta_j\neq\beta$ with $s_\alpha$ and replace every occurrence of $s_\alpha\cdot s_\beta$ by $s_\alpha$. This gives the desired equality $s_\alpha\cdot s_{\alpha_i}=s_\alpha\cdot\hat{s}_{\alpha_i}$.

Finally, we are able to prove the equality
\begin{equation}
\label{eq:reduction}
s_{\alpha_1}\cdot\ldots\cdot s_{\alpha_r}=s_{\alpha_1}\cdot\ldots\hat{s}_{\alpha_i}\cdot\ldots\cdot s_{\alpha_r}\,.
\end{equation}
Indeed, we know that $(\alpha_1,\ldots,\alpha_r)$ is the greedy decomposition of $d$ and therefore every two factors of $s_{\alpha_1}\cdot\ldots\cdot s_{\alpha_r}$ Hecke commute (Proposition~\ref{prop:zd}(\ref{item:commute})). Using this commutation relation and the equality from the previous paragraph, we find
$$
s_{\alpha_1}\cdot\ldots\cdot s_{\alpha_r}=s_{\alpha_2}\cdot\ldots s_\alpha\cdot s_{\alpha_i}\cdot\ldots\cdot s_{\alpha_r}=s_{\alpha_2}\cdot\ldots s_\alpha\cdot\hat{s}_{\alpha_i}\cdot\ldots\cdot s_{\alpha_r}=s_{\alpha_1}\cdot\ldots\hat{s}_{\alpha_i}\cdot\ldots\cdot s_{\alpha_r}\,.
$$
This proves the desired equality~(\ref{eq:reduction}).

All elements of a greedy decomposition of $d$ are $P$-cosmall (Proposition~\ref{prop:zd}(\ref{item:unique})). Therefore we know that $\alpha_j$ is $P$-cosmall for all $1\leq j\leq r$ and thus $d(\alpha_j)\in\delta_P(s_{\alpha_i})$ (Proposition~\ref{prop:delta}(\ref{item:pcosmall})). By repeated application of Proposition~\ref{prop:delta}(\ref{item:triangle}) to the sequence of degrees $d(\alpha_1),\ldots,\hat{d}(\alpha_i),\ldots,d(\alpha_r)$ we can find a degree 
$
d'\in\delta_P(s_{\alpha_1}\cdot\ldots\cdot\hat{s}_{\alpha_i}\cdot\ldots\cdot s_{\alpha_r})
$
such that $d'\leq d(\alpha_1)+\cdots+\hat{d}(\alpha_i)+\cdots+d(\alpha_r)$. Since $\hat{d}(\alpha_i)<d(\alpha_i)$ it follows that $d'<d$. By Equation~(\ref{eq:reduction}) and the fact that $s_{\alpha_1}\cdot\ldots\cdot s_{\alpha_r}$ and $z_d^P$ belong to the same class modulo $W_P$ (cf. Proposition~\ref{prop:delta}(\ref{item:invariant})) it follows that $d'\in\delta_P(z_d^P)$. But we know that $d'<d$ and by assumption we have $d\in\delta_P(z_d^P)$. This is a contradiction since two distinct elements of $\delta_P(z_d^P)$ must be incomparable (Remark~\ref{rem:uniqueelement}).
\end{proof}

\begin{cor}

Let $d$ be a degree such that $d\in\delta_P(z_d^P)$. Let $(\alpha_1,\ldots,\alpha_r)$ be a greedy decomposition of $d$. Let $\beta\in\Delta(d-d(\alpha_1))$. Then $\alpha_1$ and $\beta$ are strongly orthogonal.

\end{cor}

\begin{proof}

Let $\alpha_1$ and $\beta$ be as in the statement. Let $\alpha=\alpha_1$ for short. By Theorem~\ref{thm:orthogonality} we know that $(\alpha,\beta)=0$. By Lemma~\ref{lem:stronglyorthogonal} it suffices to show that $\alpha+\beta\notin R$. But this is accomplished by Corollary~\ref{cor:stronglyorthogonal2}.
\end{proof}

\begin{cor}
\label{cor:stronglyorthogonal}

Let $e$ be a degree such that $e\in\delta_B\left(z_e^B\right)$. Let $\alpha$ and $\alpha'$ be two different entries of a greedy decomposition of $e$. Then $\alpha$ and $\alpha'$ are strongly orthogonal.

\end{cor}

\begin{proof}

Let $(\alpha_1,\ldots,\alpha_r)$ be a greedy decomposition of $e$. By the uniqueness of the greedy decomposition up to reordering we know that $\alpha=\alpha_i$ and $\alpha'=\alpha_j$ for some $1\leq i,j\leq r$. By the choice of $\alpha$ and $\alpha'$ we know that $i\neq j$. By replacing $\alpha$ and $\alpha'$ if necessary we may assume that $i<j$. By Proposition~\ref{prop:zd}(\ref{item:greedy}) we know that $(\alpha,\alpha')$ is a greedy decomposition of $\alpha^\vee+\alpha'^\vee$. By Theorem~\ref{thm:orthogonality} applied to the greedy decomposition $(\alpha,\alpha')$ of $\alpha^\vee+\alpha'^\vee$ it follows that for all $\beta\in\Delta(\alpha')$ the orthogonality relation $(\alpha,\beta)=0$ holds. By bilinearity this means $(\alpha,\alpha')=0$. Moreover, by Corollary~\ref{cor:relation} we know that $\alpha+\alpha'\notin R$. Lemma~\ref{lem:stronglyorthogonal} implies that $\alpha$ and $\alpha'$ are strongly orthogonal.
\end{proof}

\begin{rem}

Corollary~\ref{cor:stronglyorthogonal} says in particular that two different entries of a greedy decomposition of $e\in\delta_B\left(z_e^B\right)$ are distinct. There are no repeated entries in a greedy decomposition of $e$.

\end{rem}

\section{Exceptional roots}

In this section we introduce the class of exceptional roots. From the perspective of the proof of Theorem~\ref{thm:main}, these are the roots which cause the most trouble. We have to perform a type by type check on these roots in order to conclude in full generality. While the conclusion of Theorem~\ref{thm:main} is hard to check directly, a specific inequality for degrees of exceptional roots can be verified by going through the list of all exceptional roots (cf. Lemma~\ref{lem:technical}). This inequality is comparatively easy to check and amounts to solve a finite problem which can be done by hand. 

If the root system is of type $\mathsf{A}_n$, $\mathsf{B}_2$, $\mathsf{B}_3$, $\mathsf{C}_p$, $\mathsf{D}_3$, $\mathsf{D}_4$ or $\mathsf{G}_2$, 
\todo[color=green,disable]{I have to expand this list later, when Table~\ref{table:exceptional} is complete.}
then for every positive root $\varphi$ the root system $R(\varphi)$ (in particular $R$ itself) has no exceptional roots. All other types have exceptional roots (cf. Table~\ref{table:exceptional}). The reader who is only interested in these cases, may skip this section and go directly to the next section. It is easy to adapt the proof of Theorem~\ref{thm:main} to the case that there are no exceptional roots. Many of our considerations are then empty and the most technical part is superfluous while the main ideas stay the same.

The definition of exceptional roots is ad hoc -- created for the particular purpose of the proof of Theorem~\ref{thm:main}. 
%
%
We wish to relate exceptional roots to a more geometric and intuitive notion 
%
%
in a subsequent publication. But first, we need to carry out the foundational work. 
%

\todo[inline,color=green]{Note somewhere in the section on chain cascades that $R^\circ$ might be emtpy which happens if and only if $R$ is of type $\mathsf{A}_1$ or $\mathsf{A}_2$.

We want to say that to work out the table of exceptional roots it is convenient to use the tables in \cite{curvenbhd2}. We want to explain the situation in type $\mathsf{A}$ a bit more concretely since it is super simple.

I need an explanation of the notation in Table~\ref{table:exceptional} with reference to Bourbaki. Maybe even in the caption.

It makes sense to introduce the notion of boundary roots already in the section on chain cascade. It will be useful for the the reduction and extension in type $\mathsf{A}$.}

\begin{defn}

Let $\alpha$ be a positive roots. We say that $\alpha$ is an exceptional root if the following properties are satisfied:

\begin{itemize}

\item

The support of $\alpha$ satisfies $\Delta(\alpha)=\Delta$.

\item

There exists a simple root $\beta\in\Delta\setminus\Delta^\circ$ such that $(\alpha,\beta)=0$ and such that $\alpha$ is a maximal root of $\alpha^\vee+\beta^\vee$.

\end{itemize}

\end{defn}

\begin{fact}

Suppose there exists an exceptional root $\alpha$. Then the set $\Delta\setminus\Delta^\circ$ consists of a unique simple root $\beta$. The root $\alpha$ is $B$-cosmall and $\alpha$ is strongly orthogonal to $\beta$.

\end{fact}

\begin{proof}

The set $\Delta\setminus\Delta^\circ$ consists of a unique simple root $\beta$ if and only if $R$ is not of type $\mathsf{A}$ (cf. Table~\ref{table:exceptional}). Let $\alpha$ be an exceptional root. Then $R$ cannot be of type $\mathsf{A}$ since otherwise $\alpha=\theta_1$ where $\theta_1$ is the highest root of $R$, and thus $(\alpha,\beta)\neq 0$ for all $\beta\in\Delta\setminus\Delta^\circ$. An exceptional root $\alpha$ is a maximal root of $\alpha^\vee+\beta^\vee$, in particular $B$-cosmall. Moreover, it is clear that $(\alpha,\beta)$ is a greedy decomposition of $\alpha^\vee+\beta^\vee$. Therefore it follows from Corollary~\ref{cor:relation} that $\alpha+\beta\notin R$. Lemma~\ref{lem:stronglyorthogonal} then implies that $\alpha$ and $\beta$ are strongly orthogonal.
\end{proof}

\begin{rem}

To create Table~\ref{table:exceptional} one basically has to go through the list of all $B$-cosmall roots which have full support and check if they are orthogonal to a root in $\Delta\setminus\Delta^\circ$. The examples in \cite[Example~4.1, 4.2 and 4.3]{curvenbhd2} where the reader can find a complete list of all $B$-cosmall roots in the non-simply laced types are helpful to accomplish this task.

\end{rem}

\todo[inline,color=green]{This (the previous remark on the creation of the table) is supposed to mean that we can replace in the definition of exceptional roots the requirement \enquote{$\alpha$ is a maximal root of $\alpha^\vee+\beta^\vee$} by the requirement \enquote{$\alpha$ is $B$-cosmall}. Before I say something false, I leave it as it is until I find a conceptional proof.}

\begin{table}

\begin{tabular}{lll}

Type & Exceptional roots & $\Delta\setminus\Delta^\circ$\\ \hline\hline
$\mathsf{A}_n$ & None & $\alpha_1\,$, $\alpha_n$ \\
$\mathsf{B}_2\,$, $\mathsf{B}_3$ & None & $\alpha_2$ \\
$\mathsf{B}_\ell\,$, $\ell\geq 4$ & 
$\sum_{1\leq i<j}\alpha_i+2\sum_{j\leq i\leq\ell}\alpha_i$ where $4\leq j\leq\ell$
& $\alpha_2$ \\
$\mathsf{C}_p$ & None & $\alpha_1$ \\
$\mathsf{D}_4$ & None & $\alpha_2$ \\
$\mathsf{D}_p\,$, $p>4$ & $\sum_{1\leq i<j}\alpha_i+2\sum_{j\leq i<p-1}\alpha_i+\alpha_{p-1}+\alpha_p$ where $4\leq j<p$ & $\alpha_2$ \\
$\mathsf{E}_6$ & $111211$, $112211$, $111221$, $112221$ & $\alpha_2$ \\
$\mathsf{E}_7$ & $1122111$, $1122211$, $1122221$, $1123211$, $1123221$, & $\alpha_1$ \\ 
& $1223211$, $1123321$, $1223221$, $1223321$, $1224321$ & \\
$\mathsf{E}_8$ & $11122221$, $11222221$, $11232221$, $12232221$, & $\alpha_8$\\
& $11233221$, $12233221$, $11233321$, $12243221$, & \\ 
& $12233321$, $12343221$, $12243321$, $22343221$, & \\
& $12343321$, $12244321$, $22343321$, $12344321$, & \\
& $12354321$, $22344321$, $13354321$, $22354321$, & \\
& $23354321$, $22454321$, $23454321$, $23464321$, & \\
& $23465321$, $23465421$ & \\
$\mathsf{F}_4$ & $1222$, $1242$ & $\alpha_1$ \\
$\mathsf{G}_2$ & None & $\alpha_2$

\end{tabular}

\caption{List of exceptional roots. We denote by $\Delta=\{\alpha_1,\alpha_2,\alpha_3,\ldots\}$ the set of simple roots with the numbering as in \cite[Plate I-IX]{bourbaki_roots}. We denote by $abc\cdots$ the root $a\alpha_1+b\alpha_2+c\alpha_3+\cdots$.}

\label{table:exceptional}

\end{table}

\begin{rem}

By inspecting Table~\ref{table:exceptional}, we see that all exceptional roots are long. Moreover, we see that if there exists an exceptional root, then there also exists a unique highest exceptional root $\alpha$, in the sense that $\alpha$ is an exceptional root and that $\alpha'\leq\alpha$ for all exceptional roots $\alpha'$.

\end{rem}

\todo[inline,color=green]{Lemma~\ref{lem:technical} has been checked in type $\mathsf{F}_4$. We have $\{\gamma\in\Delta\mid(\alpha,\gamma)=0\}=\{\alpha_1\}$ for all exceptional roots $\alpha$. In type $\mathsf{E}_6,\ldots$

Highest exceptional root exists. All exceptional roots are long. 

If $R$ is simply laced we declare all roots to be long and none to be short. If there is only one root length, then the highest short root does not exist. If we speak about the highest short root, we mean that there are two root lengths. Footnote at the first occurrence of the term \enquote{long}, in the example on $B$-cosmall roots (long roots are $B$-cosmall).}

\begin{lem}
\label{lem:technical}

Let $\alpha$ be an exceptional root. Let $\beta$ be the unique element of $\Delta\setminus\Delta^\circ$. Let $\varphi$ be the unique locally high root such that $\Delta(\varphi)$ is the unique connected component of $\{\gamma\in\Delta\mid(\alpha,\gamma)=0\}$ which contains $\beta$. Then $R(\varphi)$ is a root system of type $\mathsf{A}$ and $\beta$ is a boundary root of $\Delta(\varphi)$. In Example~\ref{ex:reduction} we have associated to a root system of type $\mathsf{A}$ with highest root $\varphi$ and to a boundary root $\beta$ of its set of simple roots a positive root $\alpha_{\beta,\varphi}$. Let $\alpha_{\beta,\varphi}\in R(\varphi)$ be the positive root associated to $\beta$ and $\varphi$ in this way. Then we have the inequality
\begin{equation}
\label{eq:inequality}
\alpha^\vee+\alpha_{\beta,\varphi}^\vee\leq\theta_1^\vee
\end{equation}
where $\theta_1$ denotes as usual the highest root of $R$. By definition of $\alpha_{\beta,\varphi}$ it is clear that $\beta\leq\alpha_{\beta,\varphi}$. Therefore Inequality~(\ref{eq:inequality}) means in particular that
\begin{equation}
\label{eq:inequality2}
\alpha^\vee+\beta^\vee\leq\theta_1^\vee\,.
\end{equation}
Inequality~(\ref{eq:inequality}) and (\ref{eq:inequality2}) are the same if and only if $\alpha_{\beta,\varphi}=\beta$ if and only if $R(\varphi)$ is of type $\mathsf{A}_1$ or $\mathsf{A}_2$.

\end{lem}

\begin{proof}

We do not have a conceptional proof of this lemma. The only possibility we know to prove it is to go through the list of all exceptional roots in Table~\ref{table:exceptional} and check the assertions case by case. Once we know that $R(\varphi)$ is of type $\mathsf{A}_n$, it is clear that Inequality~(\ref{eq:inequality}) and (\ref{eq:inequality2}) are the same if and only if $n\in\{1,2\}$ since the support of $\alpha_{\beta,\varphi}$ has by definition cardinality $\lceil n/2\rceil$ (cf. Example~\ref{ex:reduction}).

To illustrate the techniques, we exemplify the necessary computations in type $\mathsf{B}_\ell$ where $\ell\geq 4$. This example has the feature that the support of $\varphi$ and thus the support of $\alpha_{\beta,\varphi}$ become arbitrary large if $\ell$ is sufficiently large.

Suppose that $R$ is of type $\mathsf{B}_\ell$ where $\ell\geq 4$. Let $\alpha,\beta$ and $\varphi$ be as in the statement. By Table~\ref{table:exceptional} we know that $\beta=\alpha_2$ and that there exists a $4\leq j\leq\ell$ such that
$$
\alpha=\sum_{1\leq i<j}\alpha_i+2\sum_{j\leq i\leq\ell}\alpha_i\,.
$$
A simple computation shows that
$
\varphi=\alpha_2+\cdots+\alpha_{j-2}
$.
Therefore $R(\varphi)$ is of type $\mathsf{A}_{j-3}$ and $\beta$ is indeed a boundary root of $\Delta(\varphi)$. By definition, we must have
$$
\alpha_{\beta,\varphi}=\sum_{i=2}^{\lfloor j/2\rfloor}\alpha_i
$$
and thus
$$
\alpha^\vee+\alpha_{\beta,\varphi}^\vee=\alpha_1^\vee+2\sum_{i=2}^{\lfloor j/2\rfloor}\alpha_i^\vee+\sum_{\lfloor j/2\rfloor<i<j}\alpha_i^\vee+2\sum_{j\leq i<\ell}\alpha_i^\vee+\alpha_\ell^\vee\,.
$$
On the other hand, we have
$$
\theta_1^\vee=\alpha_1^\vee+2\alpha_2^\vee+\cdots+2\alpha_{\ell-1}^\vee+\alpha_\ell^\vee\,.
$$
Therefore, Inequality~(\ref{eq:inequality}) is obviously satisfied. We see that we even have a strict inequality $\alpha^\vee+\alpha_{\beta,\varphi}^\vee<\theta_1^\vee$ (cf. Remark~\ref{rem:strict}).
\end{proof}

\begin{lem}
\label{lem:technical2}

Let $\alpha$ be an exceptional root. Let $\varphi_1,\ldots,\varphi_k$ be locally high roots such that $\Delta(\varphi_1),\ldots,\Delta(\varphi_k)$ are the distinct connected components of $\{\gamma\in\Delta\mid(\alpha,\gamma)=0\}$. Then we have the inequality
\begin{equation}
\label{eq:inequality3}
\alpha^\vee+\sum_{i=1}^k d_{G(\varphi_i)/B(\varphi_i)}\leq d_{G/B}\,.
\end{equation}

\end{lem}

\begin{proof}

Let $\alpha,\varphi_1,\ldots,\varphi_k$ be as in the statement. Let $\beta$ be the unique element of $\Delta\setminus\Delta^\circ$. Let $1\leq j\leq k$ be the unique index such that $\beta\in\Delta(\varphi_j)$. For all $1\leq i\leq k$ we define $\phi_i=\varphi_i$ if $i\neq j$ and $\phi_j=\varphi_j-\beta$. Let $\psi_1,\ldots,\psi_{k'}$ be locally high roots such that $\Delta(\psi_1),\ldots,\Delta(\psi_{k'})$ are the distinct connected components of $\Delta^\circ$. 

By Lemma~\ref{lem:technical}, we know that $\Delta(\phi_j)=\Delta(\varphi_j)\setminus\{\beta\}$.\footnote{If $\phi_j=0$ we obviously set $\Delta(\phi_j)=\emptyset$.} It follows that $\beta\notin\Delta(\phi_i)$ or equivalently $\Delta(\phi_i)\subseteq\Delta^\circ$ for all $1\leq i\leq k$. Thus, we can find for each $1\leq i\leq k$ a unique $1\leq f(i)\leq k'$ such that $\Delta(\phi_i)\subseteq\Delta(\psi_{f(i)})$. In this way, we define a function $f\colon\{1,\ldots,k\}\to\{1,\ldots,k'\}$. By Corollary~\ref{cor:combinatorial}, we have for all $1\leq i\leq k$ the inequality 
$$
d_{G(\phi_i)/B(\phi_i)}\leq d_{G(\psi_{f(i)})/B(\psi_{f(i)})}\,.\footnote{If $\phi_j=0$ we set $d_{G(\phi_j)/B(\phi_j)}=0$ as in Example~\ref{ex:reduction}. The inequality is then obviously satisfied.}
$$
By definition, $\Delta(\varphi_1),\ldots,\Delta(\varphi_k)$ are pairwise totally disjoint, in particular $\Delta(\phi_1),\ldots,\Delta(\phi_k)$ are pairwise totally disjoint. It is very clear (for example because of Corollary~\ref{cor:identity}) that 
$$
\Delta\left(d_{G(\phi_i)/B(\phi_i)}\right)=\Delta(\phi_i)\text{ for all }1\leq i\leq k\,.
$$
Therefore the disjointness results in the inequality
$$
\sum_{i\in f^{-1}(i')}d_{G(\phi_i)/B(\phi_i)}\leq d_{G(\psi_{i'})/B(\psi_{i'})}\text{ for all }1\leq i'\leq k'\,.
$$
Summation over $1\leq i'\leq k'$ eventually leads to the inequality
\begin{equation}
\label{eq:help}
\sum_{i=1}^k d_{G(\phi_i)/B(\phi_i)}\leq\sum_{i'=1}^{k'}d_{G(\psi_{i'})/B(\psi_{i'})}\,.
\end{equation}
By Corollary~\ref{cor:identity} and Remark~\ref{rem:disjoint} it is clear that
\begin{equation}
\label{eq:help2}
d_{G/B}=\theta_1^\vee+\sum_{i'=1}^{k'}d_{G(\psi_{i'})/B(\psi_{i'})}\,.
\end{equation}
If we plug in Equality~(\ref{eq:help2}) in Inequality~(\ref{eq:help}) and use Inequality~(\ref{eq:inequality}) from Lemma~\ref{lem:technical}, we obtain
$$
\alpha^\vee+\alpha_{\beta,\varphi_j}^\vee+\sum_{i=1}^k d_{G(\phi_i)/B(\phi_i)}\leq d_{G/B}
$$
Finally, we see from Example~\ref{ex:reduction} that
$$
\sum_{i=1}^k d_{G(\varphi_i)/B(\varphi_i)}=\alpha_{\beta,\varphi_j}^\vee+\sum_{i=1}^k d_{G(\phi_i)/B(\phi_i)}\,.
$$
This completes the proof.  
\end{proof}

\todo[inline,color=green]{Note the disjointness of the union after the definition of $\mathcal{B}_R$. Classification of very $P$-cosmall roots. It's always the highest root if $P$ is not maximal. As an example at the relevant position.}

\begin{rem}
\label{rem:strict}

We have formulated Inequality~(\ref{eq:inequality}) in Lemma~\ref{lem:technical} in the way we wanted to use it in the proof of Lemma~\ref{lem:technical2}. But it is by no means optimal. In fact, the reader can convince himself that much more is true. Let $\alpha$ be an exceptional root and let $\varphi$ be associated to $\alpha$ as in the statement of Lemma~\ref{lem:technical}. Then we have $\alpha^\vee+\varphi^\vee<\theta_1^\vee$. This implies in particular that Inequality~(\ref{eq:inequality}) is strict since we obviously have $\alpha_{\beta,\varphi}^\vee\leq\varphi^\vee$. Consequently, Inequality~(\ref{eq:inequality3}) is also strict. 

\end{rem}

\section{Minimal degrees in quantum products: proof of the main theorem}
\label{sec:main}

The final section is devoted to the proof of the main theorem, which says that any minimal degree in a quantum product of two Schubert cycles is bounded by $d_X$ {\color{black} (Theorem~\ref{thm:fulton} and Theorem~\ref{thm:main})}. After all the preliminary work done up to now, the proof of the main theorem remains a pure formality. We just have to put together the statements from the previous sections. The proof proceeds by induction on the cardinality $N$ of $\Delta$. First, we clarify which statements $(\mathrm{H}_n)$ where $1\leq n\leq N$ are to be proved by induction and draw preliminary conclusions. After a series of reductions, we see that is suffices to prove a simpler statement $(\mathrm{A}_N)$ under the assumption of $(\mathrm{H}_{N-1})$. The heart of the induction step can be found in Lemma~\ref{lem:heart}.

Let $N=\mathrm{card}(\Delta)$. For all $1\leq n\leq N$, we define the following statements $(\mathrm{H}_n)$ and $(\mathrm{A}_N)$.
\begin{align*}
(\mathrm{H}_n)\Longleftrightarrow &\left\{
\begin{array}{l}    
\text{For all }\varphi\in R^+\text{ and all }e\in H_2(G(\varphi)/B(\varphi))\text{ such that}\\
\mathrm{card}(\Delta(\varphi))\leq n\text{ and such that }e\in\delta_{B(\varphi)}\left(z_e^{B(\varphi)}\right)\text{ the}\\
\text{inequality }e\leq d_{G(\varphi)/B(\varphi)}\text{ is satisfied.}
\end{array}
\right.\\
(\mathrm{A}_N)\Longleftrightarrow &\left\{
\begin{array}{l}
\text{For all degrees }e\in H_2(G/B)\text{ such that }\Delta(e)=\Delta\text{ and }\\
\text{such that }e\in\delta_B(z_e^B)\text{ the inequality }e\leq d_{G/B}\text{ holds.}
\end{array}
\right.
\end{align*}

\begin{rem}
\label{rem:base}

The statement $(\mathrm{H}_1)$ is obviously satisfied. Indeed, let $\beta$ be a simple root. By Remark~\ref{rem:uniqueelement} and Convention~\ref{conv:uniqueelement} we know that $\delta_{B(\beta)}$ is a function with values in $\mathbb{Z}$. Let $e\in H_2(G(\beta)/B(\beta))$ be a degree such that $\delta_{B(\beta)}(z_e^{B(\beta)})=e$. Since $z_e^{B(\beta)}\preceq w_o(\beta)$ it follows that $e\leq d_{G(\beta)/B(\beta)}$ (cf. Proposition~\ref{prop:delta}(\ref{item:usmallerv})). The statement $(\mathrm{H}_1)$ will serve us as induction base in the proof of Theorem~\ref{thm:main}.

\end{rem}

\begin{lem}
\label{lem:equivalence-2}

Let $N=\mathrm{card}(\Delta)$. Suppose that $(\mathrm{H}_{N-1})$ is satisfied. Let $e\in H_2(G/B)$ be a connected degree such that $\Delta(e)\neq\Delta$ and such that $e\in\delta_B(z_e^B)$. Then we have the inequality $e\leq d_{G(\alpha(e))/B(\alpha(e))}\leq d_{G/B}$.

\end{lem}

\begin{proof}

Let $e\in H_2(G/B)$ be a degree as in the statement. By Fact~\ref{fact:local} we have $z_e^B=z_e^{B(\alpha(e))}$. Thus, Theorem~\ref{thm:compatibility} implies that $e\in\delta_{B(\alpha(e))}\left(z_e^{B(\alpha(e))}\right)$. Since $\Delta(e)\neq\Delta$ by assumption, the statement $(\mathrm{H}_{N-1})$ applies and implies that $e\leq d_{G(\alpha(e))/B(\alpha(e))}$. The very last inequality $d_{G(\alpha(e))/B(\alpha(e))}\leq d_{G/B}$ follows from Corollary~\ref{cor:combinatorial} applied to $P=B$ and $\varphi=\alpha(e)$.
\end{proof}

\begin{lem}
\label{lem:equivalence-1}

Let $N=\mathrm{card}(\Delta)$. Suppose that $(\mathrm{H}_{N-1})$ is satisfied. Let $e\in H_2(G/B)$ be a disconnected degree such that $e\in\delta_B(z_e^B)$. Then we have the inequality $e\leq d_{G/B}$.

\end{lem}

\begin{proof}

Let $e\in H_2(G/B)$ be a degree as in the statement. Let $\varphi_1,\ldots,\varphi_k$ be locally high roots such that $\Delta(\varphi_1),\ldots,\Delta(\varphi_k)$ are the distinct connected components of $\Delta(e)$. By assumption, we have $k>1$ and thus $\Delta(\varphi_i)\neq\Delta$ for all $1\leq i\leq k$. Let $(\alpha_1,\ldots,\alpha_r)$ be a greedy decomposition of $e$. For each $1\leq i\leq k$ we define a degree $e_i\in H_2(G/B)$ by the equation
$$
e_i=\sum_{1\leq j\leq r\colon\Delta(\alpha_j)\subseteq\Delta(\varphi_i)}\alpha_j^\vee\,.
$$
By definition, we clearly have $\Delta(e_i)=\Delta(\varphi_i)$, i.e. $e_i$ is a connected degree for all $1\leq i\leq k$. Moreover, we have $\sum_{i=1}^k e_i=e$. By Proposition~\ref{prop:zd}(\ref{item:greedy}) and Proposition~\ref{prop:delta}(\ref{item:reduce}), it follows that $e_i\in\delta_B(z_{e_i}^B)$. Therefore Lemma~\ref{lem:equivalence-2} implies $e_i\leq d_{G/B}$ for all $1\leq i\leq k$. But $\Delta(e_1),\ldots,\Delta(e_k)$ are pairwise disjoint (even pairwise totally disjoint), so that this implies $e\leq d_{G/B}$ -- as claimed.
\end{proof}

\begin{lem}
\label{lem:equivalence}

Let $N=\mathrm{card}(\Delta)$. The statement $(\mathrm{H}_N)$ holds if and only if the statements $(\mathrm{H}_{N-1})$ and $(\mathrm{A}_N)$ hold.

\end{lem}

\begin{proof}

The implication from left to right is obvious from the definitions. Suppose that the statements $(\mathrm{H}_{N-1})$ and $(\mathrm{A}_N)$ hold. Let $\varphi$ be a positive root and let $e\in H_2(G(\varphi)/B(\varphi))$ be a degree such that $e\in\delta_{B(\varphi)}(z_e^{B(\varphi)})$. We have to show the inequality 
\begin{equation}
\label{eq:10}
e\leq d_{G(\varphi)/B(\varphi)}\,.
\end{equation} 
If $\mathrm{card}(\Delta(\varphi))<N$, then Inequality~(\ref{eq:10}) is implied by $(\mathrm{H}_{N-1})$. We may assume from now on that $\mathrm{card}(\Delta(\varphi))=N$. In other words, we have $G(\varphi)=G$, $B(\varphi)=B$ and $\Delta(\varphi)=\Delta$. If $e$ is a disconnected degree, then Inequality~(\ref{eq:10}) is implied by Lemma~\ref{lem:equivalence-1}. So, we may assume that $e$ is a connected degree. If $\Delta(e)\neq\Delta$, then Inequality~(\ref{eq:10}) is implied by Lemma~\ref{lem:equivalence-2}. So, we may assume that $e$ is a connected degree such that $\Delta(e)=\Delta$. In this final case, Inequality~(\ref{eq:10}) is implied by the statement $(\mathrm{A}_N)$.
\end{proof}

\begin{thm}
\label{thm:main}

For all $u,v\in W$ and all $d\in\delta_P(u,v)$ the inequality $d\leq d_X$ is satisfied.

\end{thm}

\begin{rem}

Note that it is obvious that for all $u,v\in W$ there exists a $d\in\delta_P(u,v)$ such that $d\leq d_X$ (cf. Proposition~\ref{prop:delta2}(\ref{item:usmallerv2}) applied to $u'=v'=w_o$ and Theorem~\ref{thm:descriptionofdelta}). The point of Theorem~\ref{thm:main} is that for all $u,v\in W$ and \emph{for all} $d\in\delta_P(u,v)$ the inequality $d\leq d_X$ holds.

\end{rem}

\begin{proof}

By Theorem~\ref{thm:descriptionofdelta} and Theorem~\ref{thm:delta2} we have we the equality
$$
\bigcup_{u,v\in W}\delta_P(u,v)=\left\{d\text{ a degree such that }d\in\delta_P(z_d^P)\right\}\,.
$$
Therefore Theorem~\ref{thm:main} is equivalent to the following theorem.
\end{proof}

\begin{thm}
\label{thm:main2}

Let $d$ be a degree such that $d\in\delta_P(z_d^P)$. Then the inequality $d\leq d_X$ is satisfied.

\end{thm}

\begin{rem}

Note that Theorem~\ref{thm:main2} (or equivalently Theorem~\ref{thm:main}) is obvious for maximal parabolic subgroups $P$. Moreover, it is obvious that for every degree $d$ such that $\delta_P(z_d^P)=d$ the inequality $d\leq d_X$ is satisfied (cf. Proposition~\ref{prop:delta}(\ref{item:usmallerv})). Also, it is obvious that no degree $d$ such that $d\in\delta_P(z_d^P)$ can satisfy the inequality $d>d_X$. The non-trivial part of Theorem~\ref{thm:main2} is about to show that a degree $d$ such that $d\in\delta_P(z_d^P)$ is comparable to $d_X$ even if $\delta_P(z_d^P)$ does not consist of a unique element.

\end{rem}

\begin{proof}

Let $d$ be as in the statement. By Theorem~\ref{thm:resandind} there exists a degree $e\in\delta_B(z_e^B)$ such that $e_P=d$. If $e\leq d_{G/B}$, then it follows by Corollary~\ref{cor:identity} that $d\leq d_X$. Therefore it suffices to prove the following theorem.
\end{proof}

\begin{thm}

Let $e\in H_2(G/B)$ be a degree such that $e\in\delta_B(z_e^B)$. Then the inequality $e\leq d_{G/B}$ is satisfied.

\end{thm}

\begin{proof}

Let $N=\mathrm{card}(\Delta)$. We show by induction that the statement $(\mathrm{H}_n)$ is true for all $1\leq n\leq N$. This clearly implies the theorem by applying $(\mathrm{H}_N)$ to the highest root $\varphi=\theta_1$. The induction base $(\mathrm{H}_1)$ is obviously satisfied (cf. Remark~\ref{rem:base}). Let $2\leq n\leq N$ and assume that $(\mathrm{H}_{n-1})$ is satisfied. Let $\varphi$ be a positive root such that $\mathrm{card}(\Delta(\varphi))\leq n$. We want to show:
\begin{equation}
\label{eq:hvarphi}
\text{For all }e\in H_2(G(\varphi)/B(\varphi))\text{ such that }e\in\delta_{B(\varphi)}\left(z_e^{B(\varphi)}\right)\text{ we have }e\leq d_{G(\varphi)/B(\varphi)}\,.
\end{equation}
If $\mathrm{card}(\Delta(\varphi))<n$, then the we clearly have $$
(\mathrm{H}_{n-1})\implies(\ref{eq:hvarphi})\,.
$$
We may assume from now on that $\mathrm{card}(\Delta(\varphi))=n$. Let us define the statements $(\mathrm{H}_n)_{R(\varphi)}$ and $(\mathrm{H}_{n-1})_{R(\varphi)}$ with respect to $R(\varphi)$ in the same way we defined $(\mathrm{H}_n)$ and $(\mathrm{H}_{n-1})$ with respect to $R$. Then we have two trivial implications
$$
(\mathrm{H}_{n-1})\implies(\mathrm{H}_{n-1})_{R(\varphi)}\text{ and }(\mathrm{H}_n)_{R(\varphi)}\implies(\ref{eq:hvarphi})
$$
These implications follow in view of the natural identifications: For all $\phi\in R(\varphi)^+$ we have $R(\varphi)(\phi)=R(\phi)$ and similarly for $G$, $B$ and $\Delta$. Therefore it suffices to prove $(\mathrm{H}_n)_{R(\varphi)}$ under the assumption of $(\mathrm{H}_{n-1})_{R(\varphi)}$. By replacing $R(\varphi)$ and $R$, we may assume that $n=N$. In other words, it suffices to prove $(\mathrm{H}_N)$ under the assumption of $(\mathrm{H}_{N-1})$. In view of Lemma~\ref{lem:equivalence} it therefore suffices to prove the following lemma.
\end{proof}

\todo[inline,color=green]{I make statements for $G/P$ and $G/B$ concerning $d\in\delta(z_d)$. Plus remark after statement for $G/P$ (both statements) -- obvious when\ldots}

\begin{lem}
\label{lem:heart}

Let $N=\mathrm{card}(\Delta)$. Suppose that $(\mathrm{H}_{N-1})$ is satisfied. Then the statement $(\mathrm{A}_N)$ is also satisfied.

\end{lem}

\begin{proof}\renewcommand{\qedsymbol}{}

Let $e\in H_2(G/B)$ be a degree such that $\Delta(e)=\Delta$ and such that $e\in\delta_B(z_e^B)$. Under the assumption of $(\mathrm{H}_{N-1})$, we have to show that $e\leq d_{G/B}$. Note that $e$ is necessarily a connected degree, since $\Delta$ is obviously a connected subset of the Dynkin diagram. Let $\alpha=\alpha(e)$ be the unique first entry of a greedy decomposition of $e$ (cf. Proposition~\ref{prop:connecteddegree}).We know that $\Delta(\alpha)=\Delta(e)=\Delta$. Let $\Delta_\alpha^\circ=\{\gamma\in\Delta\mid(\alpha,\gamma)=0\}$ for short. Let $\varphi_1,\ldots,\varphi_k$ be locally high roots such that
$$
\Delta(\varphi_1),\ldots,\Delta(\varphi_k)\text{ are the distinct connected components of}\begin{cases}
\Delta^\circ & \text{if }\Delta(e-\alpha^\vee)\subseteq\Delta^\circ\\
\Delta_\alpha^\circ &\text{otherwise.}
\end{cases}
$$
\end{proof}

\begin{proof}[First step]\renewcommand{\qedsymbol}{}

We first want to show the inequality
\begin{equation}
\label{eq:firststep}
e-\alpha^\vee\leq\sum_{i=1}^k d_{G(\varphi_i)/B(\varphi_i)}\,.
\end{equation}
To this end, let $(\alpha_1,\ldots,\alpha_r)$ be a greedy decomposition of $e$. As observed before, we necessarily have $\alpha=\alpha_1$. For each $1\leq i\leq k$, we can define a degree $e_i\in H_2(G/B)$ by the equation
$$
e_i=\sum_{2\leq j\leq r\colon\Delta(\alpha_j)\subseteq\Delta(\varphi_i)}\alpha_j^\vee\,.
$$
By Theorem~\ref{thm:orthogonality} applied to $P=B$, we know that $\Delta(e-\alpha^\vee)\subseteq\Delta_\alpha^\circ$ which means that $\Delta(e-\alpha^\vee)\subseteq\bigcup_{i=1}^k\Delta(\varphi_i)$. Thus, for every $2\leq j\leq r$ there exists a unique $1\leq i_j\leq k$ such that $\Delta(\alpha_j)\subseteq\Delta(\varphi_{i_j})$. This means in particular that we have $\sum_{i=1}^k e_i=e-\alpha^\vee$. In order to show Inequality~(\ref{eq:firststep}), it therefore suffices to show $e_i\leq d_{G(\varphi_i)/B(\varphi_i)}$ for all $1\leq i\leq k$. 

Fix some $1\leq i\leq k$. Let $\gamma$ be an entry in a greedy decomposition of $e_i$. By Proposition~\ref{prop:zd}(\ref{item:greedy}), it follows that $\gamma=\alpha_j$ for some $2\leq j\leq r$ such that $i_j=i$. Therefore, it follows that $\Delta(\gamma)\subseteq\Delta(\varphi_i)$ and thus $\gamma\leq\varphi_i$. Since $\gamma$ was an arbitrary entry in a greedy decomposition of $e_i$, it follows from Fact~\ref{fact:local} that $z_{e_i}^B=z_{e_i}^{B(\varphi_i)}$.

By Proposition~\ref{prop:zd}(\ref{item:greedy}) and Proposition~\ref{prop:delta}(\ref{item:reduce}), it follows that $e_i\in\delta_B(z_{e_i}^B)$. The previous paragraph and Theorem~\ref{thm:compatibility} then show that $e_i\in\delta_{B(\varphi)}\left(z_{e_i}^{B(\varphi_i)}\right)$. By definition, it is clear that $\Delta(\varphi_i)\neq\Delta$ (we even have $\Delta(\varphi_1)\cup\cdots\cup\Delta(\varphi_k)\neq\Delta$). Thus, the statement $(\mathrm{H}_{N-1})$ applies and implies that $e_i\leq d_{G(\varphi_i)/B(\varphi_i)}$. This completes the proof of Inequality~(\ref{eq:firststep}).
\end{proof}

\begin{proof}[Second step: the case $\Delta(e-\alpha^\vee)\subseteq\Delta^\circ$]\renewcommand{\qedsymbol}{}

In this step we assume that $\Delta(e-\alpha^\vee)\subseteq\Delta^\circ$ and prove that $e\leq d_{G/B}$. Indeed, by Remark~\ref{rem:disjoint}, Corollary~\ref{cor:identity} and by definition of $\varphi_1,\ldots,\varphi_k$ we clearly have
$$
d_{G/B}=\theta_1^\vee+\sum_{i=1}^k d_{G(\varphi_i)/B(\varphi_i)}\,.
$$
This and Inequality~\eqref{eq:firststep} yield the inequality $e+\theta_1^\vee-\alpha^\vee\leq d_{G/B}$. It is clear that $\alpha$ is $B$-cosmall (Proposition~\ref{prop:zd}(\ref{item:unique})) and that $\theta_1$ is $B$-cosmall (Example~\ref{ex:highestroot2}) and that $\alpha\leq\theta_1$. Thus, \cite[Lemma~4.7(a)]{curvenbhd2} clearly implies that $\alpha^\vee\leq\theta_1^\vee$. The desired inequality follow from this. This completes the proof of this case. 
\end{proof}

\begin{proof}[Third step: the case $\Delta(e-\alpha^\vee)\not\subseteq\Delta^\circ$, i.e. $\alpha$ is an exceptional root]

In this step we assume that $\Delta(e-\alpha^\vee)\not\subseteq\Delta^\circ$, i.e. there exists a simple root $\beta\in\Delta(e-\alpha^\vee)\setminus\Delta^\circ$. As it was said before, by Theorem~\ref{thm:orthogonality} applied to $P=B$, we have $\Delta(e-\alpha^\vee)\subseteq\Delta_\alpha^\circ$, in particular $(\alpha,\beta)=0$. By definition $\alpha$ is a (in fact the unique) maximal root of $e$, in particular $\alpha$ is a (in fact the unique) maximal root of $\alpha^\vee+\beta^\vee\leq e$. All in all, this means by definition that $\alpha$ is an exceptional root. In view of Lemma~\ref{lem:technical2}, Inequality~\eqref{eq:firststep} implies that $e\leq d_{G/B}$ -- as required.
\end{proof}

\begin{cor}

Let $d$ be a degree such that $d\in\delta_P(z_d^P)$. Then we have $d_\beta\in\delta_{P_\beta}(z_{d_\beta}^{P_\beta})$ for all $\beta\in\Delta\setminus\Delta_P$.

\end{cor}

\begin{proof}

Let $d$ be as in the statement. Let $\beta\in\Delta\setminus\Delta_P$. By Theorem~\ref{thm:equalwx} we clearly have
\begin{equation}
\label{eq:d_beta}
\bigcup_{u,v\in W}\delta_{P_\beta}(u,v)=\left\{e\in\mathbb{Z}\text{ a degree such that }e\in\delta_{P_\beta}(z_e^{P_\beta})\right\}=\left\{0,1,\ldots,d_{G/P_\beta}\right\}\,.
\end{equation}
By Theorem~\ref{thm:main} and Corollary~\ref{cor:identity} (or Theorem~\ref{thm:uniqueness}) we have $0\leq d_\beta\leq (d_X)_\beta=d_{G/P_\beta}$. Thus, Equation~\eqref{eq:d_beta} implies $d_\beta\in\delta_{P_\beta}(z_{d_\beta}^{P_\beta})$ -- as claimed.
\end{proof}

\todo[inline,color=green]{We actually may have a strict inclusion
$$
\{\text{minimal degrees in quantum products of two Schubert cycles}\}\subsetneq\{0\leq d\leq d_X\}\,.
$$
Compare the example in type $\mathsf{G}_2$ and $d=2\alpha_1^\vee+\alpha_2^\vee$ (notes on September 15). I will write $\bigcup_{u,v\in W}\delta_P(u,v)$.

I should introduce the notion of a disconnected degree as soon as I need it. Notion introduced September 25.}

\begin{ex}[{\cite[Example~5.9]{thesis}}]
\label{ex:5.9}

We give an example of a positive root $\alpha$ and a simple root $\beta$ such that $\delta_{P_\beta}(s_\alpha)<(\omega_\beta,\alpha^\vee)$, showing that Lemma~\ref{lem:5.14} is not necessarily true if $R$ is not simply laced. Indeed, by Theorem~\ref{thm:descriptionofdelta} and Theorem~\ref{thm:main} it suffices to find $\alpha\in R^+$ and $\beta\in\Delta$ such that $d_{G/P_\beta}<(\omega_\beta,\alpha^\vee)$. Let $R$ be of type $\mathsf{G}_2$. Let $\alpha_1$ be the short simple root and let $\alpha_2$ be the long simple root as in \cite[Plate~IX]{bourbaki_roots}. Let $\alpha=\theta_s=2\alpha_1+\alpha_2$ be the highest short root and let $\beta=\alpha_2$. Let $\theta_1=3\alpha_1+2\alpha_2$ be the highest root and let $\theta_2=\alpha_1$. Then we have $\mathcal{B}_{R}=\{\theta_1,\theta_2\}$ and thus, by Corollary~\ref{cor:identity}, that 
$$
d_{G/P_\beta}=(\omega_\beta,\theta_1^\vee)+(\omega_\beta,\theta_2^\vee)=2+0=2\,.
$$
On the other hand, it is clear that $(\omega_\beta,\alpha^\vee)=3$.



\end{ex}

\begin{ex}
\label{ex:5.9'}

We give an example of a positive root $\alpha$ such that $\alpha^\vee\notin\delta_B(s_\alpha)$, showing that Theorem~\ref{thm:simplylaced} is not necessarily true if $R$ is not simply laced even if $P=B$. Indeed, let $R$ be of type $\mathsf{G}_2$. Let $\alpha$ and $\beta$ be as in Example~\ref{ex:5.9}. Suppose for a contradiction that $\alpha^\vee\in\delta_B(s_\alpha)$. By Theorem~\ref{thm:descriptionofdelta} and Theorem~\ref{thm:main} we then have $\alpha^\vee\leq d_{G/B}$ and thus by Corollary~\ref{cor:identity} that $(\omega_\beta,\alpha^\vee)\leq (d_{G/B})_\beta=d_{G/P_\beta}$. But we already saw in Example~\ref{ex:5.9} that $d_{G/P_\beta}<(\omega_\beta,\alpha^\vee)$.

\end{ex}

\begin{ex}
\label{ex:inclusionstrict}

Theorem~\ref{thm:main} shows that we have an inclusion
$$
\bigcup_{u,v\in W}\delta_P(u,v)\subseteq\{0\leq d\leq d_X\}\,.
$$
In this example we show that this inclusion might be strict -- not an equality -- even in the case that $P=B$ by finding a degree $e\in H_2(G/B)$ such that $e\leq d_{G/B}$ but $e\notin\delta_B(u,v)$ for all $u,v\in W$. Indeed, let $R$ be of type $\mathsf{G}_2$. Let $\alpha_1,\alpha_2,\theta_1,\theta_2$ be as in Example~\ref{ex:5.9}. By Corollary~\ref{cor:identity} we have
$$
d_{G/B}=\theta_1^\vee+\theta_2^\vee=2\alpha_1^\vee+2\alpha_2^\vee\,.
$$
Let $e=2\alpha_1^\vee+\alpha_2^\vee$. We clearly have $0\leq e\leq d_{G/B}$. On the other hand, a simple computation shows that
$
(3\alpha_1+\alpha_2,\alpha_1)
$
is a (in this case the unique) greedy decomposition of $e$ and that its two different entries $3\alpha_1+\alpha_2$ and $\alpha_1$ are not orthogonal (we actually have $(3\alpha_1+\alpha_2,\alpha_1^\vee)=3$). Corollary~\ref{cor:stronglyorthogonal} implies that $e\notin\delta_B(z_e^B)$. Thus, Theorem~\ref{thm:delta2} shows that $e\notin\delta_B(u,v)$ for all $u,v\in W$.

\end{ex}

\bibliographystyle{amsplain}
\bibliography{diplom}


\end{document}